\documentclass[final]{siamltex1213}

\usepackage{floatrow}
\newfloatcommand{capbtabbox}{table}[][\FBwidth]

\usepackage{amsmath}
\usepackage{amssymb}
\usepackage{amsfonts}
\usepackage{mathrsfs}
\usepackage{graphicx}
\usepackage{algorithm}
\usepackage{algorithmicx}
\usepackage{algpseudocode}
\algdef{SE}[DOWHILE]{Do}{doWhile}{\algorithmicdo}[1]{\algorithmicwhile\ #1}
\usepackage{centernot}
\usepackage{cancel}
\usepackage{cite}
\usepackage[section]{placeins}

\usepackage{array}
\usepackage{booktabs}
\usepackage{subcaption}

\newcommand{\ExExpanding}{(a)}
\newcommand{\ExFootball}{(b)}
\newcommand{\ExTwoCircles}{(c)}
\newcommand{\ExOscillatingCircle}{(d)}
\newcommand{\ExEscapingCircle}{(e)}
\newcommand{\ExThreeLeavedRose}{(f)}

\newcommand{\RR}{\mathbb{R}}
\newcommand{\Aa}{\mathcal{A}}
\newcommand{\CC}{\mathcal{C}}
\newcommand{\MM}{\mathcal{M}}

\newcommand{\UU}{\mathcal{U}}
\newcommand{\VV}{\mathcal{V}}

\newcommand{\uvec}{\mathbf{u}}

\newcommand{\xvec}{\mathbf{x}}
\newcommand{\yvec}{\mathbf{y}}

\newcommand{\nav}{\bar{\nu}}
\newcommand{\ntwo}{\hat{\mathfrak{n}}}
\newcommand{\nthree}{\hat{\nu}}
\newcommand{\first}{1^{\mathrm{st}}}

\newcommand{\bull}{\textbullet~}
\newcommand{\resp}{resp. }

\newcommand{\ie}{i.e., }
\newcommand{\eg}{e.g., }
\newcommand{\vs}{vs. }
\newcommand{\Ex}{Ex. }

\newcommand{\Acc}{\mathcal{A}}
\newcommand{\NB}{\mathcal{N}}

\newcommand{\Book}{\textit{Book}}

\title{A fast-marching algorithm for non-monotonically evolving fronts \thanks{Submitted to the journal's Methods and Algorithms for Scientific Computing section, April 16th, 2015. Accepted for publication April 14th, 2016.}} 

\author{Alexandra Tcheng, Jean-Christophe Nave \thanks{Department of Mathematics and Statistics, McGill University, 805 Sherbrooke Street West, Montreal, Quebec H3A 0B9, Canada. Emails: \texttt{alexandra.tcheng@mail.mcgill.ca}, \texttt{jcnave@math.mcgill.ca}
.}}

\begin{document}
\maketitle

\slugger{sisc}{xxxx}{xx}{x}{x--x}

\begin{abstract} The non-monotonic propagation of fronts is considered. When the speed function $F:\RR^{n} \times [0,T]\rightarrow \RR$ is prescribed, the non-linear advection equation $\phi_{t}+F|\nabla \phi|=0$ is a Hamilton-Jacobi equation known as the level-set equation. It is argued that a small enough neighbourhood of the zero-level-set $\MM$ of the solution $\phi: \RR^{n} \times [0,T] \rightarrow \RR$ is the graph of $\psi:\RR^{n} \rightarrow \RR$ where $\psi$ solves a Dirichlet problem of the form $H(\uvec,\psi(\uvec),\nabla \psi(\uvec))=0$. A fast-marching algorithm is presented where each point is computed using a discretization of such a Dirichlet problem, with no restrictions on the sign of $F$. The output is a directed graph whose vertices evenly sample $\MM$. The convergence, consistency and stability of the scheme are addressed. 
Bounds on the computational complexity are estimated, and experimentally shown to be on par with the Fast Marching Method. 
Examples are presented where the algorithm is shown to be globally first-order accurate. 
The complexities and accuracies observed are independent of the monotonicity of the evolution.
\end{abstract}



\pagestyle{myheadings}
\thispagestyle{plain}
\markboth{A.TCHENG, J.-C.NAVE}{A fast-marching algorithm for non-monotone fronts}

\section{Introduction}
\label{sec:Introduction}

Front propagation is a time-dependent phenomenon occurring when the boundary between two distinct regions of space is evolving. 
It is possible to make the distinction between \emph{monotone} and \emph{non-monotone} motion of fronts. 
Consider a fire propagating through a forest: the interface divides space into a burnt and an unburnt region. It evolves monotonically in that, if a point $\xvec = (x_{1},\ldots,x_{n})$ of space belongs to the burnt region, then it cannot belong to the unburnt region at a later time \cite{OsherSethian}. In contrast, if a plate containing water is gently rocked, the front separating the dry region from the wet one may advance or recede. In this paper, we present an algorithm that dynamically builds a sampling of the subset of $\RR^{n} \times [0,T)$ consisting of the surface traced out by the front as it evolves through time. The structure of the scheme is akin to the Fast Marching Method, yet our approach is oblivious to the monotonicity of the evolution.

Given an initial front $\CC_{0}$ as a codimension-one subset of $\RR^{n}$ and an advection rule, our goal is to recover the front $\CC_{t}$ at later times $t>0$. Let each point of the front evolve with a prescribed speed $F:\RR^{n} \times [0,T] \rightarrow \RR$ in the direction of the outward normal to the front $\ntwo_{t} : \CC_{t} \rightarrow S^{n}$. The evolution is non-linear, such that even if $F$ and $\CC_{0}$ are smooth the front may become $C^{0}$ and undergo topological changes \cite{CrandallLions}.

On the one hand, a robust numerical method for tracking either kind of evolution is the Level-Set Method (LSM) \cite{OsherSethian,OsherFedkiw,sethian1999level}. This implicit approach embeds the front as the zero-level-set of a function $\phi:\RR^{n} \times [0,T) \rightarrow \RR$, and solves the initial-value problem:
\begin{eqnarray}
\label{eq:LSE}
\left\{ \begin{array}{rcll} 
\phi_{t}+F|\nabla\phi| &=& 0 & \quad \mathrm{in}~ \RR^{n} \times (0,T) \\
\phi(\xvec,0) &=& \phi_{0}(\xvec) & \quad \mathrm{on}~ \RR^{n} \times \{ 0 \}
\end{array} \right.
\end{eqnarray}
where $\phi_{0}$ satisfies $\{ \xvec : \phi_{0}(\xvec) =0 \} = \CC_{0}$. 
Assuming the computational grid comprises $N^{n}$ points, the total complexity of the first-order algorithm solving (\ref{eq:LSE}) is $\mathcal{O}(N^{n+1})$ when endowed with the usual CFL condition $\Delta t \propto 1/N$.
On the other hand, the Fast Marching Method (FMM) \cite{Sethian,SethianFMM,sethian1999level,SethianBookVariational,Tsitsi} may be used when $F=F(\xvec)\geq \delta >0$ and is therefore suited for monotone propagation. This approach reduces the dimensionality of the problem by building the \emph{first arrival time function} $\hat{\psi}: \RR^{n} \rightarrow \RR$, \ie $t=\hat{\psi}(\xvec)$ gives the unique time at which the front reaches $\xvec$. This function solves the boundary-value problem: 
\begin{eqnarray}
\label{eq:EikonalNoTime}
\left\{ \begin{array}{rcll} 
|\nabla\hat{\psi}| &=& \frac{1}{F} & \quad \mathrm{in}~ \mathcal{B}_{0} \\ 
\hat{\psi}(\xvec) &=& 0 & \quad \mathrm{on}~ \CC_{0}
\end{array} \right.
\end{eqnarray} 
where $\mathcal{B}_{0}$ is the unbounded component of $\RR^{n} \setminus \CC_{0}$ for $F>0$. 
The FMM solves
(\ref{eq:EikonalNoTime}) in $\mathcal{O}(N^{n}\log N^{n})$ time using a variant of Dijkstra's algorithm \cite{Dijkstra}. Schemes derived from the FMM include: an extension to the case where $F$ depends on time \cite{Vlad,VladThesis}; and the Generalized FMM \cite{GFMM}. Problem (\ref{eq:EikonalNoTime}) may also be solved using the Fast Sweeping Method \cite{FastSweeping}, which is a first-order accurate method running in $\mathcal{O}(N^{n})$ time. When $F\equiv1$, higher order methods exist, such as the one of Cheng and Tsai \cite{TsaiRedistancing} which exploits the relation between (\ref{eq:EikonalNoTime}) and the time-dependent Eikonal equation.

The GFMM of \cite{GFMM} allows $F$ to vanish and change sign albeit at the expense of the computational time. 
In \cite{Article1}, when $n=2$ the authors propose a scheme that can handle such speed functions \textsl{while} featuring a computational complexity comparable to that of the FMM. 
This approach considers the set $\MM := \{ (\xvec,t) : \xvec \in \CC_{t} \}$.
Each point $p \in \MM$ may be described as $p = (x,y,\hat{\psi}(x,y))$ and/or $p = (\tilde{\psi}(y,t),y,t)$ and/or $p = (x,\bar{\psi}(x,t),t)$. 
In regions where $|F|\geq \delta>0$, the algorithm builds $\hat{\psi}$ using the FMM, whereas near points where $F=0$, it solves a PDE satisfied by either $\tilde{\psi}$ or $\bar{\psi}$. 
Despite a number of satisfactory results, the mechanism to change representation is cumbersome, and $\MM$ is undersampled near regions where $F$ vanishes.

In the present paper, rather than limiting ourselves to the $n+1$ representations just mentioned, we now describe $\MM$ locally with functions of the form $\psi(\uvec)$ 
where the points $\uvec=(u_{1},\ldots , u_{n})$ belong to \emph{some hyperplane} lying in $\RR^{n} \times [0,T]$. 
Consider the orthonormal spanning set for this hyperplane consisting of $\{ \hat{u}_{1}, \ldots, \hat{u}_{n}\}$, and let the normal be $\hat{u}_{n+1}$. 
So that in this convention, any point $(\uvec,u_{n+1})$ of space-time $\RR^{n}\times [0,T]$ may be written as $(\uvec,u_{n+1})= u_{1}\hat{u}_{1} + \ldots + u_{n} \hat{u}_{n}+ u_{n+1} \hat{u}_{n+1}$.
Section \ref{sec:Hyperplane} shows that $\psi:\RR^{n} \rightarrow \RR$ solves a Dirichlet problem of the form:
\begin{eqnarray}
\label{DirichletProblem}
\left\{ \begin{array}{rcll}
H(\uvec,\psi(\uvec),\nabla \psi(\uvec)) &=& 0 & \mathrm{in}~ \UU \subset \RR^{n} \\ 
\psi(\uvec(\xvec,t)) &=& u_{n+1}(\xvec,t) & \mathrm{on}~ \xvec \in \CC_{t} \cap \mathcal{V} 
\end{array}
\right.
\end{eqnarray} 
for appropriate neighbourhoods $\UU$ and $\VV$. The \emph{Hamiltonian} function $H:\RR^{n} \times \RR \times \RR^{n} \rightarrow \RR$ depends on $\psi$ through the speed function $F$ and involves constants that capture the relative orientations of the two coordinate systems in use. 
Section \ref{sec:Discretization} presents the discretization of (\ref{DirichletProblem}) using finite-differences, as well as the design of a constrained minimization problem.
Section \ref{sec:Algorithms} discusses the different parts of the algorithm for the case $n=2$. The protocol we propose is similar to the FMM: Points sampling $\MM$ first belong to the \emph{narrow band} $\NB$ before moving to the \emph{accepted} set $\Aa$. 
When $p_{a} \in \NB$ is accepted, a local $\{ \hat{u}_{1}, \ldots, \hat{u}_{n+1}\} $-coordinate system is found. 
Using $p_{a}$ and another point in $\Aa$, a new point is computed through the finite-difference solvers and the optimization problem.  
Section \ref{sec:Properties} highlights the properties of the solvers and the global scheme. In particular, convergence of the local solvers is addressed and the total computational complexity is estimated. 
As is illustrated in \S \ref{sec:Examples} with numerous examples, the output of the algorithm is a \emph{directed graph}. Its meshsize is bounded below by a predetermined parameter $h$, and its vertices provide a discrete sampling of $\MM$. Given $t\in (0,T)$, the front $\CC_{t}$ can be recovered using interpolation.

Methods with a similar flavour have been explored in \cite{Runborg2014} which presents a high order fast interface tracking method, and \cite{guckenheimer2004fast} where the authors formulate $n$ `quasi-linear PDEs satisfied by the manifold's local parametrization' and `solve that system locally in an Eulerian framework'.

For clarity of illustration, the setting of our paper is $n=2$, so that $\xvec = (x_{1},x_{2}) = (x,y)$, $\uvec = (u_{1},u_{2}) = (u,v)$ and $u_{3}=w$. It is worth noting however that most of the underlying ideas, which are presented in \S \ref{sec:Hyperplane} and \S \ref{sec:Discretization}, easily extend to the general higher dimensional case. Current obstacles to the full generalisation of the method lie in the design of a practical interface tracker, able to capture the global features of the front at all times. The successful results of Sections \ref{sec:Algorithms}, \ref{sec:Properties} and \ref{sec:Examples} when $n=2$ provide a proof of concept that such a goal should be pursued in future work. 

Our approach offers numerous advantages when $n=2$. It extends previous work: If the local and global systems coincide then $\psi = \hat{\psi}$, whereas if $(\hat{u}_{1},\hat{u}_{2}) = (\hat{y},\hat{t})$ and $\hat{u}_{3}=\hat{x}$ then $\psi = \tilde{\psi}$. By construction, the transition from one representation to another is smooth, and the resolution of the sampling is regular. The initialization is almost identical to the main procedure, and does not require any information away from $\CC_{0}$. The accuracy of the scheme is $\mathcal{O}(h)$. If $\CC_{0}$ is sampled by $m$ points, the computational complexity is bounded by $CN$ where $C = \max \{ \mathcal{O}(m) , \mathcal{O}\left(10^{n+2}\right) \}$ with $N = \mathcal{O}(m^{n})$. When run on a monotone example, for the same computational time, our method yields a more accurate solution than the FMM.

\section{Hyperplane representation}
\label{sec:Hyperplane}

In this section, we derive Problem (\ref{DirichletProblem}) from the IVP (\ref{eq:LSE}) after making a few assumptions.
We then argue that the solution of (\ref{DirichletProblem}) locally describes $\MM$, and that a finite number of such solutions provides a covering of $\MM$. 

\subsection{Assumptions}
\label{sec:Assumptions}

The set $\mathcal{C}_{0}$ is a closed, co-dimension-one subset of $\Omega \subset \RR^{2}$ without boundary, where $\Omega$ is bounded. Moreover, $\CC_{0}$ is the graph of a $C^{1}$ function. The speed function $F: \mathbb{R}^{2} \times [0,T] \rightarrow \mathbb{R}$ is continuous. It may vanish and change sign. 
Under those assumptions, the LSE in (\ref{eq:LSE}) is a Hamilton-Jacobi equation with a unique continuous viscosity solution \cite{CrandallLions,UserGuideViscosity,crandall1984some}. It follows that $\MM$ embeds in $\RR^{2}\times [0,T]$ as a $C^{0}$ manifold. We denote as $\hat{\nu}(\xvec,t)$ the outward normal to $\MM$ at $(\xvec,t)$. 

\subsection{PDE on an arbitrary plane}
\label{subsec:PDE}

\begin{figure}
\begin{center}
\begin{subfigure}{.58\textwidth}
  \centering
  \includegraphics[width=\linewidth]{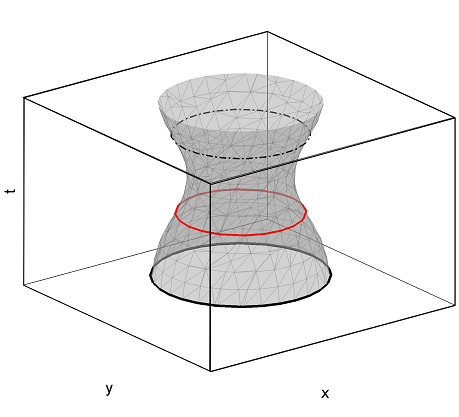}
  \caption{$\MM \subset \RR^{2}\times [0,T]$. The $xyt$-coordinate system is the global one.}
\end{subfigure}
\hspace{.2cm}
\begin{subfigure}{.38\textwidth}
  \centering
  \includegraphics[width=\linewidth]{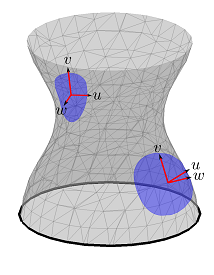}
  \caption{Two neighbourhoods and their $uvw$-coordinate systems.}
\end{subfigure}
\caption{Illustration of the notation for $n=2$.}
\label{fig:Hyperplane}
\end{center}
\end{figure}

Let a normal vector $\bar{\nu} \in S^{3}$ be given, and consider a corresponding plane lying in $xyt$-space. We denote as $\hat{u}\hat{v}\hat{w}$ the unique right-handed orthonormal coordinate system satisfying the following conditions: $\hat{w} = \bar{\nu}$, the $\hat{v}\hat{w}$-plane is vertical and the $\hat{t}$-component of $\hat{v}$ is positive. See Figure \ref{fig:Hyperplane}. We have: 
\begin{eqnarray}
\label{eq:Dictionary}
\left( \begin{array}{c} \hat{u} \\ \hat{v} \\ \hat{w} \end{array} \right)
= 
\left( \begin{array}{ccc} 
\alpha_{1} & \alpha_{2} & \alpha_{3} \\ 
 \beta_{1} & \beta_{2} & \beta_{3} \\ 
\gamma_{1} & \gamma_{2} & \gamma_{3}  \end{array} \right)
\left( \begin{array}{c} \hat{x} \\ \hat{y} \\ \hat{t} \end{array} \right) 
: = M
\left( \begin{array}{c} \hat{x} \\ \hat{y} \\ \hat{t} \end{array} \right) 
\end{eqnarray}
where the $\alpha$'s, $\beta$'s and $\gamma$'s are constants. In particular, $\alpha_{3} = 0$ and $\beta_{3}>0$ (see Appendix \ref{app:A} for details).
Let us now assume that a neighbourhood $\mathcal{V}$ of $(\xvec_{0},t_{0}) \in \MM$ may be represented as the graph of the $C^{1}$ function: 
\begin{eqnarray}
\psi: \RR^{2} \rightarrow \RR \qquad \qquad \psi: \uvec \mapsto \psi(\uvec)
\end{eqnarray}
If $\bar{\nu}\cdot \nthree(\xvec_{0},t_{0}) > 0$, then by the Implicit Function Theorem $\phi(\xvec,t) = w-\psi(\uvec)$ where $\phi$ is the solution of (\ref{eq:LSE}). Using implicit differentiation, we replace the derivatives appearing in (\ref{eq:LSE}) by
\begin{eqnarray}
\phi_{x_{i}} = \gamma_{i} - \alpha_{i}\psi_{u} - \beta_{i}\psi_{v} \quad i=1,2 \quad \mathrm{and} \quad
\phi_{t} = \gamma_{3} - \alpha_{3}\psi_{u} - \beta_{3}\psi_{v}
\end{eqnarray}
Using the orthonormality of the $\hat{u}\hat{v}\hat{w}$-coordinate system, (\ref{eq:LSE}) can be rewritten solely in terms of $\alpha_{3}$, $\beta_{3}$ and $\gamma_{3}$ as
\begin{eqnarray}
\label{eq:AnyPlane}
\left[ \gamma - \left( \alpha \psi_{u} + \beta \psi_{v} \right) \right] + G \sqrt{ 1 + \psi^{2}_{u} + \psi^{2}_{v} - \left[ \gamma - \left( \alpha \psi_{u} + \beta \psi_{v} \right) \right]^{2} } &=& 0
\end{eqnarray}
where the subscript $\cdot_{3}$ was dropped, and $G(\uvec,\psi(\uvec)) := F \left( \xvec(\uvec,\psi(\uvec)) , t(\uvec,\psi(\uvec)) \right)$ was introduced.
Equivalently, Equation (\ref{eq:AnyPlane}) can be rearranged using the vectors $\hat{R} := (\alpha,\beta,\gamma) $ and $\nu := (-\psi_{u},-\psi_{v},1)$, as: 
\begin{eqnarray}
\label{eq:PDEvector}
\nu \cdot \hat{R} +G \sqrt{\nu \cdot \nu - \left( \nu \cdot \hat{R} \right)^{2}} = 0
\end{eqnarray}
Note that this formulation is independent of dimension. In the event where $\bar{\nu}=\hat{\nu}(\xvec_{0},t_{0})$, the point $(\xvec_{0},t_{0})$ is a local maximum of $\mathcal{V}$, and we may explicitly relate the orientation coefficients to the speed function: 
\begin{eqnarray}
\label{eq:SpecialR}
\hat{R} = \left( ~0 ~,~ \frac{1}{\sqrt{1+F^{2}(\xvec_{0},t_{0})}} ~,~ \frac{-F(\xvec_{0},t_{0}) }{\sqrt{1+F^{2}(\xvec_{0},t_{0})}} \right)
\end{eqnarray}
We arrived at the following conclusion: Suppose $\psi$ satisfies the boundary condition:
\begin{eqnarray}
\psi(\uvec(\xvec,t^{0})) = w(\xvec,t^{0})  \qquad \mathrm{when}~ \xvec \in \CC_{t^{0}} \cap \mathcal{V} 
\end{eqnarray}
and solves (\ref{eq:PDEvector}) on $\UU \subset \RR^{2}$. Then letting $t^{\ast} := t(\uvec,\psi(\uvec))$ for $\uvec \in \UU$, we have that $ \xvec^{\ast} : = \xvec(\uvec,\psi(\uvec)) \in \CC_{t^{\ast}}$.


\subsection{Well-posedness} 
\label{subsec:Well}

Equation (\ref{eq:AnyPlane}) is a Hamilton-Jacobi equation of the form:
\begin{eqnarray}
\label{eq:DefinitionH}
H(\uvec,\psi,\nabla \psi;\hat{R}) 
\, := \, - \left[ \gamma - \beta \psi_{v} \right] - G(\uvec,\psi) \sqrt{ \psi^{2}_{u} + \left( \beta + \gamma \psi_{v} \right)^{2}} 
\, = \, 0
\end{eqnarray}
where we defined the \emph{Hamiltonian} $H:\RR^{2} \times \RR \times \RR^{2} \rightarrow \RR$. The well-posedness of this equation is guaranteed in the class of viscosity solutions \cite{achdou2013hamilton}. The following theorem justifies defining $H$ as the negative of the left-hand side of (\ref{eq:AnyPlane}). 

\begin{theorem}
\label{thm:ViscousSolution} The viscosity solutions of Equations (\ref{eq:LSE}) and (\ref{eq:DefinitionH}) are equivalent in the sense that the zero level-set of $\phi$ at various times $t = t(\uvec,w)$ is precisely the set of points $\xvec = \xvec(\uvec,w)$ for which $\psi(\uvec) =w$.
\end{theorem}

Note that this equivalence is investigated in the general case in Osher's paper \cite{OsherDirichlet}.

\begin{proof} If $\MM$ is $C^{1}$ at a point $\bar{p}$, then the manipulations of \S \ref{subsec:PDE} hold in the strong sense. Thus, let $\MM$ be singular at the point $\bar{p}$, and consider the second-order PDE: 
\begin{eqnarray}
\label{eq:LSEvanishing}
\phi^{\epsilon}_{t}+F|\nabla \phi^{\epsilon}| = \epsilon \Delta \phi^{\epsilon} \qquad \qquad \epsilon>0
\end{eqnarray}
\ie A viscous term was added to the right-hand side of the level-set equation. Define $\psi^{\epsilon} := w - \phi^{\epsilon}(u,v)$. Following the same reasoning as in \S \ref{subsec:PDE}, using the orthonormal relations, as well as the fact that $\alpha_{3}=0$, Equation (\ref{eq:LSEvanishing}) becomes:
\begin{eqnarray}
\left[ \gamma - \beta \psi_{v} \right] + G \sqrt{ \psi^{2}_{u} + \left( \beta + \gamma \psi_{v} \right)^{2}} &=& - \epsilon \left[ 
\psi^{\epsilon}_{uu} 
+ \left( \beta^{2}_{1} + \beta^{2}_{2} \right) \psi^{\epsilon}_{vv} \right]
\end{eqnarray}
Taking $\epsilon$ to zero, the arguments detailed for example in \S 1.5 of Lions' book \cite{LionsBook} apply to yield that the Hamilton-Jacobi equation of interest is (\ref{eq:DefinitionH}).
\end{proof}

In the light of this result, Equation (2.8) is preferred over (2.4) in the rest of this paper. 

\subsection{Geometric properties} We illustrate how $\MM$ can be globally described by a covering of $\psi$-functions. \\

\paragraph{$\psi$ provides a local representation of $\MM$}
Suppose that the solution of (\ref{DirichletProblem}) is such that $\VV := \UU \times \psi(\UU)$ is $C^{0}$. Let $t_{\min} = \inf \{ t : (\xvec,t)\in\VV \}$ and $t_{\max} = \sup \{ t : (\xvec,t)\in\VV \}$. 

\begin{theorem}
\label{thm:Equivalence}
For any given $t^{\ast} \in (t_{\min},t_{\max})$, we have: 
\begin{eqnarray*}
\{ \xvec \in \RR^{2} \,:\, \psi(\uvec(\xvec,t^{\ast})) = w(\xvec,t^{\ast}) \} \, = \, \CC_{t^{\ast}} \cap \VV
\end{eqnarray*}
\end{theorem} The proof follows the same argument as the proof of Theorem 4.2 in \cite{Article1}, and we therefore omit it. \\

\paragraph{Covering of $\MM$} 
It follows from compactness that $\MM$ can be covered by a finite number of images of functions $\psi$, each solving a problem of the form (\ref{DirichletProblem}). The construction presented below builds one such covering that is particularly suited for algorithmic purposes. 
Given $\Delta t>0$ and $0 \leq t_{0} \leq T-\Delta t$, consider the strip 
\begin{eqnarray}
\label{StripDefinition}
\mathrm{St} : = \{ (\xvec,t) \in \MM \, : \, t_{0} \leq t \leq t_{0}+\Delta t \}
\end{eqnarray}

\begin{theorem}
\label{thm:Covering} Suppose that $\CC_{t_{0}}$ has codimension one in $\RR^{n}\times[0,T)$. 
There exists a finite covering of $\mathrm{St}$ consisting of the graphs of functions $\psi_{i}$, where each $\psi_{i}$ solves a Dirichlet problem of the form (\ref{DirichletProblem}) with boundary conditions imposed on $\CC_{t_{0}}$.
\end{theorem}

\begin{proof}
Consider an open covering of $\CC_{t_{0}}$ by open segments $s$ with the following property: Defining
\begin{eqnarray}
\nu_{0}^{\mathrm{av}} := \frac{1}{|s|} \int \nthree \, ds 
\quad \mathrm{where} \quad |s| = \int ds 
\end{eqnarray}
for each $p\in s$ we have: $\nthree(p) \cdot \nu_{0}^{\mathrm{av}} > 0$. 
Since $\CC_{t_{0}}$ is compact, we may extract a finite open covering, say $\mathcal{S}_{0}= \{ s_{i} : i = 1,2,\ldots,I \}$. Fix $i$, and define \emph{the future domain of influence of $s_{i}$} as the subset of $\mathrm{St}$ satisfying: 
\begin{eqnarray*}
\mathcal{W}_{i} := &\{& 
(\xvec, t) 
\, : \, 
 \exists \, \{ (\xvec_{n},t_{n}) \}^{\infty}_{n=1} \rightarrow (\xvec, t) \, \mathrm{such~that~for~each~} \\ 
&\,& (\xvec_{n},t_{n}) \, \exists \mathrm{~a~characteristic~of~the~LSE~starting~in~} s_{i} \mathrm{~and~ending~at~} (\xvec_{n},t_{n}) \}  
\end{eqnarray*}
(See for example \cite{evans2010partial} for details on the method of characteristics.)
See Figure \ref{fig:Strips} for an illustration. We have $\mathrm{St} = \cup_{i=1}^{I} \mathcal{W}_{i}$. Define: 
\begin{eqnarray}
\nu^{\mathrm{av}}_{\Delta t} := \frac{1}{|\mathcal{W}_{i}|} \int \nthree \, d\mathcal{W}_{i} 
\quad \mathrm{where} \quad |\mathcal{W}_{i}| = \int d\mathcal{W}_{i}
\end{eqnarray}
Note that $\lim_{\Delta t \rightarrow 0} \nu^{\mathrm{av}}_{\Delta t} = \nu^{\mathrm{av}}_{0}$. Picking $\Delta t > 0$ small enough, it follows from the continuity of $\MM$ that for each $p\in\mathcal{W}_{i}$, we have: $\nthree(p)\cdot \nu^{\mathrm{av}}_{\Delta t} >0$. Consider the Dirichlet problem $H(\uvec, \psi, \nabla \psi; \hat{R}(\nu^{\mathrm{av}}))=0$ with boundary conditions imposed at $s_{i} \subset \CC_{t_{0}}$. From Theorem \ref{thm:Equivalence}, $\mathcal{W}_{i}$ is contained in the graph of $\psi$. 
\end{proof}

A global covering may then be obtained from decomposing $\MM$ into strips of the form (\ref{StripDefinition}) and extracting a finite subcover. 

\begin{figure}
\begin{center}
\begin{subfigure}{.48\textwidth}
  \centering
  \includegraphics[width=\linewidth]{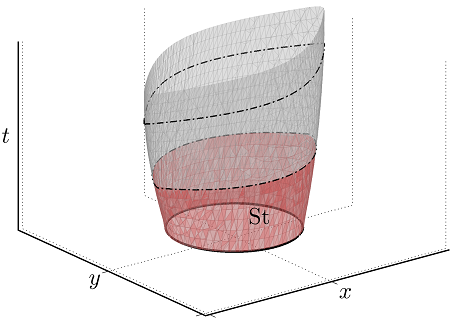}
\end{subfigure}
\hspace{.2cm}
\begin{subfigure}{.48\textwidth}
  \centering
  \includegraphics[width=\linewidth]{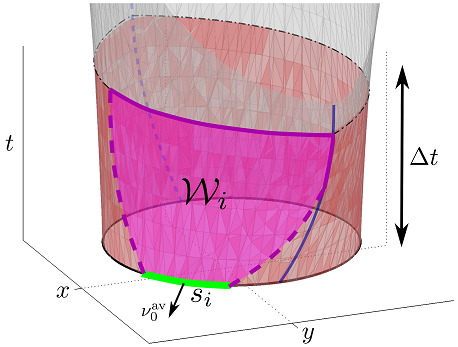}
\end{subfigure}
  \caption{Illustration of the notation in Theorem \ref{thm:Covering}. $\MM$ is $C^{0}$ along the blue line.}
\label{fig:Strips}
\end{center}
\end{figure}

\section{Discretization} 
\label{sec:Discretization}

Throughout this section, assume that the $uvw$-coordinate system is \emph{fixed}. Suppose we are given two points belonging to $\MM$ of the form $p_{a} = (\uvec_{a},w_{a})$ and $p_{b} = (\uvec_{b},w_{b})$; and wish to compute $w_{d} = \psi(\uvec_{d})$ at a location $\uvec_{d} = (u_{d},v_{d})$. We will refer to $p_{a}$ and $p_{b}$ as the \emph{parents} of the \emph{child} point $p_{d} = (\uvec_{d},w_{d})$. In \S \ref{subsec:SolvingPDE}, we first assume that $\uvec_{d}$ is \emph{given} and propose two different solvers that take as input $p_{a}$, $p_{b}$ and $\uvec_{d}$, and return a value $w_{d}$. Then in \S \ref{subsec:OptimalSampling}, we turn to the question of how $\uvec_{d}$ should be determined. The answer takes the form of a \emph{constrained minimization problem}, for which we propose two different methods. 

Section \ref{sec:Discretization} thus provides a framework for solving the problem locally. Global issues also need to be addressed -- See \S\ref{subsec:GlobalConstraints}.

\subsection{Solving the PDE}
\label{subsec:SolvingPDE}

As illustrated in Figure \ref{fig:DefineSandT} define: 
\begin{eqnarray}
\vec{s}_{a} := \uvec_{d} - \uvec_{a} 
\qquad
\vec{s}_{b} := \uvec_{d} - \uvec_{b} 
\qquad
\vec{s}_{c} := \uvec_{d} - \uvec_{c} 
\end{eqnarray}
and let $\hat{s}_{i} = ((s_{i})_{u},(s_{i})_{v})$ where $i=a,b,c$ be the corresponding \emph{unit} vectors. Then 
\begin{eqnarray}
\label{eq:ApproxDerivatives}
\psi_{i} = \frac{w_{d}-w_{i}}{|\vec{s}_{i}|}
\qquad \mathrm{where} ~ i=a,b,c
\end{eqnarray}
provide $\first$ order approximations of the directional derivatives of $\psi$ at $\uvec_{d}$. 
Making use of the calculus identity: $\psi_{k} = \nabla \psi \cdot \hat{k}$ for any unit vector $\hat{k}$, we have:
\begin{eqnarray}
\label{Rotation}
\left( \begin{array}{c}
\psi_{u} \\
\psi_{v} 
\end{array} \right)
= 
\left( \begin{array}{cc}
(s_{a})_{u} & (s_{a})_{v} \\
(s_{i})_{u} & (s_{i})_{v} 
\end{array} \right)^{-1}
\left( \begin{array}{c}
\psi_{a} \\
\psi_{i} 
\end{array} \right)
=: B^{-1}
\left( \begin{array}{c}
\psi_{a} \\
\psi_{i} 
\end{array} \right)
~ \mathrm{where} ~ i=b,c
\end{eqnarray}
provided that $\hat{s}_{a}$ and $\hat{s}_{i}$ are not colinear. We write
\begin{eqnarray}
\label{eq:ApproxDeriv}
\tilde{\nu} = -\vec{M} \, w_{d}-\vec{N} 
\quad \mathrm{with} \quad 
\vec{M} := (m_{u},m_{v},0) 
\quad \mathrm{and} \quad 
\vec{N} := (n_{u},n_{v},-1)
\end{eqnarray}
where the constants $m_{u}$, $m_{v}$, $n_{u}$ and $n_{v}$ depend on $(\uvec_{a},w_{a})$, $(\uvec_{i},w_{i})$ where $i=b$ or $c$, and $\uvec_{d}$. The quantity $\tilde{\nu}$ provides a $\first$ order approximation of $\nu$. 

\begin{figure}
		\centering
			\includegraphics[width=0.8\textwidth,natheight=493,natwidth=334]{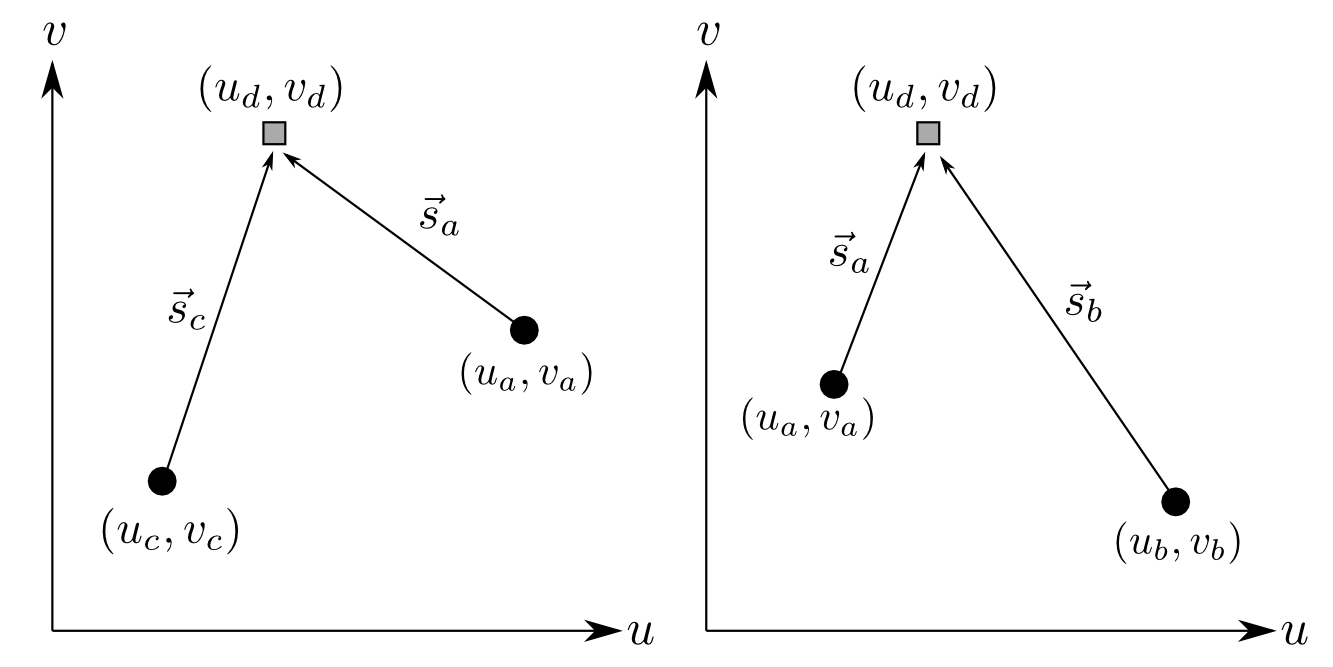}
		\caption{The points $p_{a}=(u_{a},v_{a},w_{a})$, $p_{b}=(u_{b},v_{b},w_{b})$ and $p_{c}=(u_{c},v_{c},w_{c})$ belong to $\Aa$.  Given $\uvec_{d}=(u_{d},v_{d})$, we wish to compute $w_{d}$.}
	\label{fig:DefineSandT}
	\end{figure}

\subsubsection{Direct solver}  
\label{subsubsec:DirectSolver}

Consider approximating $G$ by setting it equal to $G_{0} = G(p_{a})$. 
Note that this is independent of $w_{d}$, which implies that the relation 
\begin{eqnarray}
\label{eq:RelationAbove}
\tilde{\nu} \cdot \hat{R} + G_{0} \sqrt{\tilde{\nu} \cdot \tilde{\nu} - \left( \tilde{\nu} \cdot \hat{R} \right)^{2}} = 0
\end{eqnarray}
can be rearranged as a \emph{quadratic} in $w_{d}$. Indeed, first rewriting (\ref{eq:RelationAbove}) as
\begin{eqnarray}
\left( \tilde{\nu}\cdot \hat{R} \right)^{2} \left( 1+G_{0}^{2} \right) = G_{0}^{2} \left( \tilde{\nu}\cdot \tilde{\nu} \right)
\end{eqnarray}
and then making use of (\ref{eq:ApproxDeriv}), we arrive at 
\begin{eqnarray}
\left( \tilde{\nu}\cdot \hat{R} \right)^{2}  
= k_{1} w_{d}^{2} + 2 k_{2} w_{d} + k_{3} 
\qquad \mathrm{and} \qquad 
\tilde{\nu}\cdot \tilde{\nu}
= k_{4} w_{d}^{2} + 2 k_{5} w_{d} + k_{6}
\end{eqnarray}
where
\begin{eqnarray}
\label{eq:Coeff}
\begin{array}{lll}
k_{1} = \left( \hat{R} \cdot \vec{M} \right)^{2} &
k_{2} = \left( \hat{R} \cdot \vec{M} \right) \left( \hat{R} \cdot \vec{N} \right) &
k_{3} = \left( \hat{R} \cdot \vec{N} \right)^2 \\
k_{4} = \vec{M}\cdot \vec{M} &
k_{5} = \vec{M}\cdot \vec{N} &
k_{6} = \vec{N}\cdot \vec{N}
\end{array}
\end{eqnarray}
Rearranging further yields 
\[
\left( k_{1} + G_{0}^{2} (k_{1} - k_{4}) \right) w_{d}^{2} + 2 \left( k_{2} + G_{0}^{2} (k_{2} - k_{5}) \right) w_{d} + \left( k_{3} + G_{0}^{2} (k_{3} - k_{6}) \right)  = 0 \]
The discriminant of this quadratic reads:
\begin{eqnarray}
0+\rho_{1}G_{0}^{2} + \rho_{2}G_{0}^{4} = G_{0}^{2} \left( \rho_{1}+\rho_{2}G_{0}^{2} \right)
\end{eqnarray}
where $\rho_{1} = - 2k_{2}k_{5}  + k_{1}k_{6} + k_{3}k_{4}$ and $\rho_{2} = \rho_{1} +  k^{2}_{5} - k_{4}k_{6}$. We set  
\begin{eqnarray}
\label{SolutionQuadratic}
w_{d}
&=& \frac{ -\left( k_{2} + G^{2}_{0} (k_{2} - k_{5}) \right) + G_{0} \sqrt{ \rho_{1}+\rho_{2}G_{0}^{2} }}{\left( k_{1} + G^{2}_{0} (k_{1} - k_{4}) \right)} 
\end{eqnarray}
Since $\mathrm{sign}(\gamma) = -\mathrm{sign}(G_{0})$, this choice of root minimizes $t_{d} = \alpha u_{d} + \beta v_{d} + \gamma w_{d}$ where $\alpha = 0$ and $\beta>0$. This is consistent with the control theoretical approach to this problem \cite{EvansNotes,FalconeMin,Tsitsi,Vlad}.
If $0 = k_{1}+G_{0}^{2}(k_{1}-k_{4})$, the quadratic relation degenerates to a linear one with solution: 
\begin{eqnarray}
\label{SolutionLinear}
w_{d} = -\frac{k_{3} + G_{0}^{2} (k_{3} - k_{6})}{2 \left( k_{2} + G_{0}^{2} (k_{2} - k_{5}) \right)}
\end{eqnarray}
The outward normal to $\MM$ at $p_{d}$ can then be computed as $\nthree_{d} = \tilde{\nu}/|\tilde{\nu}|$. 

\textit{Remark:} As was already mentioned in \S \ref{subsec:Well}, when $\hat{R} = (0,0,\pm1)$, the PDE reduces to the Eikonal equation. If working on a Cartesian grid, formulas (\ref{SolutionQuadratic}) and (\ref{SolutionLinear}) can be verified to agree with the solvers used in the traditional FMM.

\subsubsection{Iterative solver} 
\label{subsubsec:IterativeSolver}

An alternative to approximating the speed term $G$ as in \S \ref{subsubsec:DirectSolver} is to solve (\ref{eq:PDEvector}) iteratively in pseudo-time $\tau$. \ie Let
\begin{eqnarray}
\tilde{\nu}^{n} := - \vec{M} \, w^{n}_{d} - \vec{N}
\qquad \mathrm{and} \qquad
G^{n}:= F( \, \xvec (\uvec_{d},w^{n}_{d}), t(\uvec_{d},w^{n}_{d}) \, )
\end{eqnarray}
and iterate:
\begin{eqnarray}
\label{eq:IterativeSolver}
\tilde{\nu}^{n} \cdot \hat{R} + G^{n} \sqrt{\tilde{\nu}^{n} \cdot \tilde{\nu}^{n} - \left( \tilde{\nu}^{n} \cdot \hat{R} \right)^{2}} = \frac{w^{n+1}_{d}-w^{n}_{d}}{\Delta \tau}
\end{eqnarray}
until $|w^{n+1}_{d}-w^{n}_{d}|$ is below some pre-determined tolerance.
The direct solver can be used to initialize $w^{0}_{d}$. The size of the pseudo-time step is determined based on Definition 5 of \cite{Oberman} -- see Appendix \ref{app:IterativeSolver} for details.

\subsection{Optimal sampling} 
\label{subsec:OptimalSampling}

Suppose that a set $\Aa$ of points sampling a part of $\MM$ is available. Given $p_{a}$, $p_{i} \in \Aa$, we now consider \emph{where} to compute a new point $p_{d}$. Constraints are heuristically imposed so as to efficiently build a regular sampling of $\MM$. Rigorous justifications of these choices are provided in \S \ref{sec:Properties} where the properties of the scheme are analyzed. 

\subsubsection{Constrained minimization problem} 

\paragraph{Accuracy} According to (\ref{eq:ApproxDerivatives}), the accuracy of either solver increases as $|\vec{s}_{a}|$ and $|\vec{s}_{i}|$ become smaller. 
We combine those requirements into a single function to be minimized:
\begin{eqnarray*}
f(\uvec_{d}) := |\vec{s}_{a}|^{2}+ |\vec{s}_{i}|^{2} 
\end{eqnarray*} 
The result is that the child point $p_{d}$ should lie close to its parents. However, we now argue that it should not be \emph{too} close. 

\paragraph{Evenness of the sampling} 
Repulsion between $p_{d}$ and its parents is introduced to avoid oversampling parts of $\MM$. 
The minimum three-dimensional (or $(n+1)$-dimensional) distance between $p_{d}$ and any other point in $\Aa$ should be at least $h$, where $h$ is a predetermined parameter. As will be clear in \S \ref{sec:Examples}, $h$ bounds the meshsize of the resulting graph from below.

\paragraph{Characteristic structure} Lastly, we make sure that $p_{d}$ does not violate the characteristic structure of the solution: $w_{d}$ must be real, and in addition $p_{d}$ has to satisfy $t_d \geq \max \{ t_a, t_i \}$ for causality to hold. We impose the more stringent condition: 
\begin{eqnarray*}
t_d \geq \max \{ t_a, t_i \} + \Delta t
\quad \mathrm{where} \quad
\Delta t = \frac{h}{\sqrt{G^{2}(\uvec_{j})+1}}
\end{eqnarray*} 
with $j=a$ if $\max \{ t_a, t_i \} = t_a$ and $j=i$ otherwise. 
This choice of $\Delta t$ speeds up computations and simplifies our analysis, as will be evident in \S \ref{subsubsec:OptimizationProperties} where $\Delta t = \beta h$.

\paragraph{Optimal sampling -- Summary} 
Given a pair of points $p_{a}$, $p_{i} \in \Aa$, we wish to minimize the objective function
\begin{align}
\label{ACCURACY}
f(\uvec_{d}) &= |\vec{s}_{a}|^{2} + |\vec{s}_{i}|^{2} 
\tag{A}
\end{align}
subject to the constraints: 
\begin{align}
\label{VALID1}
g_{1}(\uvec_{d}) &= \Im(w_{d}) = 0 
\tag{V1} 
\\ 
\label{VALID2}
g_{2}(\uvec_{d}) &= t_{d} - \max \{ t_a, t_i \} \geq \Delta t
\tag{V2} \\
\label{EVEN}
g_{3}(\uvec_{d}) &= \min_{ p \in \Aa } \{ \|p_{d} - p \| \} \geq h
\tag{E} 
\end{align}

\subsubsection{Solving the constrained minimization problem}
\label{subsubsec:SolvingOpt}

Although the objective function is simple, the constraints are non-linear functions of $\uvec_{d}$. Unlike the other two, constraint (\ref{VALID1}) does not require evaluating $w_{d}$. If using the direct solver, it amounts to checking the sign of the discriminant. Constraint (\ref{VALID2}) requires conversion of the data to $\xvec t$-coordinates. Constraint (\ref{EVEN}) can be very costly to verify if $\Aa$ is large. This is why it is preferable to work with a predetermined small subset of $\Aa$. 
To be efficient, a scheme solving this problem should not require too many evaluations of the constraints. We briefly discuss two possible methods, and relegate the details of the procedure to Appendix \ref{app:GridMethod}.

\paragraph{The grid method} This simple yet efficient approach is the one we use in practice. The problem is solved using grids sampling a neighbourhood of $p_{a}$ and $p_{i}$ in the $uv$-plane. The objective function $f$ is computed at those points where all three constraints are satisfied. The location of the minimum of $f$ is found, and a finer grid centered at this point is defined. The procedure is repeated until convergence \eg The change in the value of the minimum is below some pre-defined tolerance.

\paragraph{The Lagrangian method} Following \cite{bertsekas1996constrained} (\S 2.1, 2.2, 3.1), the augmented Lagrangian $L$ is defined by adding the three constraints to $f$ along with Lagrange multipliers $\mu_{i}$ and penalty coefficients $c_{i}>0$, where $i=1,2,3$. The following procedure is iterated until convergence \eg The change in the value of the minimum is below some pre-defined tolerance. First, for some values of the $c_{i}$ and $\mu_{i}$, the unconstrained minimum of $L$ is found. This can be done using BFGS along with line minimization \cite{NumericalRecipes}. Then the $c_{i}$'s are increased and the $\mu_{i}$'s are updated. Although convergence is guaranteed, this method requires a large number of evaluations of the constraints, which makes it too slow for our purposes.

\section{Algorithms}
\label{sec:Algorithms}

We present the algorithms used to generate the results in \S \ref{sec:Examples}.

\subsection{Initialization} Little information is required to initialize the algorithm.
The input is a sampling of $\CC_{0}$ consisting of $m$ points and $m-1$ undirected segments. The normal $\nthree$ at each point and a Final Time must also be provided.

Store those pieces of information into a list, say $\mathcal{L}$. In the following, we will assume that each row of $\mathcal{L}$ contains data pertaining to one point, \eg if $p=(\xvec,t)=(x_{1},x_{2},\ldots,t)$ has normal $\nthree(\xvec,t)$, then $\mathcal{L}$ might initially look like:
\begin{eqnarray*}
\begin{array}{c|c|c|c|c|c|c|c|c|c}
~ & x_{1} & x_{2} & \ldots & t & \hat{\nu}_{1}(\xvec,t) & \hat{\nu}_{2}(\xvec,t) & \ldots & \hat{\nu}_{n+1}(\xvec,t) \\ 
\hline \hline
 \mathrm{row~1}  & -.1625 & 0.0080 & \ldots & 0.0342 \times 10^{-15} &  -.1634 & 0.301 & \ldots & -0.5668 \\ 
 \mathrm{row~2}  & -.1705 & 0.0099 & \ldots & 0.0398 \times 10^{-15} &  -.1706 & 0.385 & \ldots & -0.5909 \\ 
~  & ~ & ~ & ~ & \vdots & ~ & ~ & ~ & ~ \\ 
\end{array}
\end{eqnarray*}
The variables $M_{\NB}$ and $M_{\Acc}$ help monitor the number of points in $\NB$ and $\Acc$, respectively. See Algorithm \ref{algo:Init}.
\begin{algorithm}[h!]
\caption{Initialization}\label{algo:Init}
\begin{algorithmic}[1]

	\State $\Acc \gets \mathcal{L}$, $M_{\Acc} \gets m+1$
	\State $\NB \gets \mathcal{L}$, $M_{\NB} \gets m+1$
	\State $h \gets \left( \min_{p, \, q \in \mathcal{L} } \| p-q \| \right) /2$. 

\end{algorithmic}
\end{algorithm}

\subsection{Get a local representation}
This routine is called when a point $p_{a}$ is accepted to go from a global to a local representation. See Algorithm \ref{algo:Local}.

\begin{algorithm}[h!]
\caption{Get a local representation}\label{algo:Local}
\begin{algorithmic}[1]
		\State Find the $L$ closest neighbours of $p_{a}$ among $\tilde{\Acc}$ and form the list $\mathcal{S}$. Set $\tilde{\Acc} = \mathcal{S}$.		
		\State Find $\bar{\nu}$ such that $\bar{\nu}\cdot \nthree(p_{i})>0$ for each $p_{i} \in \mathcal{S}$. The default value is $\bar{\nu} = \nthree_{a}$. 
		\State Compute the $uvw$-coordinates of the points in $\mathcal{S}$ using relation (\ref{eq:Dictionary}). 
		\State Find the point $p_{b} \in \mathcal{S}$ lying closest to $p_{a}$ with $u_{b}-u_{a}>h/2$.
		\State \textbf{Return:} $\tilde{\Acc}$, the change of coordinates matrix $M$, and $p_{b}$. 
\end{algorithmic}
\end{algorithm}

\subsection{Computing a new point} Once the minimum $\uvec_{d}$ is found using the direct solver, the iterative solver is run to improve the accuracy of the solution $w_{d}$. See Algorithm \ref{algo:ConsOpt}.
\begin{algorithm}[h!]
\caption{Compute a new point}\label{algo:ConsOpt}
\begin{algorithmic}[1]

	\State $\Delta t = h/\sqrt{G^{2}(\uvec_{a})+1}$, ~$p_{d} = \mathbf{0}$, ~ $\nthree_{d} = \mathbf{0}$.

	\State Solve the constrained optimization problem defined by $p_{a}$, $p_{i}$, $G(\uvec_{a})$, $h$ and $\Delta t$ to get $\uvec_{d}$. See \S \ref{subsec:OptimalSampling} and Appendix \ref{app:GridMethod} for details. 

	\If{a solution to the optimization problem was found}
		\State Compute $w_{d}$ associated with $\uvec_{d}$ using the iterative solver detailed in \S \ref{subsubsec:IterativeSolver}.
		\State Compute the normal $\nthree_{d}$ at $(\uvec_{d},w_{d})$. 
		\State Get $p_{d}$ and $\nthree_{d}$ in $(\xvec,t)$-coord. using the change of coordinates matrix $M$.
	\EndIf 
	
\State \textbf{Return:} $[p_{d}$; $\nthree_{d}]$.
\end{algorithmic}
\end{algorithm} 

\subsection{Global constraints}
\label{subsec:GlobalConstraints}

The global features of the manifold are monitored with the help of the structure \Book \, which consists of a list of segments. For example, if $\NB$ has the form: 
\begin{eqnarray}
\NB = \left[ p_{1}; p_{2}; \ldots; p_{M_{\NB}} \right] ^{T} 
\quad \mathrm{then} \quad 
\mathrm{\Book} = \left[ p_{2}-p_{5}; \, p_{11}-p_{23}; \,\ldots; \, p_{14}-p_{8} \right]^{T} 
\end{eqnarray}
Segments are not directed, \ie $[p_{1}-p_{2}]$ is the same as $[p_{2}-p_{1}]$. Initially, the collection of segments lies in the $xy$-plane and provides a piecewise linear approximation of $\CC_{0}$. We impose the following constraints on \Book.

\noindent \bull \textsl{Multiplicity:} Each point $p \in \NB$ appears exactly twice in \Book.

\noindent \bull \textsl{Loops:} A segment cannot start and end at the same point. A segment can only appear once in \Book. 

\noindent \bull \textsl{Intersections:} Segments cannot intersect.  

\noindent \bull \textsl{Spikes:} The acute angle formed by two segments sharing an endpoint cannot be less than $\approx 0.2\pi$. 

The situations are illustrated with simple examples on Figures \ref{fig:Multiplicity} - \ref{fig:Spikes}. The framed subfigures are obtained after connecting \emph{hanging nodes} which are nodes that appear only once in \Book. This is done as follows:

\paragraph{Stitching hanging nodes} Given a hanging node $p_{1} = (\xvec_{1},t_{1})$, let $p_{2}=(\xvec_{2},t_{2})$ be the hanging node that minimizes the $xy$-distance $\|\xvec_{1}-\xvec_{2} \|$. Then add the segment $[p_{1}-p_{2}]$ to \Book. Repeat until there are no hanging nodes left. 

Those global constraints are imposed to $\NB$ in Algorithm \ref{algo:BookKeeping}. Note that we omit the details of the updating procedure of \Book, which consists in relabelling nodes and possibly deleting some segments. See \cite{MyThesis} for details. 

\textsl{Remark:} When a new point is computed by Algorithm \ref{algo:ConsOpt}, the checks performed at lines 7 and 8 need only involve the segments attached to that new point, which requires an $\mathcal{O}(1)$ number of operations.

\begin{algorithm}[h!]
\caption{Book keeping of $\NB$}\label{algo:BookKeeping}
\begin{algorithmic}[1]

\If{$M_\NB>1$}
	\If{a new point was computed by Algorithm \ref{algo:ConsOpt}}
		\State Update \Book.
	\Else
		\State Update and clean \Book.
	\EndIf	

	\State Check intersections, and clean \Book \, if a point was removed from $\NB$.
	\State Check spikes, and clean \Book \, if a point was removed from $\NB$.

\EndIf

\begin{center} \rule{8cm}{0.4pt} \end{center}

\Procedure{Clean \Book.}{}
	\Do
\State Multi $= 0$; Loop $= 0$; Inter $= 0$;
		\State Check multiplicity, and set Multi to 1 if a point was removed from $\NB$. 
		\State Check loops, and set Loop to 1 if a point was removed from $\NB$. 
		\State Check intersections, and set Inter to 1 if a point was removed from $\NB$. 
	\doWhile{Multi+Loop+Inter $> 0$}
\EndProcedure

\end{algorithmic}
\end{algorithm}

\begin{figure}
  \centering
  \includegraphics[width=\linewidth]{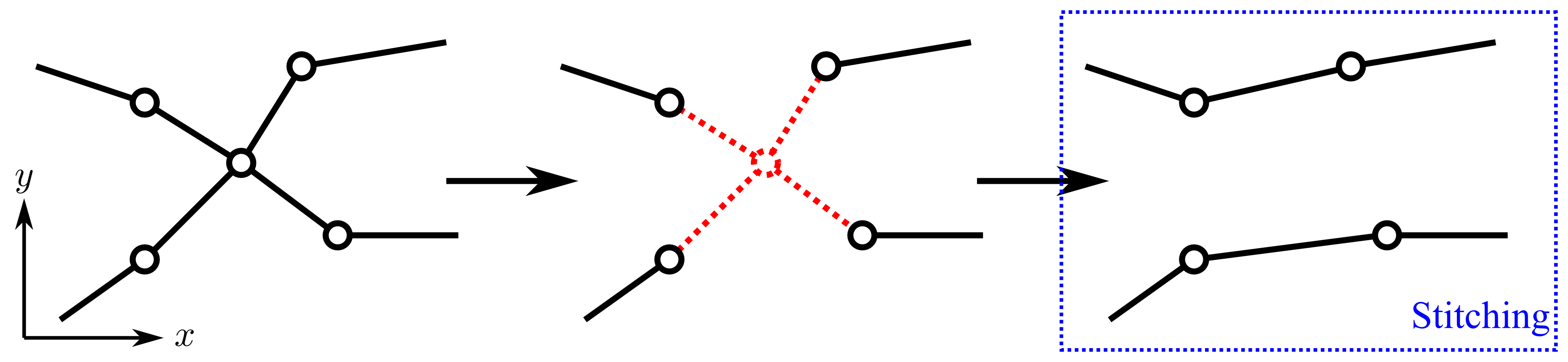}
  \caption{Multiplicity.}
	\label{fig:Multiplicity}
\end{figure}

\begin{figure}
  \centering
  \includegraphics[width=\linewidth]{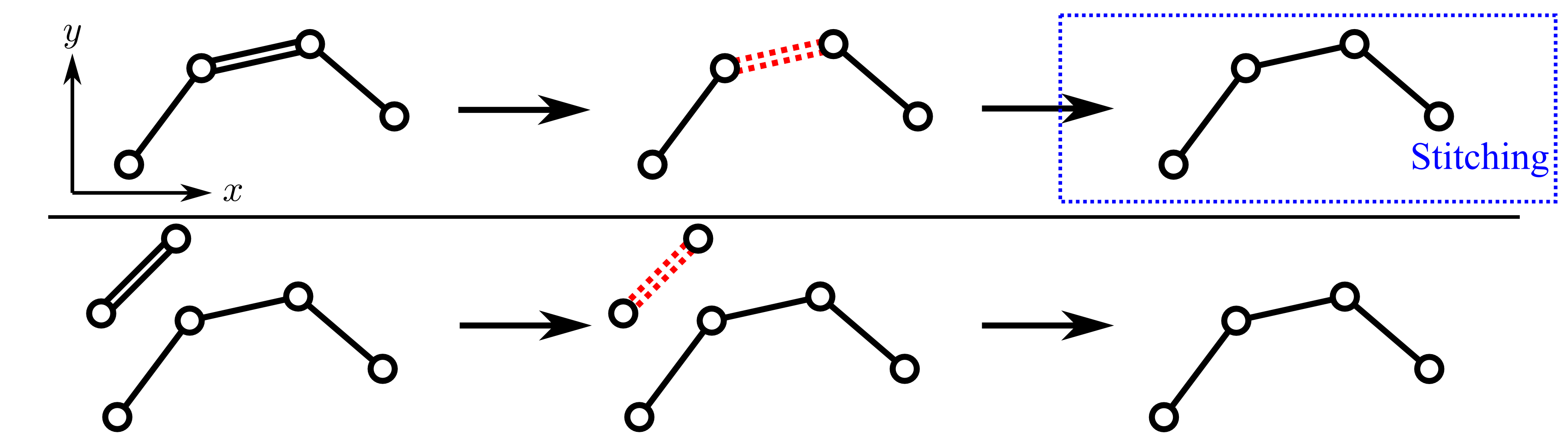}
  \caption{Loops.  Once the double-loop is removed, the nodes are either stitched (top situation) or removed from $\NB$ (bottom situation).}
	\label{fig:Loops}
\end{figure} 

\begin{figure}[b]
  \centering
  \includegraphics[width=\linewidth]{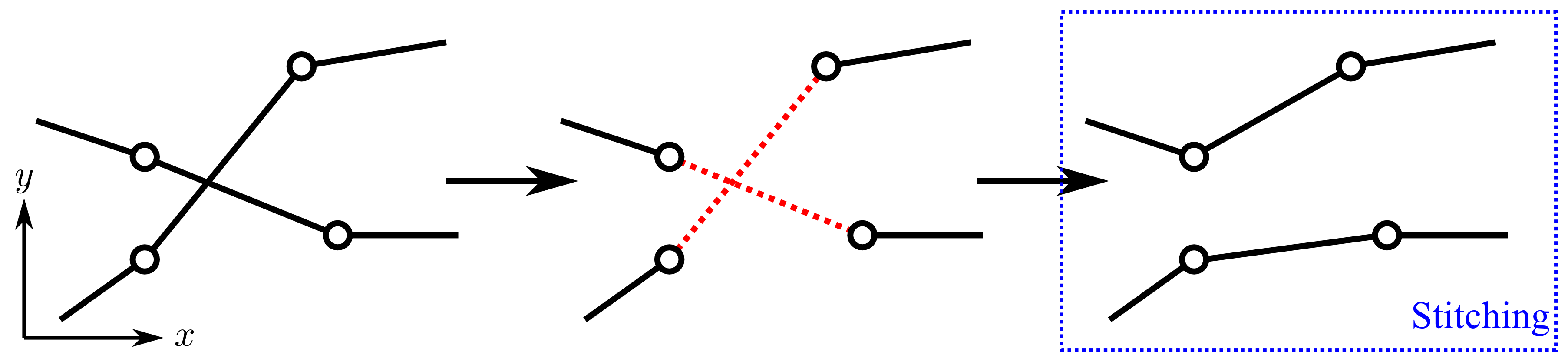}
  \caption{Intersections. The final configuration must differ from the original one.}
	\label{fig:Intersections}
\end{figure} 

\begin{figure}[b]
  \centering
  \includegraphics[width=\linewidth]{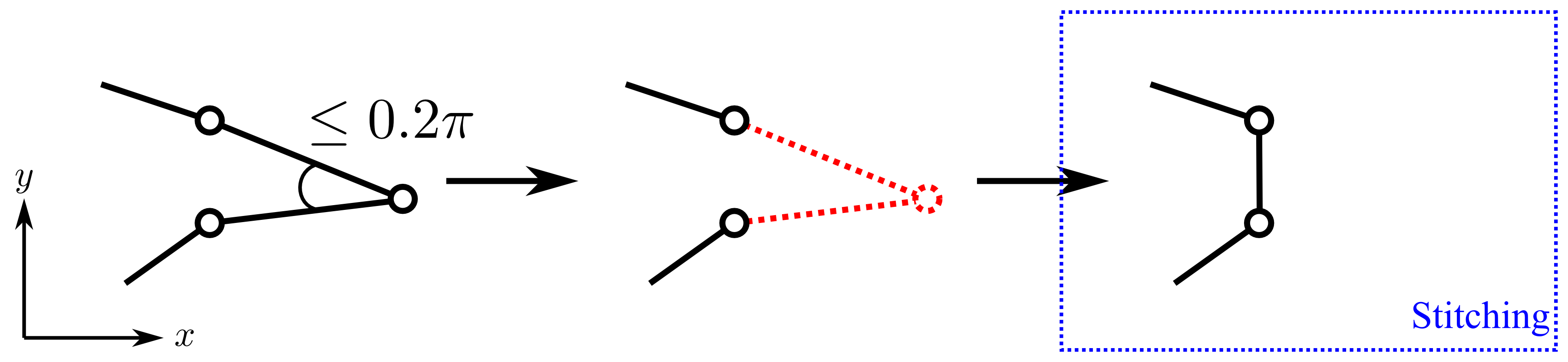}
  \caption{Spikes}
	\label{fig:Spikes}
\end{figure}

\clearpage

\subsection{Main loop}

A couple of remarks on Algorithm \ref{algo:MainLoop} are in order.
Lines 5-7 are ignored during the first $m$ iterations of the Main Loop: This can be considered as the end of the initialization procedure. 
The fact that the algorithm does not compute a new value when $t_{a}>$ Final Time implies that $\NB$ is eventually empty, at which point the algorithm naturally ends. 
If Algorithm \ref{algo:ConsOpt} fails, next-to-nearest neighbours are used (line 14). Since $|\tilde{\Acc}|\leq L$, this occurs at most $L$ times. Here, we present the algorithms with $i=b$. The procedure is unchanged if $i=c$.

\begin{algorithm}
\caption{Main Loop}\label{algo:MainLoop}
\begin{algorithmic}[1]

\State counter = 1
\While{$\NB$ is not empty} 

	\State Find $p\in\NB$ with the smallest $t$ value, call it $p_{a}=(\xvec_{a},t_{a})$.
	\State Remove $[p_{a};\nthree_{a}]$ from $\NB$. $M_{\NB} \gets M_{\NB}-1$
	\If{counter$>m$}
		\State Add $[p_{a};\nthree_{a}]$ to $\Acc$ in the $M_{\Acc}$-th row. $M_{\Acc} \gets M_{\Acc}+1$
	\EndIf 

	\If{$t_{a} <$ Final time} 
		\State Set $\tilde{\Acc} = \Acc$ and get a local representation using Algorithm \ref{algo:Local}.
		\State Run Algorithm \ref{algo:ConsOpt} with the pair $p_{a}\,$-$\,p_{b}$ to get $p_{d}$, $\nthree_{d}$.
		\If{Algorithm \ref{algo:ConsOpt} was successful at finding a new point}
			\State Add $[p_{d};\nthree_{d}]$ to $\NB$ in the $M_{\NB}$-th row. $M_{\NB} \gets M_{\NB}+1$
		\Else
			\State Remove $[p_{b};\nthree_{b}]$ from $\tilde{\Acc}$, find a new $p_{b} \in \tilde{\Acc}$, and go back to line 10. 
		\EndIf
	\EndIf 

\State Run Algorithm \ref{algo:BookKeeping} to impose the global constraints on $\NB$. 

\State counter $\gets$ counter+1
\EndWhile

\end{algorithmic}
\end{algorithm}

\section{Properties of the scheme}
\label{sec:Properties}
We first discuss local properties of the scheme by approximating the solution of the optimization problem, and asserting the convergence of the solvers. We then turn to global properties. 

\subsection{Local properties}
\label{subsec:LocalProperties}

Let us assume that $p_{a}$ and $p_{i} \in \Aa$ are exact, \ie $p_{a}$, $p_{i} \in \MM$; and that $\|p_{a}-p_{i} \| = \mathcal{O}(h)$. Moreover, consider that we are working on a neighbourhood where $\psi$ is $C^{1}$, so that $\bar{\nu} = \hat{\nu}_{a}$ and $w_{i}=w_{a}+o(h^{2})$. We investigate the properties of the algorithms used to compute $p_{d}$. 
We show that the constraints of the optimization problem are such that the discretizations of Equation (\ref{eq:PDEvector}) are monotone.

\subsubsection{Optimization}
\label{subsubsec:OptimizationProperties}
The objective function can be rewritten as: 
\begin{eqnarray}
\label{eq:objReformulated}
f(\uvec_{d}) &=& 2 \left(  \| \uvec_{d} - \uvec_{\min} \|^{2} - \uvec_{a}\cdot \uvec_{i}^{T} \right)
\quad \mathrm{where} \quad \uvec_{\min} = \left( \frac{u_{a}+u_{i}}{2} , \frac{v_{a}+v_{i}}{2} \right)
\end{eqnarray}
\ie $f$ is a paraboloid which reaches its minimum at $\uvec_{\min}$. It follows from the convexity of $f$ that the solution of our optimization problem either is $\uvec_{\min}$, or lies on the boundary of the feasibility set. 
In the limit where $h\rightarrow 0$, we expect $w_{\min} \approx (w_{a}+w_{i})/2$. Recalling the relation $t=\beta v+\gamma w$, if $\gamma$ is small we may have $t_{\min} \leq \max \{ t_{a},t_{b} \}$ which violates constraint (V2). As a result, in general $\uvec_{\min}$ does not belong to the feasibility set. 

Turning to constraint (V1) requires the study of the sign of $\rho_{1}+\rho_{2}G^{2}_{0}$. It can be shown that 
\begin{eqnarray}
\rho_{1} 
\geq \|M\|^{2}\|N\|^{2} \epsilon 
\qquad \mathrm{for~some~} \epsilon = \epsilon(\hat{R},\vec{M},\vec{N}) >0
\end{eqnarray}
Picking $\delta>0$ such that $\epsilon \geq \delta^{2}/(1+\delta^{2})$ gives $\rho_{1}+\rho_{2}G^{2}_{0} \geq 0$ when $|G_{0}| <\delta$. When $|G_{0}| \geq \delta >0$, using the approximation $w_{i} = w_{a}$ yields the following estimate: 
\begin{eqnarray}
\rho_{1}+G_{0}^{2} \rho_{2}
~& \approx &~ \left( \gamma^{2} - \beta^{2} G_{0}^{2} \right) m^{2}_{u} + m^{2}_{v} 
~=~ m^{2}_{v} \geq 0
\end{eqnarray}

Next, we momentarily focus on the solution to the following subproblem: Minimize (A) subject only to constraint (V2). The constraint is equivalent to
\begin{eqnarray}
\label{eq:SolSub}
\beta \left( v_{d}-v_{a} \right) + \gamma \left( w_{d}-w_{a} \right) \geq \beta h
~ \Longrightarrow ~ 
\left( v_{d}-v_{a} \right) +o(h^{2}) \geq h
\end{eqnarray}
so that it primarily constrains $v_{d}$. Thus, as $h \rightarrow 0$, the solution tends to
\begin{eqnarray}
u_{d} = u_{\min} \qquad
v_{d} = \max \left\{ h + v_{a} , v_{\min} \right\} 
\end{eqnarray}

The interplay between (V2) and (E) must be addressed. If $p_{j} \in \Aa$ since $t_{a}\geq t_{j}$, 
\begin{eqnarray}
\beta \left( v_{d}-v_{j} \right) + \gamma \left( w_{d}-w_{j} \right) \geq \beta h
~ \Longrightarrow ~ 
\left( v_{d}-v_{j} \right) +\mathcal{O}(h^{2}) \geq h
\end{eqnarray}
and so $\|p_{d}-p_{i} \| \geq h$ as $h\longrightarrow 0$.

In conclusion: Under the assumption that $\psi$ is locally $C^{1}$, in the limit where $h \longrightarrow 0$, we have that (V1) holds, and satisfying constraint (V2) implies satisfying constraint (E). It follows that the solution of the optimization problem is given by the minimum of (\ref{eq:objReformulated}) subject to the constraint (\ref{eq:SolSub}). Note that in the event where the assumption on the regularity of $\psi$ is dropped, neither existence nor uniqueness can easily be asserted.

\subsubsection{Direct solver}
\label{subsubsec:DirectSolverProperties} 
Given a location $\uvec_{d}$, the direct solver returns a value $w^{\mathrm{dir}}_{d}$ using either (\ref{SolutionQuadratic}) or (\ref{SolutionLinear}). We first proceed to showing that this discretization is degenerate elliptic in the sense of \cite{Oberman}, \ie Letting
\begin{eqnarray}
&~& \bar{H}(\psi_{i},\psi_{a}) := -\left[ \gamma - \beta \psi_{v}(\psi_{a},\psi_{i}) \right] - G_{0} \sqrt{ \left( \beta + \gamma \psi_{v}(\psi_{a},\psi_{i}) \right)^{2} + \psi^{2}_{u}(\psi_{a},\psi_{i}) }
\end{eqnarray}
we show that $\bar{H}$ is nondecreasing in each variable. Recall Relation (\ref{Rotation}):
\begin{eqnarray*}
\psi_{u}(\psi_{a},\psi_{i}) = \frac{1}{\det B} \left( (s_{i})_{v} \psi_{a}  - (s_{a})_{v} \psi_{i} \right) 
\quad
\psi_{v}(\psi_{a},\psi_{i}) = \frac{1}{\det B} \left( -(s_{i})_{u} \psi_{a} + (s_{a})_{u} \psi_{i} \right) 
\end{eqnarray*}
From the previous subsection, we expect $|\vec{s}_{a}|$ and $|\vec{s}_{i}|$ to be $\mathcal{O}(h)$, as well as
\begin{eqnarray}
-\frac{(s_{i})_{u}}{\det B} \, = \, c_{1} \, \geq \, 0
\quad \mathrm{and} \quad 
\frac{(s_{a})_{u}}{\det B} \, = \, c_{2}  \, \geq \, 0
\end{eqnarray}
Moreover, given our assumption that $\bar{\nu} = \nthree(p_{a})$, we have that $\nabla \psi(\uvec_{a}) = \vec{0}$, and
\begin{eqnarray}
\nabla \psi(\uvec_{d}) = \nabla \psi(\uvec_{a}) + \vec{\mathfrak{h}} = \vec{\mathfrak{h}}
\end{eqnarray}
where $\vec{\mathfrak{h}}$ if such that $|\vec{\mathfrak{h}}| = \mathcal{O}(h)$.  We have:
\begin{eqnarray}
\frac{\partial \bar{H}}{\partial \psi_{a}} 
&=& \beta \frac{\partial \psi_{v}}{\partial \psi_{a}} - G_{0}
\frac{ (\beta+\gamma\psi_{v}) \gamma \frac{\partial \psi_{v}}{\partial \psi_{a}} + \psi_{u}\frac{\partial \psi_{u}}{\partial \psi_{a}}}{\sqrt{ \left( \beta + \gamma \psi_{v} \right)^{2} + \psi^{2}_{u} }} \\
&=& 
\beta \frac{\partial \psi_{v}}{\partial \psi_{a}}
- G_{0} \frac{\gamma \frac{\partial \psi_{v}}{\partial \psi_{a}} + \left( \gamma^{2} \psi_{v}\frac{\partial \psi_{v}}{\partial \psi_{a}} + \psi_{u}\frac{\partial \psi_{u}}{\partial \psi_{a}} \right)/\beta }{\sqrt{ 1 + \left( 2 \beta \gamma \psi_{v} + \gamma^{2} \psi^{2}_{v} + \psi^{2}_{u} \right)/\beta^{2} }} \\
&=& 
\beta c_{1} 
- G_{0} \frac{\gamma c_{1} + \mathcal{O}(h) }{\sqrt{ 1 + \mathcal{O}(h)}} \\
&=& 
\left( \beta-G_{0}\gamma \right) c_{1} + \mathcal{O}(h)\\ 
&=& 
\sqrt{1+G^{2}_{0}} \, c_{1} + \mathcal{O}(h) ~ \geq ~ 0 \qquad \mathrm{for~} h \mathrm{~small~enough}
\end{eqnarray}
The argument that $\partial \bar{H} / \partial \psi_{i} \geq 0$ is completely analogous. 
We then show that the solver is stable in the sense of Barles \& Souganidis \cite{barles1991convergence}. Consider again the quadratic:
\[
 \left( (G_{0}^{2}+1)k_{1} - G^{2}_{0} k_{4} \right) w_{d}^{2} 
+ 2 \left( (G_{0}^{2}+1) k_{2} - G_{0}^{2}k_{5} \right) w_{d} 
+ \left( (G_{0}^{2}+1) k_{3} - G_{0}^{2}k_{6} \right)  = 0 \]
Using (\ref{eq:SpecialR}), (\ref{eq:ApproxDeriv}), (\ref{eq:Coeff}) and letting $|\vec{s}_{i}|=c_{i}h$, where $c_{i}$ is $\mathcal{O}(1)$, we may rewrite it as:
\begin{eqnarray}
aw^{2}_{d} + 2(b_{1}+hb_{2})w_{d}+(c_{1}+hc_{2}) = 0
\quad \mathrm{where~} a,b_{1},b_{2},c_{1},c_{2} ~\mathrm{are}~ \mathcal{O}(1)
\end{eqnarray}
Assuming that $h<1$, it is easily shown that the solution can be bounded by a bound independent of $h$. 
Verifying the consistency of the scheme reveals that it is first-order accurate with respect to $h$. Locally uniform convergence follows from \cite{barles1991convergence}.

\subsubsection{Iterative solver}
\label{subsubsec:IterativeSolverProperties} 

Given a location $\uvec_{d}$, the iterative solver returns a value $w^{\mathrm{iter}}_{d}$ using (\ref{eq:IterativeSolver}). The initial guess is provided by the direct solver. We let:
\begin{eqnarray*}
\tilde{H}(w^{n}_{d},\psi^{n}_{a},\psi^{n}_{i}) : = -\left[ \gamma - \beta \psi_{v}(\psi^{n}_{a},\psi^{n}_{i}) \right] - G^{n} \sqrt{ \left( \beta + \gamma \psi_{v}(\psi^{n}_{a},\psi^{n}_{i}) \right)^{2} + \psi^{2}_{u}(\psi^{n}_{a},\psi^{n}_{i}) }
\end{eqnarray*} 
where $\psi^{n}_{i} = (w^{n}_{d}-w_{i})/|\vec{s}_{i}|$ for $i=a,b,c$, and 
$G^{n}:= F( \, \xvec (\uvec_{d},w^{n}_{d}), t(\uvec_{d},w^{n}_{d}) \, )$. Without loss of generality, we assume that the scheme is proper in the sense of \cite{Oberman}, considering $\tilde{H}+\tilde{\epsilon}w^{n}$ if necessary. From the previous section, it follows that $\tilde{H}$ is a degenerate elliptic scheme. Theorem 8 of \cite{Oberman} guarantees convergence to the unique solution for arbitrary initial data. 

\subsection{Global properties}

\subsubsection{The sets $\Aa$ and $\NB$}
We demonstrate some properties of these sets. 

\begin{lemma}
The size of $\NB$ cannot exceed $m$, the number of points sampling $\CC_{0}$.
\end{lemma}
\begin{proof}
Initially, the size of $\NB$ is $m$. 
Each iteration of the main loop successively: Removes exactly one point from $\NB$ (line 4); Adds at most one point to $\NB$ (line 12); May remove points from $\NB$ (Algorithm \ref{algo:BookKeeping}).
\end{proof}

The following lemma justifies promoting the point in $\NB$ with the smallest time value to the set $\Aa$. 
\begin{lemma}
Subsequent iterations of the Main Loop will not affect the value of the point in $\NB$ with the smallest time value. 
\end{lemma}
\begin{proof}
Suppose that during the $k$-th iteration of the main loop, $p_{a} \in \NB$ is promoted to the set $\Aa$, and let us denote by $p_{\ast}$ the point in $\NB$ with the smallest time value. It follows from constraint (V2) that if $p_{a}$ has a child point, its time value satisfies $t_{d} \geq t_{a} + \Delta t > t_{a} \geq t_{\ast}$. Therefore $p_{d}$ cannot belong to the past domain of influence of $p_{\ast}$. 
\end{proof}

\textsl{Remark:} It is not possible to closely mimick the proofs of Sethian \cite{Sethian} or Tsitsiklis \cite{Tsitsi}. Indeed, those proofs assume the existence of a graph where the values of $\psi$ (\resp $V$) are to be found. In constrast, in our framework, the choice of which point gets promoted to the set $\Aa$ directly influences the shape of the graph sampling $\MM$. 
$\triangle$

\subsubsection{Global convergence}
\label{subsubsec:GlobalConv}

When discussing the global convergence of the scheme, we distinguish between different cases based on the regularity and topological properties of $\MM$. 

\bull \textsl{$\MM$ is $C^{1}$.} Then the local solvers may be used anywhere on $\MM$, since the correspondence between the LSE and the Dirichlet problem highlighted in \S \ref{subsec:PDE} holds in the strong sense. As $h\rightarrow 0$, the assumptions listed in the beginning of \S \ref{subsec:LocalProperties} hold and imply that the local solvers are convergent. Considering the relation $\Delta t = h/\sqrt{G^{2}(\uvec_{a})+1}$ as a CFL condition ensures the stability of the global scheme. Global convergence is thus expected. 

\bull \textsl{$\MM$ is $C^{0}$ and the singularities arise from topological changes.} Our argument is that the global constraints imposed on $\NB$ are such that the local solvers are only used in regions of $\MM$ that are smooth. Indeed, topological changes imply that the segments monitored by \Book \, intersect. Such a situation should be prevented by Algorithm \ref{algo:BookKeeping}. It is therefore expected that the scheme is convergent in those situations as well -- although no proof is available at this point. 

\bull \textsl{$\MM$ is $C^{0}$.} Should the speed $F$ be merely continuous, $\MM$ may develop singularities without undergoing topological changes. In those situations, it is unclear whether the scheme globally converges. Let us note that the iterative scheme in \S\ref{subsubsec:IterativeSolverProperties} should converge regardless of the regularity of the Hamiltonian. Nonetheless, the size of the pseudo-time step $\Delta \tau$ required for convergence may be hard to determine. For that reason, the examples considered in the next section only involve speeds that are $C^{1}$ everywhere along the curve.

\subsubsection{Computational time}
\label{subsubsec:ComputationalTime}
Recall that, in what follows, $n=2$ and $m$ is the number of points sampling $\CC_{0}$.
We first bound the number of operations in one iteration of the main loop. 
Finding the minimum among $\NB$ can be done in $\mathcal{O}(\log m)$ iterations using a binary heap \cite{Sethian}. Finding the $L$ closest neighbours of $p_{a}$ among $\Aa$, where $L$ is typically $\sim 10$ requires $\sim mL$ number of operations. (Since $\Aa$ is ordered in time, it is only necessary to look at the tail of the list.)

The grid method performs at most five iterations on an $s^{n}$ grid, where $s$ is typically $\sim 10$. Computing $w$ and evaluating constraints (V1) and (V2) at each point requires $\sim 10$ operations. Evaluating constraint (E) requires $\sim L$ operations. We arrive at a total of $5 s^{n} \times 10 \times L = \mathcal{O}(10^{n+2})$ operations to find $\uvec_{d}$. 
The complexity of the iterative solver is hard to bound \textsl{a priori}. Indeed, it is currently unclear how many operations may be required to compute the CFL condition. However, when $F \in C^{1}(\RR^{n}\times [0,T])$, we find that devoting $\sim 10^2$ operations to this task results in a solver that only requires $\mathcal{O}(10)$ iterations to converge. 
Updating \Book \, in Algorithm \ref{algo:BookKeeping} requires $\mathcal{O}(m)$ operations, and as noted earlier, the other procedures usually require $\mathcal{O}(1)$ operations, but may require up to $\mathcal{O}(m)$ (\eg When topological changes occur).  
We arrive at the conclusion that the cost of one iteration is bounded by: $\max \{ \mathcal{O}(m) , \mathcal{O}\left(10^{n+2}\right) \}$.

The parameter $m$ is determined by the user and directly influences the total number of points $N$ computed by the algorithm. 
We have $h \approx \mathrm{length}(\CC_{0})/(2m)$ so that $h = \mathcal{O}(1/m)$. 

\textsl{When $F \equiv C$ or $F=F(t)$:} Then $\MM$ is roughly decomposed into $I$ strips comprising $\mathcal{O}(m)$ points. 
We have $I\leq T_{\mathrm{F}} / (\Delta t)_{\min}$ where $(\Delta t)_{\min} = h/\sqrt{\max G^{2}+1} =\mathcal{O}(h)$, which yields a reliable estimate for the total number of points:
\begin{eqnarray}
N 
= \left( \# \mathrm{~of~strips} \right) \, \times \, \left( \# \mathrm{~points~per~strip} \right)
= \frac{T_{\mathrm{F}}}{\mathcal{O}(h)} \, \times \, \mathcal{O}(m) 
= \mathcal{O}(m^{2})
\end{eqnarray}

\textsl{When $F = F(\xvec,t)$:} Then $\MM$ no longer exhibits a stratified structure. However, in general, we still expect a strip of average height $\Delta t = \mathcal{O}(h)$ to contain $\mathcal{O}(m)$ points, and therefore arrive at the same estimate for $N$.

In summary, the total cost of the algorithm is $CN$ where
\begin{eqnarray}
C = \max \{ \mathcal{O}(m) , \mathcal{O}\left(10^{n+2}\right) \}
~~ \mathrm{and} ~~
N = \mathcal{O}(m^{2}) 
\end{eqnarray}

\textsl{Remark:} Although the number of operations required to find $\uvec_{d}$ is admittedly large, they can be performed fairly efficiently. In MatLab for example, the entire procedure can be vectorized. Empirically, it is found to take no more than 33\% of the total computational time. (See Appendix \ref{app:TestsProcedures}.) $\qquad \triangle$

\section{Results}
\label{sec:Examples}

We illustrate the properties of the algorithm through various examples. 
More details about the tests procedures can be found in Appendix \ref{app:TestsProcedures}. 

\subsection{Examples} 
With the exception of \Ex \ExTwoCircles, $\CC_{0}$ consists of a circle centered at the origin. The corresponding manifolds $\MM$ are presented on Figure \ref{fig:ExactM}. The exact normal is used to initialize the algorithm. 

\paragraph{(a) The expanding circle} Using the constant speed $F \equiv 1$ yields a manifold $\MM$ that consists of a truncated cone. 

\paragraph{(b) The football} The speed is set to $F(t)=1-e^{10t-1}$ resulting in a  circle that first expands and then contracts. $\MM$ then resembles an american football. 

\paragraph{(c) Two circles} As in the previous example, the time-dependent speed $F(t)=1-2t$ changes sign. However, the fact that $\CC_{0}$ now consists of two disjoint circles implies that topological changes occur during the evolution.

\paragraph{(d) The oscillating circle} The speed is set equal to $F= a\sin(b(t+c))$ for some constants $a$, $b$ and $c$. 

\paragraph{(e) The escaping circle} The speed is such that the circle first expands, and then moves in the positive $x$-direction while growing. 

\paragraph{(f) The 3-leaved rose} $F$ depends on the polar angle in such a way that $\CC_{t}$ eventually takes the shape of a 3-leaved rose. 

\begin{figure}
\begin{subfigure}{.48\textwidth}
  \centering
  \includegraphics[width=\linewidth]{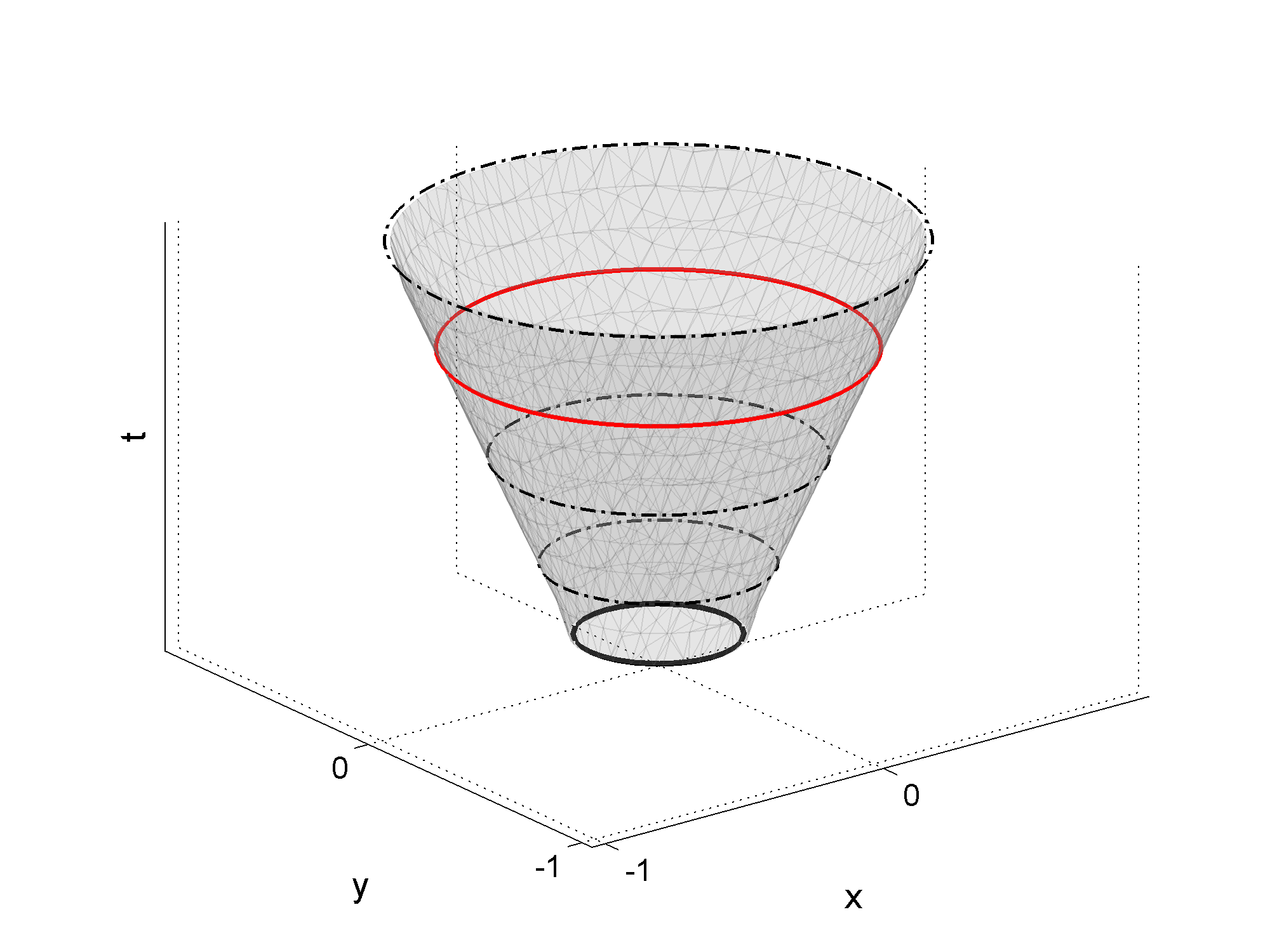}
  \caption{The expanding circle}
  \label{fig:Ma}
\end{subfigure}
\begin{subfigure}{.48\textwidth}
  \centering
  \includegraphics[width=\linewidth]{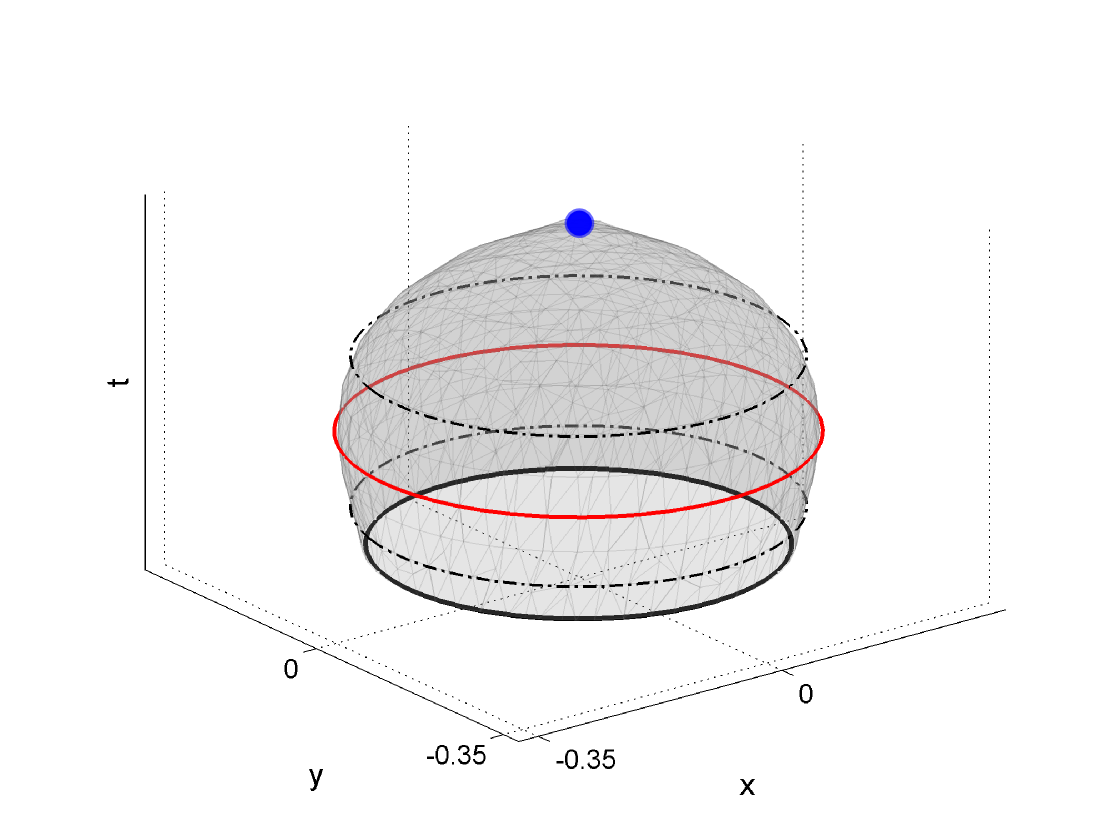}
  \caption{The football}
  \label{fig:Mb}
\end{subfigure} \\
\begin{subfigure}{.48\textwidth}
  \centering
  \includegraphics[width=\linewidth]{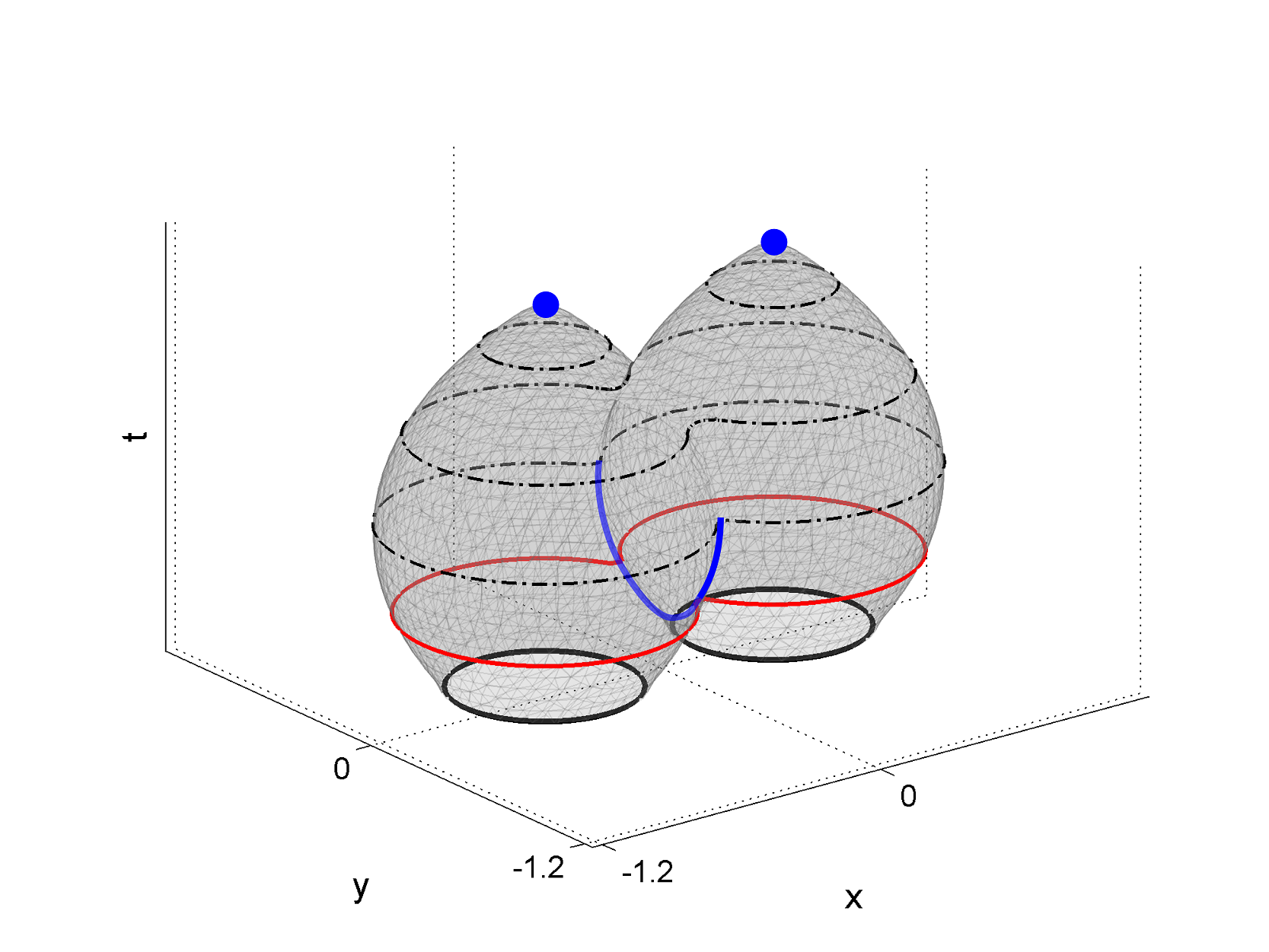}
  \caption{Two circles}
  \label{fig:Mc}
\end{subfigure}
\begin{subfigure}{.48\textwidth}
  \centering
  \includegraphics[width=\linewidth]{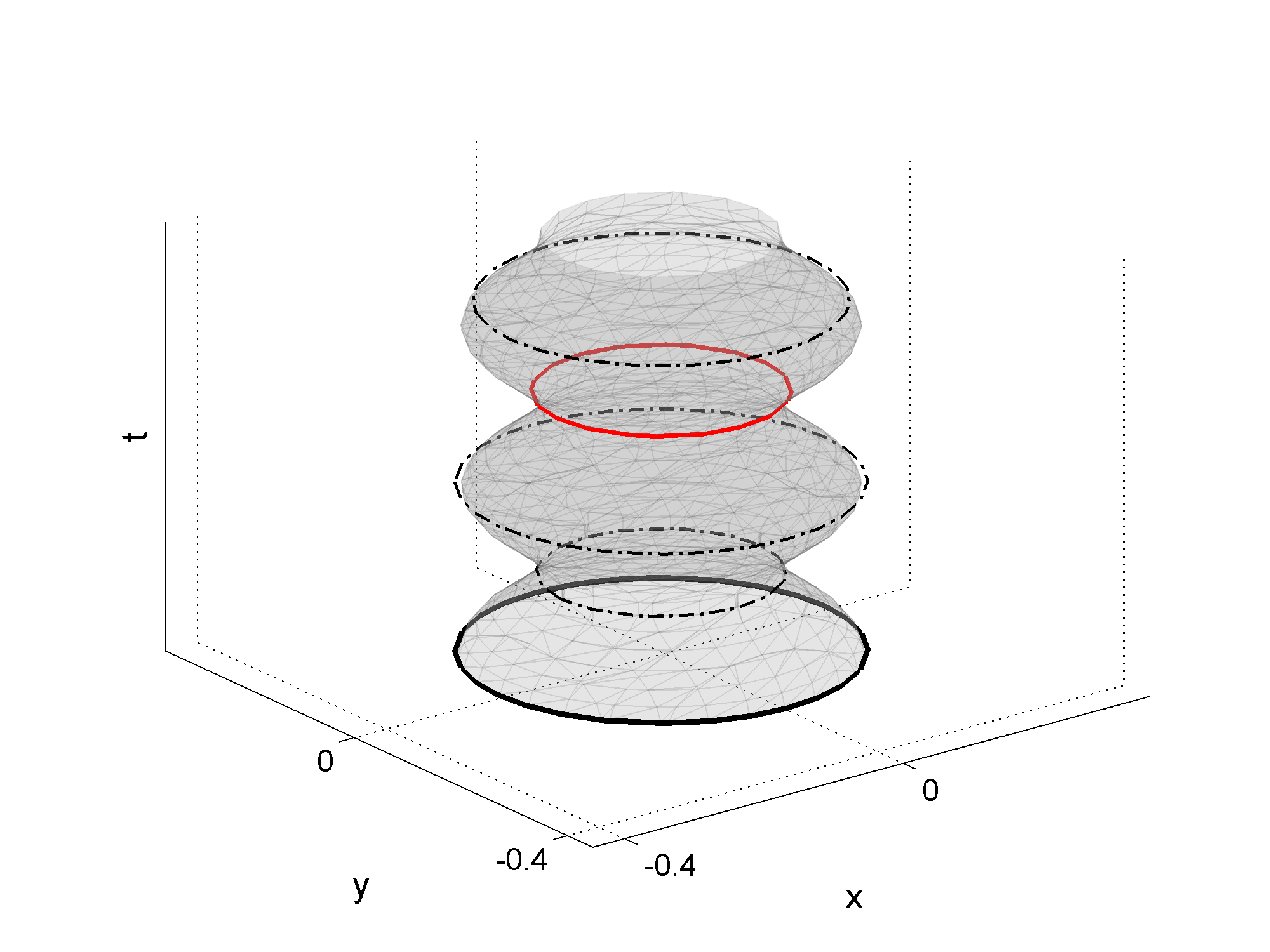}
  \caption{The oscillating circle}
  \label{fig:Md}
\end{subfigure} \\
\begin{subfigure}{.48\textwidth}
  \centering
  \includegraphics[width=\linewidth]{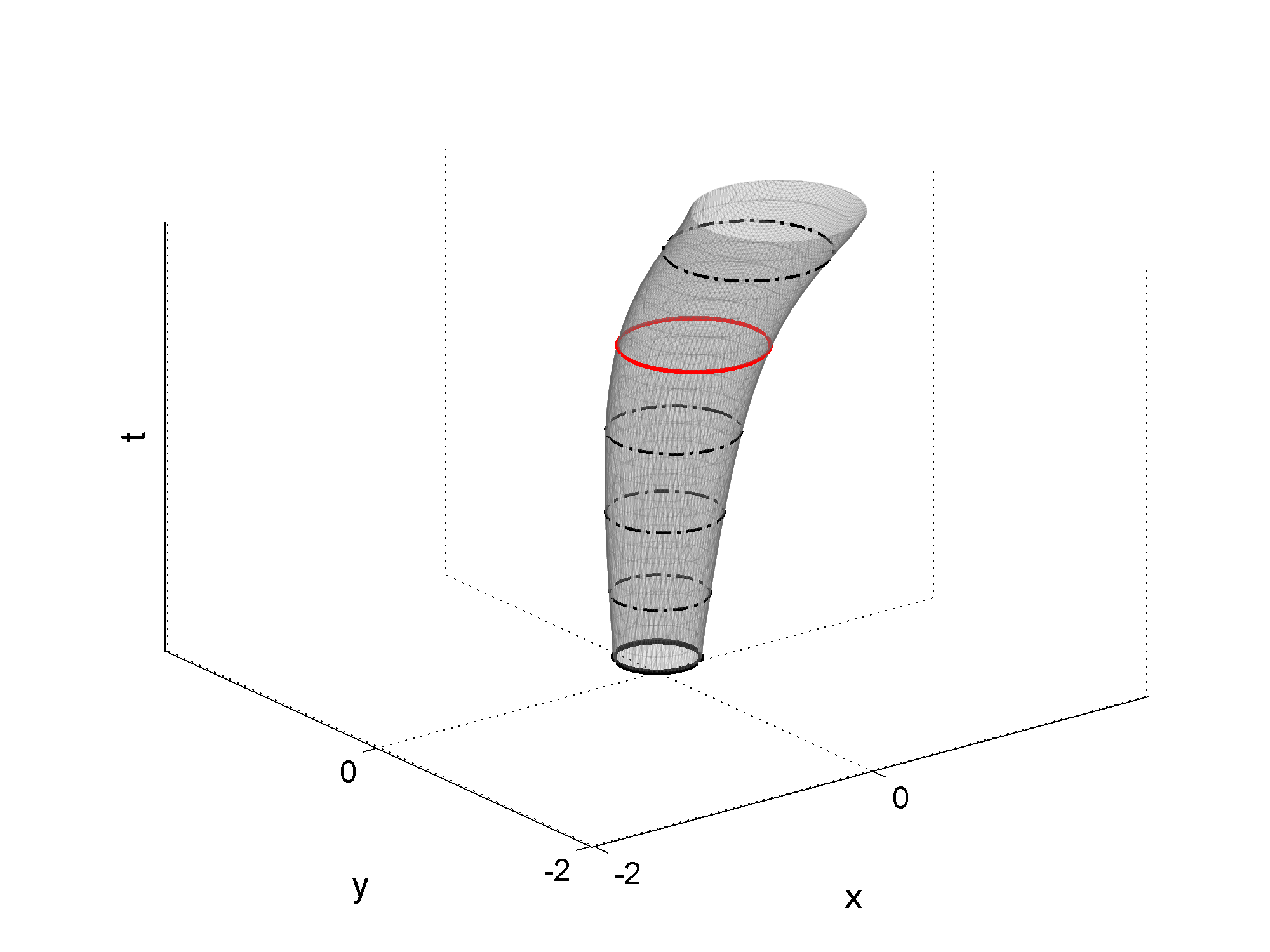}
  \caption{The escaping circle}
  \label{fig:Me}
\end{subfigure}
\begin{subfigure}{.48\textwidth}
  \centering
  \includegraphics[width=\linewidth]{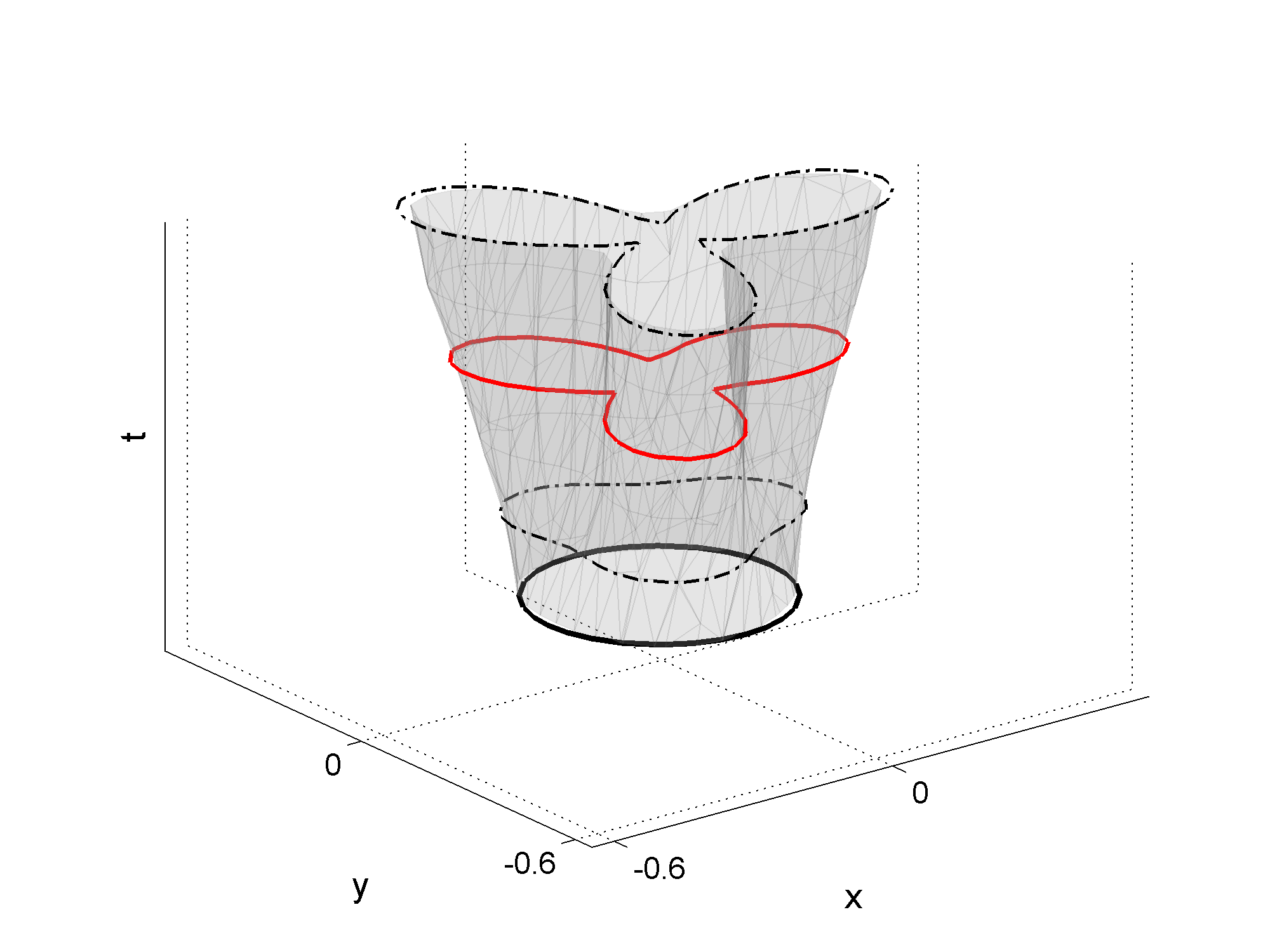}
  \caption{The 3-leaved rose}
  \label{fig:Mg}
\end{subfigure}
\caption{The manifold $\MM$ in $xyt$-space. $C^{0}$ points are featured in blue.}
\label{fig:ExactM}
\end{figure}


\begin{figure}
\begin{floatrow}
\ffigbox{%
  \includegraphics[width=0.48\textwidth]{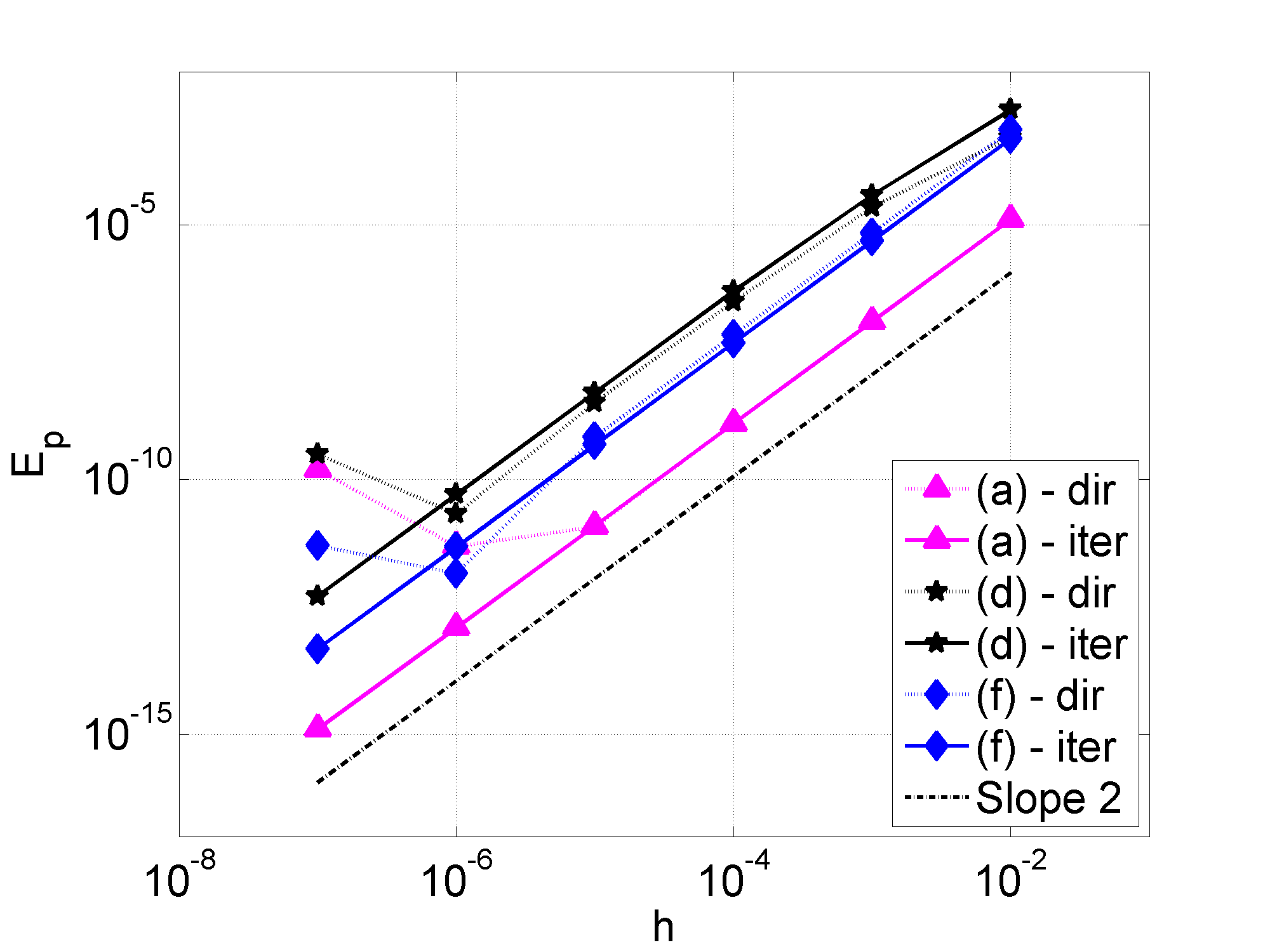}
}{%
		\caption{Local convergence results for examples \ExExpanding, \ExOscillatingCircle \, and \ExThreeLeavedRose.}
	\label{fig:LocalConvergence}
}
\capbtabbox{%

\footnotesize
\begin{tabular}{c||c|c|c}
\toprule
$h$ & (a) & (d) & (f) \\
\midrule
$10^{-2}$ & 1 & 10 & 9 \\ 
$10^{-3}$ & 1 & 9 & 9 \\
$10^{-4}$ & 1 & 9 & 8 \\
$10^{-5}$ & 4 & 8 & 7 \\ 
$10^{-6}$ & 6 & 7 & 6 \\
$10^{-7}$ & 8 & 10 & 7 \\
\bottomrule
\end{tabular}
\bigskip
\bigskip
}{%
\caption{Number of iterations.}
\label{tab:LocalProperties}
}
\end{floatrow}
\end{figure}

\normalsize

\begin{figure}
\begin{subfigure}{.48\textwidth}
  \centering
  \includegraphics[width=\linewidth]{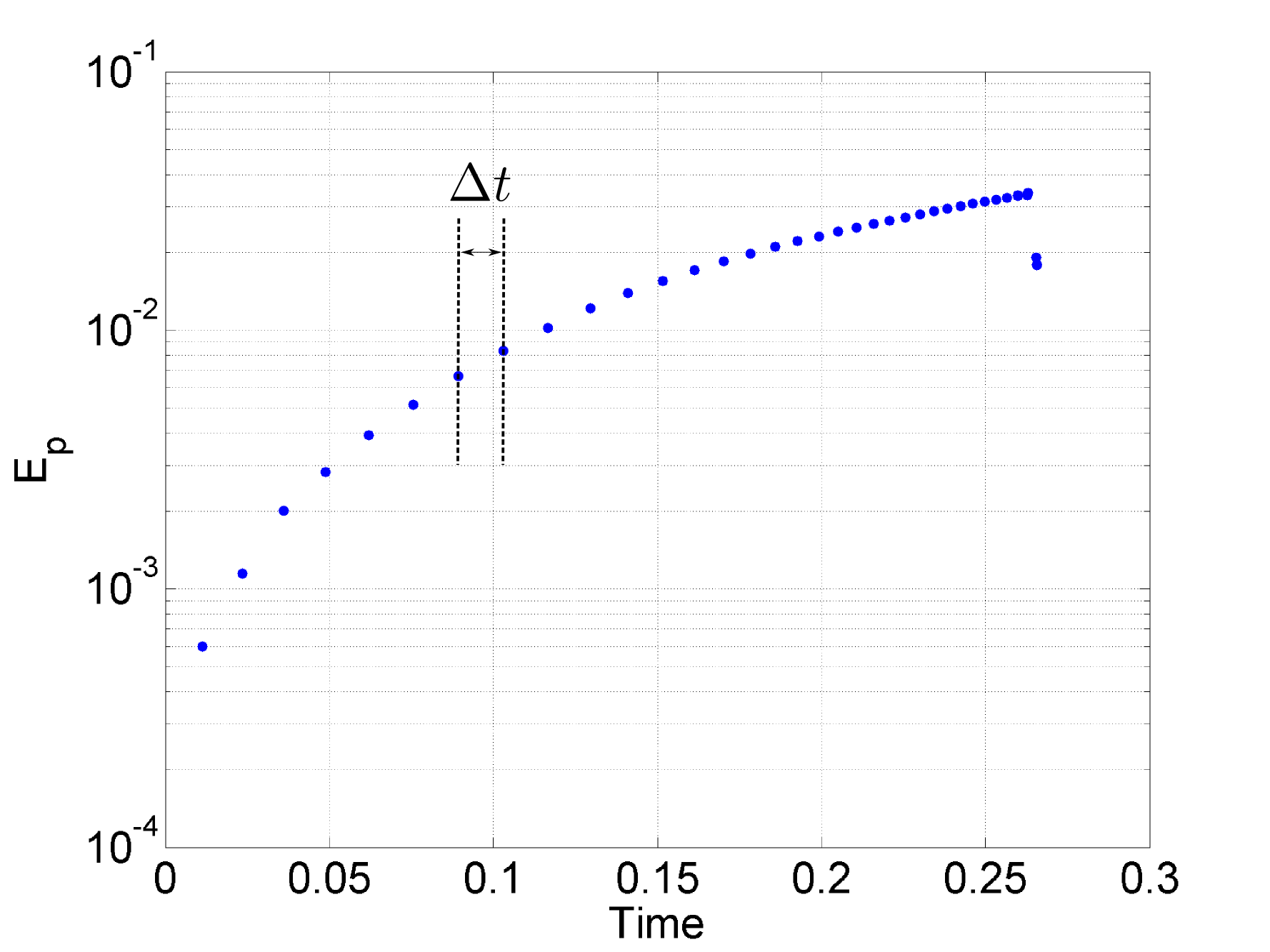}
  \caption{The football, \Ex \ExFootball 
}
  \label{fig:ErrorVsTimec}
\end{subfigure}
\begin{subfigure}{.48\textwidth}
  \centering
  \includegraphics[width=\linewidth]{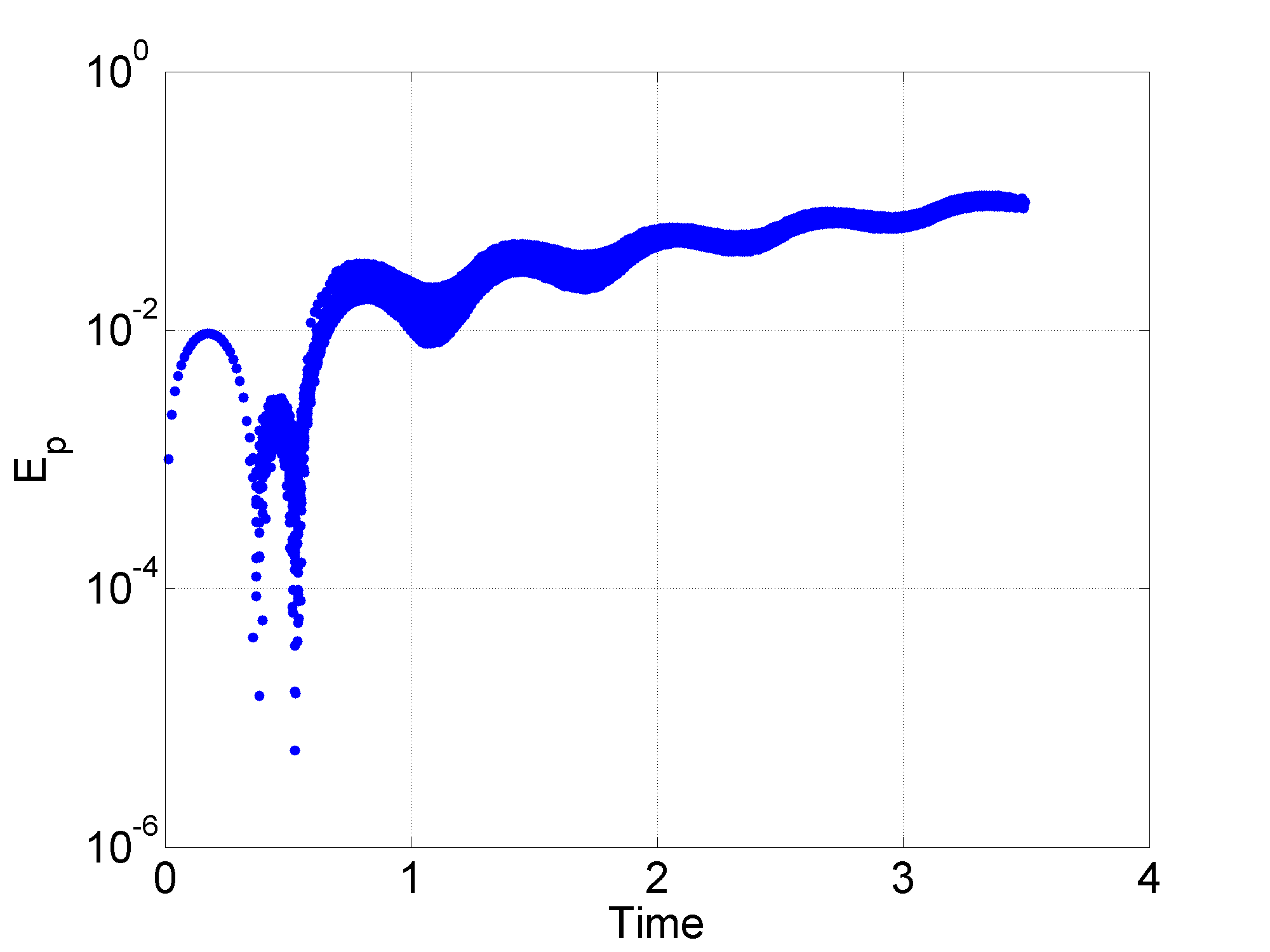}
  \caption{The oscillating circle, \Ex\ExOscillatingCircle  (6 periods). 
}
  \label{fig:ErrorVsTimed}
\end{subfigure} 
		\caption{Error versus time.}
	\label{fig:Convergence1Simulation}
\end{figure}

\subsection{Local properties} In the following, the error associated to each point is measured as $E_{p} := |\phi(p)| = |\phi(\xvec, t)|$, where $\phi$ is the exact solution to the LSE (\ref{eq:LSE}).

\subsubsection{Optimization problem} We start with an analysis of the grid method used to solve the optimization problem. 
Convergence results are presented on Figure \ref{fig:LocalConvergence} for examples \ExExpanding, \ExOscillatingCircle \, and \ExThreeLeavedRose. The errors associated with the solutions of the direct and the iterative solver are both recorded. Second order convergence with respect to the repulsion parameter $h$ is clear. This result is consistent with the test procedure: As $h$ gets smaller, the child point moves closer to its parents, but the distance between the two parent points also decreases. For the most general case represented by \Ex \ExThreeLeavedRose, the iterative solver slightly increases the accuracy of the solution. The number of iterations required to reach the tolerance is $\mathcal{O}(1)$, as recorded in Table \ref{tab:LocalProperties}.

\subsubsection{Error propagation}

We now discuss data obtained from running a full simulation. Convergence results for examples \ExFootball \, and \ExOscillatingCircle \, are presented on Figure \ref{fig:Convergence1Simulation}. In Subfigure \ref{fig:ErrorVsTimec}, the error is seen to behave nicely even near the ``tip" of $\MM$. Note that the code naturally stops once $\CC_{t} = \varnothing$. Moreover, the stratified structure alluded to in \S \ref{subsubsec:ComputationalTime} is evident. The symmetry of \Ex \ExOscillatingCircle \, gives rise to cancellations as the circle contracts. The error appears to increase at a slow pace. 

\subsection{Global properties} Define:
\begin{eqnarray} 
L_{1}(E) =h^{n} \, \sum_{p \in \Aa} E_{p}
\qquad
L_{2}(E) = \sqrt{ h^{n} \, \sum_{p \in \Aa} E^{2}_{p} }
\qquad
L_{\infty}(E) = \max_{p \in \Aa} E_{p}
\end{eqnarray}
with $n=2$, as well as the Haussdorff distance between the reconstructed and the exact curve, denoted by $L_{H}$. We first consider qualitative features of the manifold $\MM$ before turning to global convergence results and speed tests. 

\subsubsection{Evenness of the sampling}

As can be seen from Figure \ref{fig:EvenSampling}, the sampling of $\MM$ is regular. This is confirmed when the distance from a child to its parents is monitored in \Ex (b); see Figure \ref{fig:Distance}.

\begin{figure}[b!]

\begin{subfigure}{.48\textwidth}
  \centering
  \includegraphics[width=\linewidth]{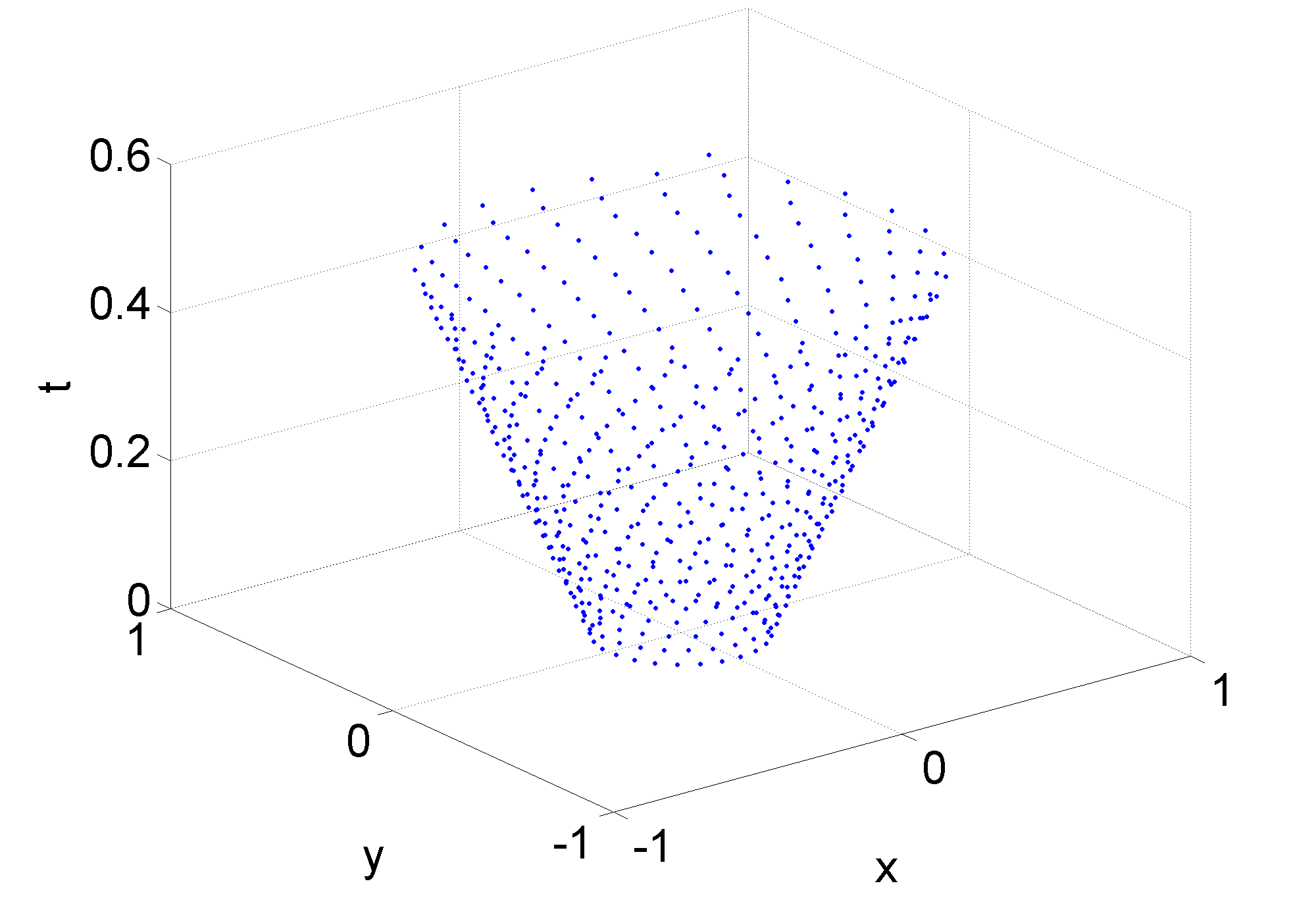}
  \caption{The expanding circle, \Ex \ExExpanding}
\label{fig:SurfaceExpandingCircle}
\end{subfigure}
\begin{subfigure}{.48\textwidth}
  \centering
  \includegraphics[width=\linewidth]{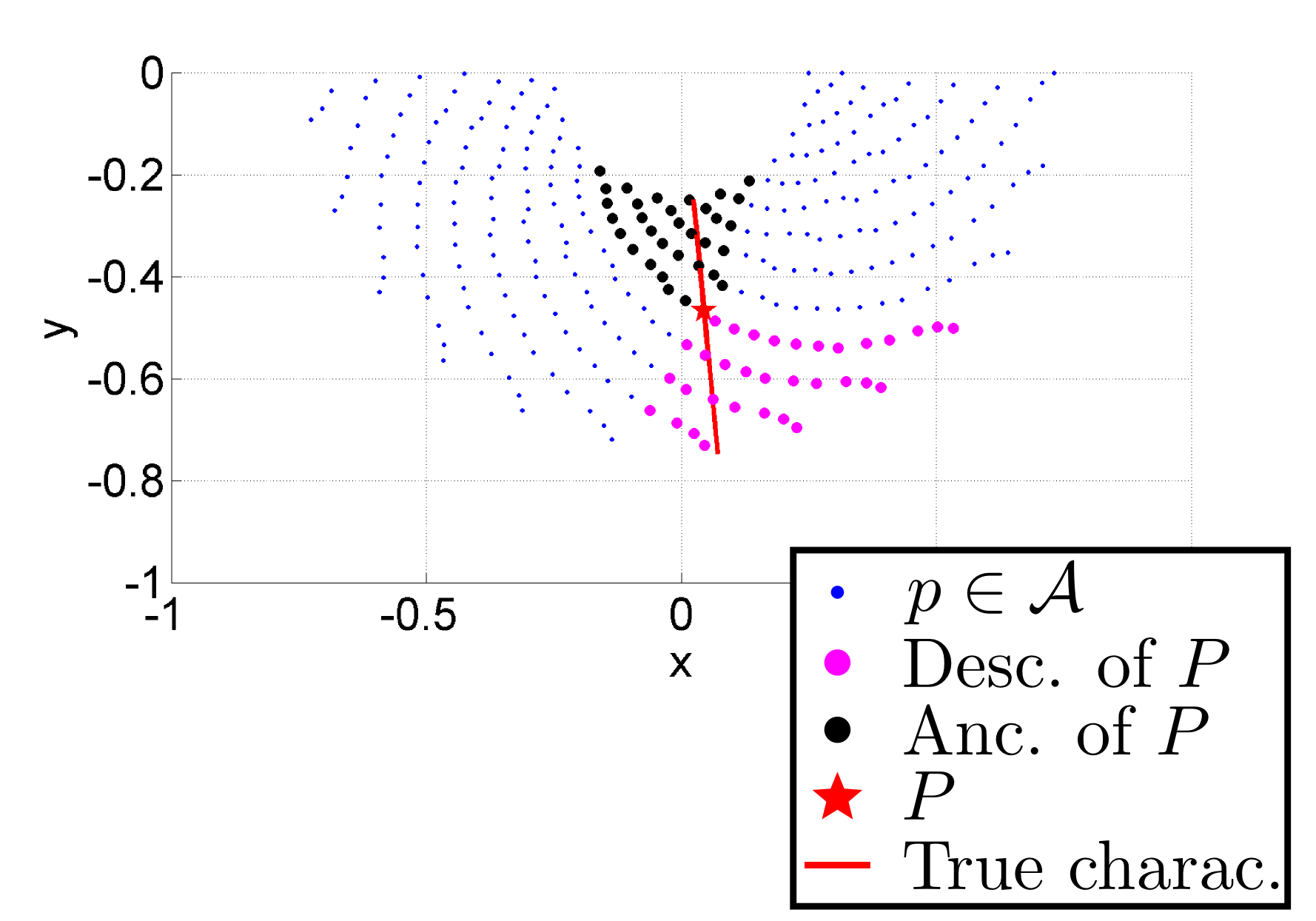}
  \caption{The expanding circle, \Ex \ExExpanding}
\label{fig:SurfaceExpandingCircleWithRelations}
\end{subfigure}

\begin{subfigure}{.48\textwidth}
  \centering
  \includegraphics[width=\linewidth]{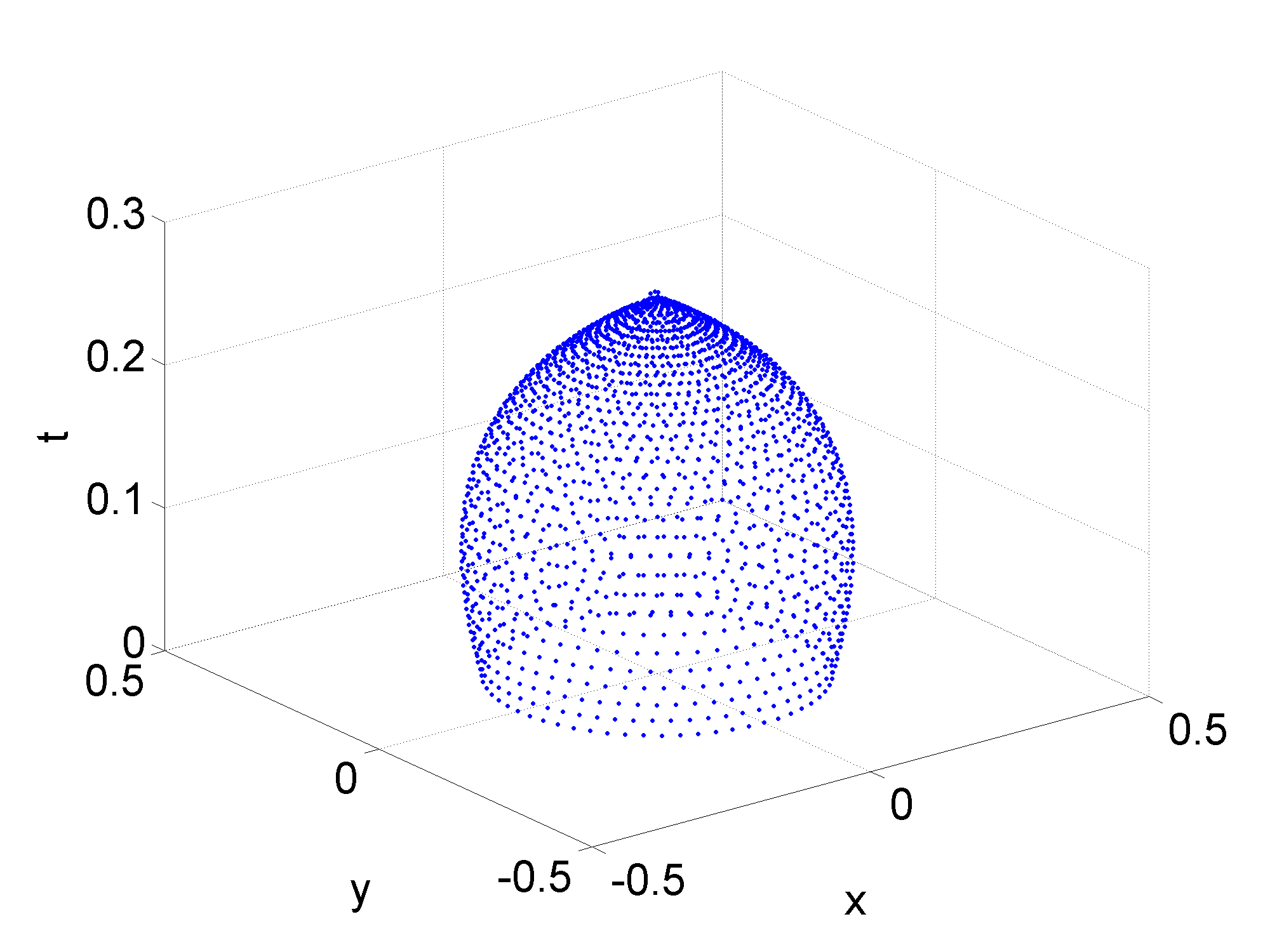}
  \caption{The football, \Ex \ExFootball 
}
\end{subfigure} 
\begin{subfigure}{.48\textwidth}
  \centering
  \includegraphics[width=\linewidth]{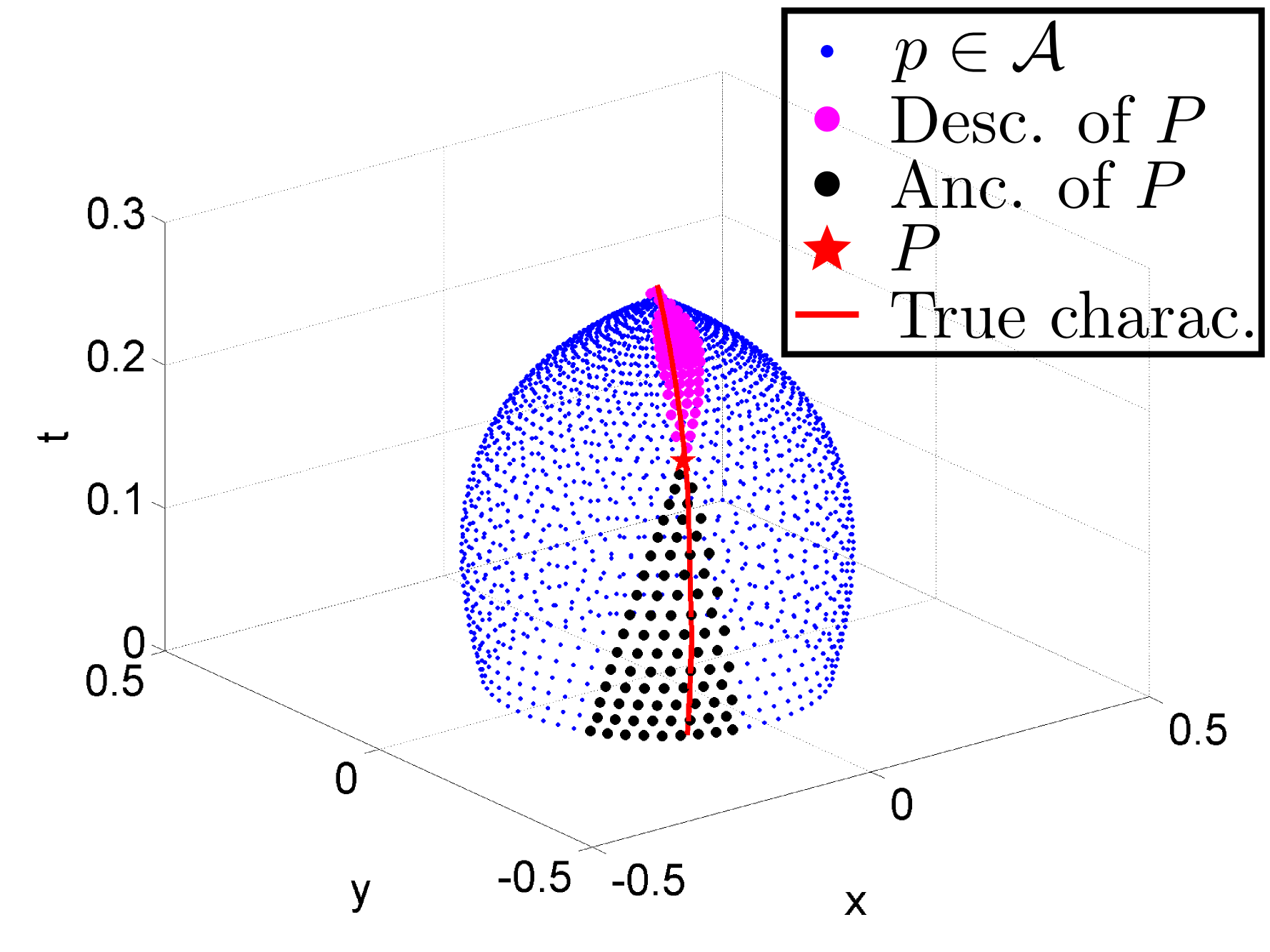}
  \caption{The football, \Ex \ExFootball  
}
\label{fig:SurfaceFootballWithRelationsModified}
\end{subfigure} 

\begin{subfigure}{.48\textwidth}
  \centering
  \includegraphics[width=\linewidth]{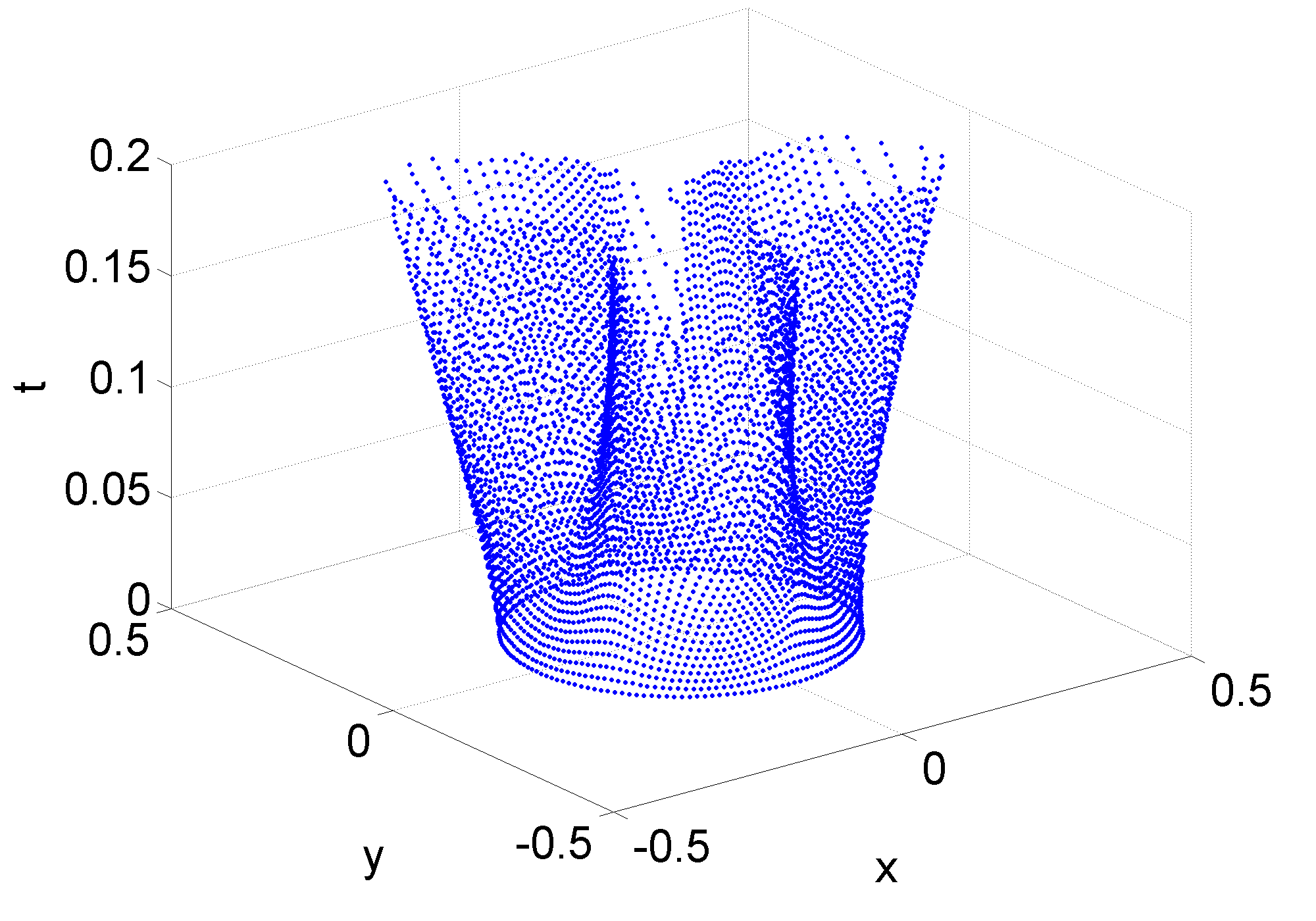}
  \caption{The 3-leaved rose, \Ex \ExThreeLeavedRose}
\end{subfigure}
\begin{subfigure}{.48\textwidth}
  \centering
  \includegraphics[width=\linewidth]{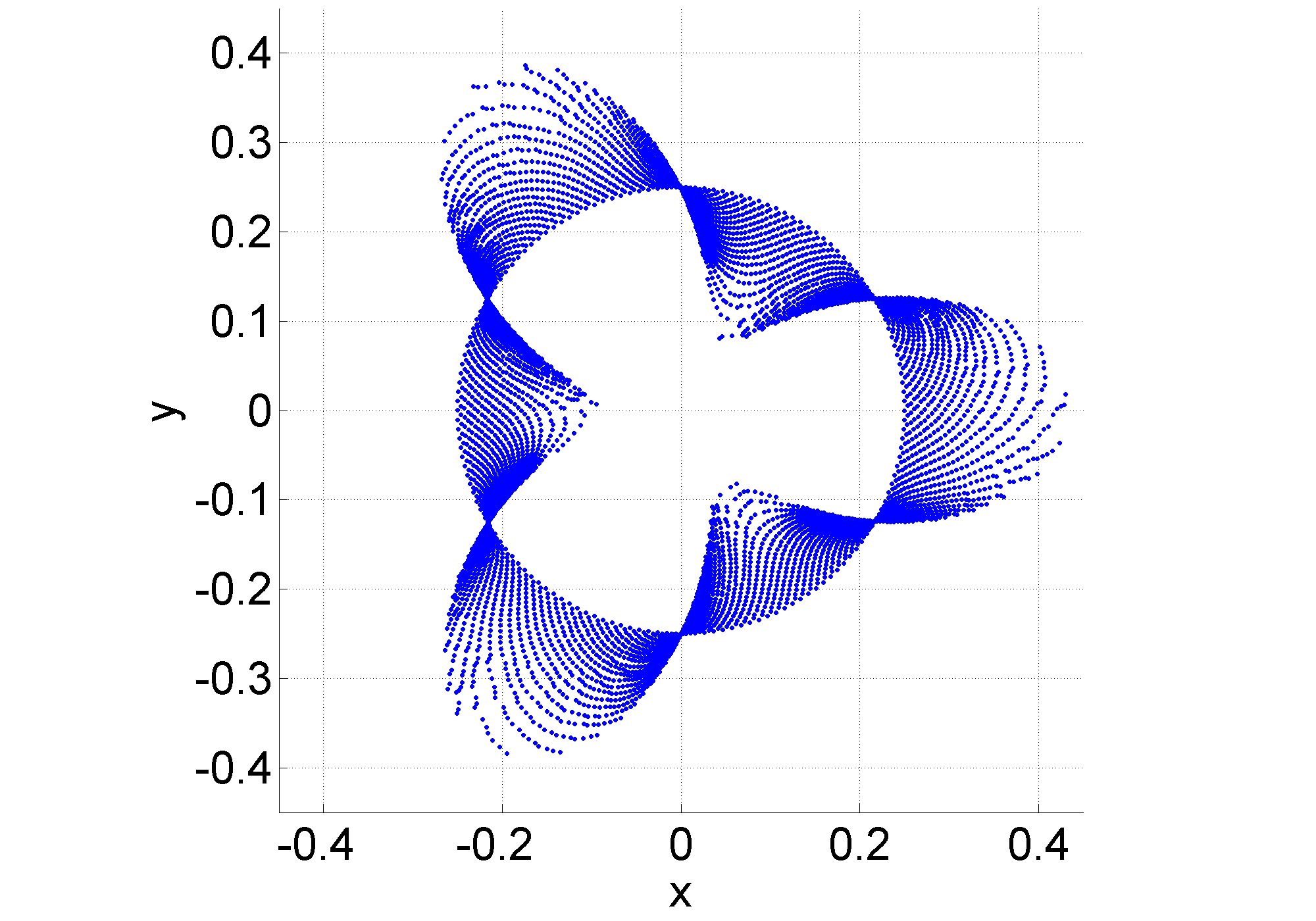}
  \caption{The 3-leaved rose, \Ex \ExThreeLeavedRose}
\end{subfigure}

		\caption{Sampling of $\MM$ returned by the the algorithm.}
	\label{fig:EvenSampling}
\end{figure}

\clearpage

\begin{figure}[h!]

\begin{subfigure}{.48\textwidth}
  \centering
  \includegraphics[width=\linewidth]{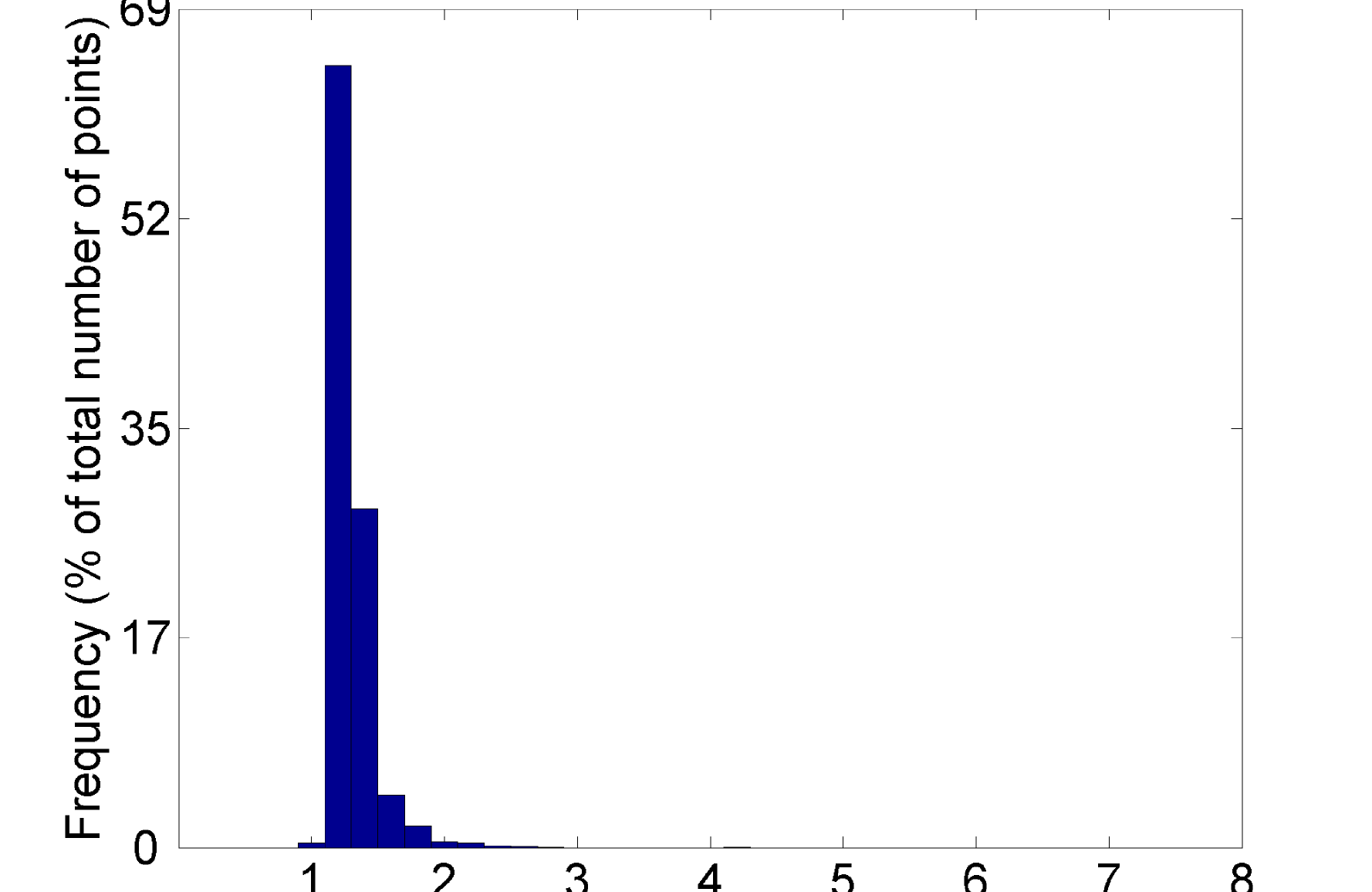}
  \caption{Distance to Parent $a$}
\end{subfigure} 
\begin{subfigure}{.48\textwidth}
  \centering
  \includegraphics[width=\linewidth]{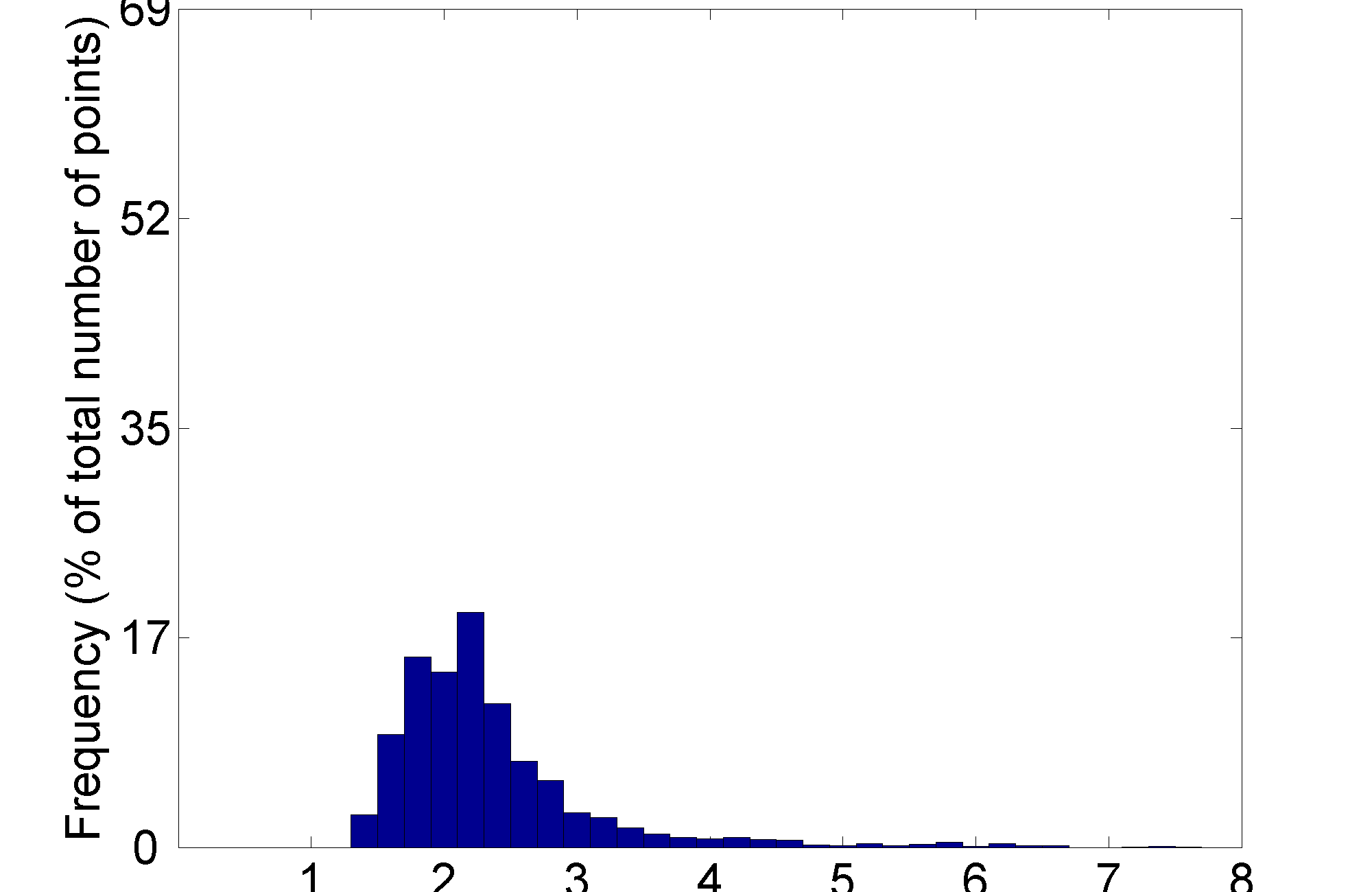}
  \caption{Distance to Parent $b$}
\end{subfigure}
\caption{Three-dimensional distance from a child to each of its parents, in units of $h$. 
}
\label{fig:Distance}
\end{figure}

\subsubsection{Domain of influence}

Subfigure \ref{fig:SurfaceExpandingCircle} may give the impression that the data propagate by spiralling outwards. However, this is not the case, as can be seen from tracking a point's ancestry and descendance, as on Subfigures \ref{fig:SurfaceExpandingCircleWithRelations} and \ref{fig:SurfaceFootballWithRelationsModified}. It is of particular interest to note that the true characteristic going through the point does lie in the numerical past and future domains of influence. 

\subsubsection{Topological changes} The two-circles example illustrates the ability of the algorithm to deal with topological changes. As can be seen from Figure \ref{fig:Topo}, the code naturally stops computing points when it reaches the $y$-axis. As a result, the circles appropriately merge. Their separation is also well-captured.

\begin{figure}[h!]

\begin{subfigure}{.48\textwidth}
  \centering
  \includegraphics[width=\linewidth]{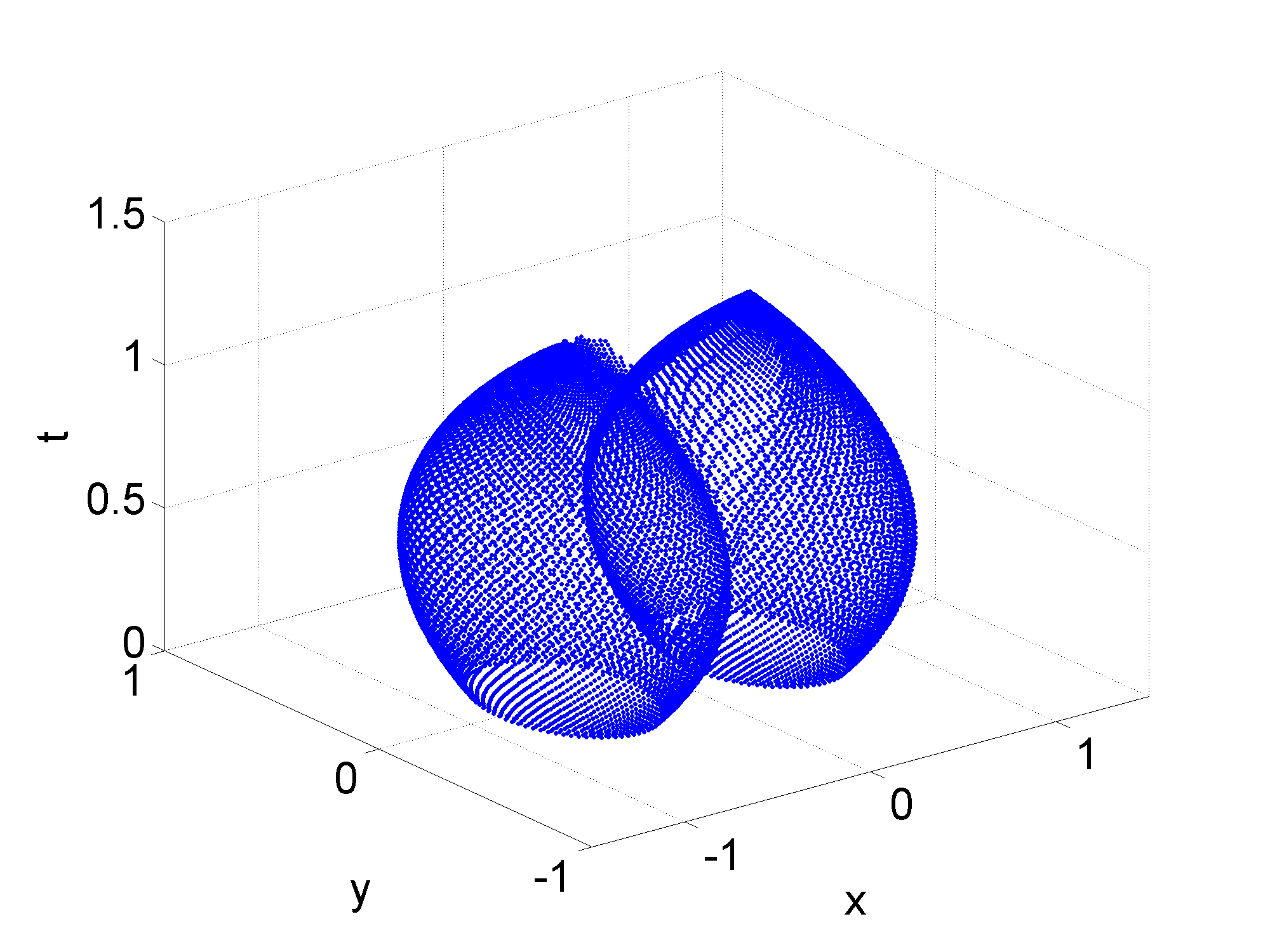}
  \caption{Full set.}
\end{subfigure} 
\begin{subfigure}{.48\textwidth}
  \centering
  \includegraphics[width=\linewidth]{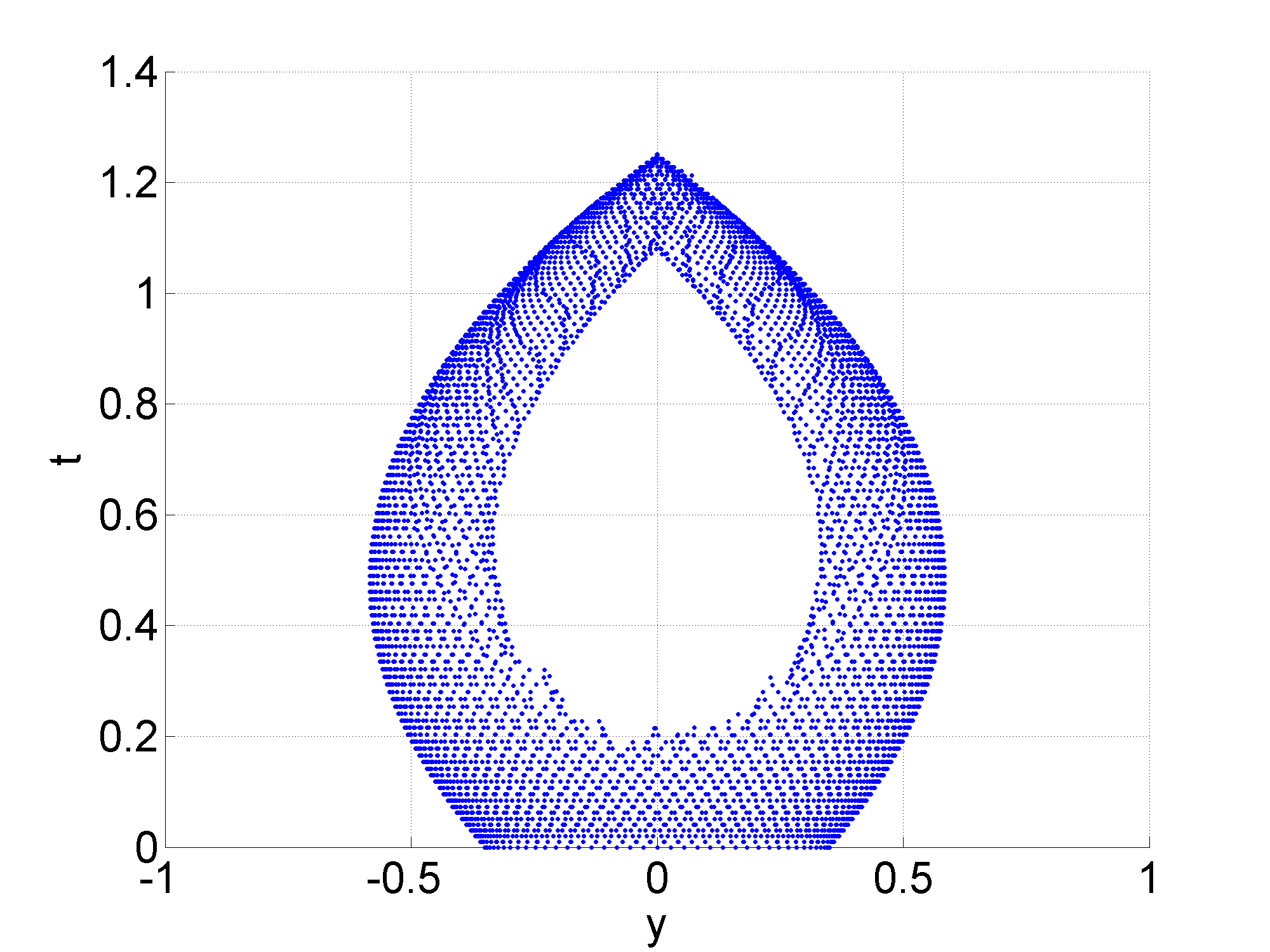}
  \caption{Side view of those points with $|x|<0.5$.}
\end{subfigure}
\caption{Sampling of $\MM$ returned by the algorithm for \Ex \ExTwoCircles, the two circles.}
\label{fig:Topo}
\end{figure}

\vspace{-.3cm}

\subsubsection{Convergence results}

Convergence results are presented in Figure \ref{fig:ConvergenceResults}, along with the exact and the reconstructed curves obtained for some simulations. 

\begin{figure}
\begin{subfigure}{.48\textwidth}
  \centering
  \includegraphics[width=\linewidth]{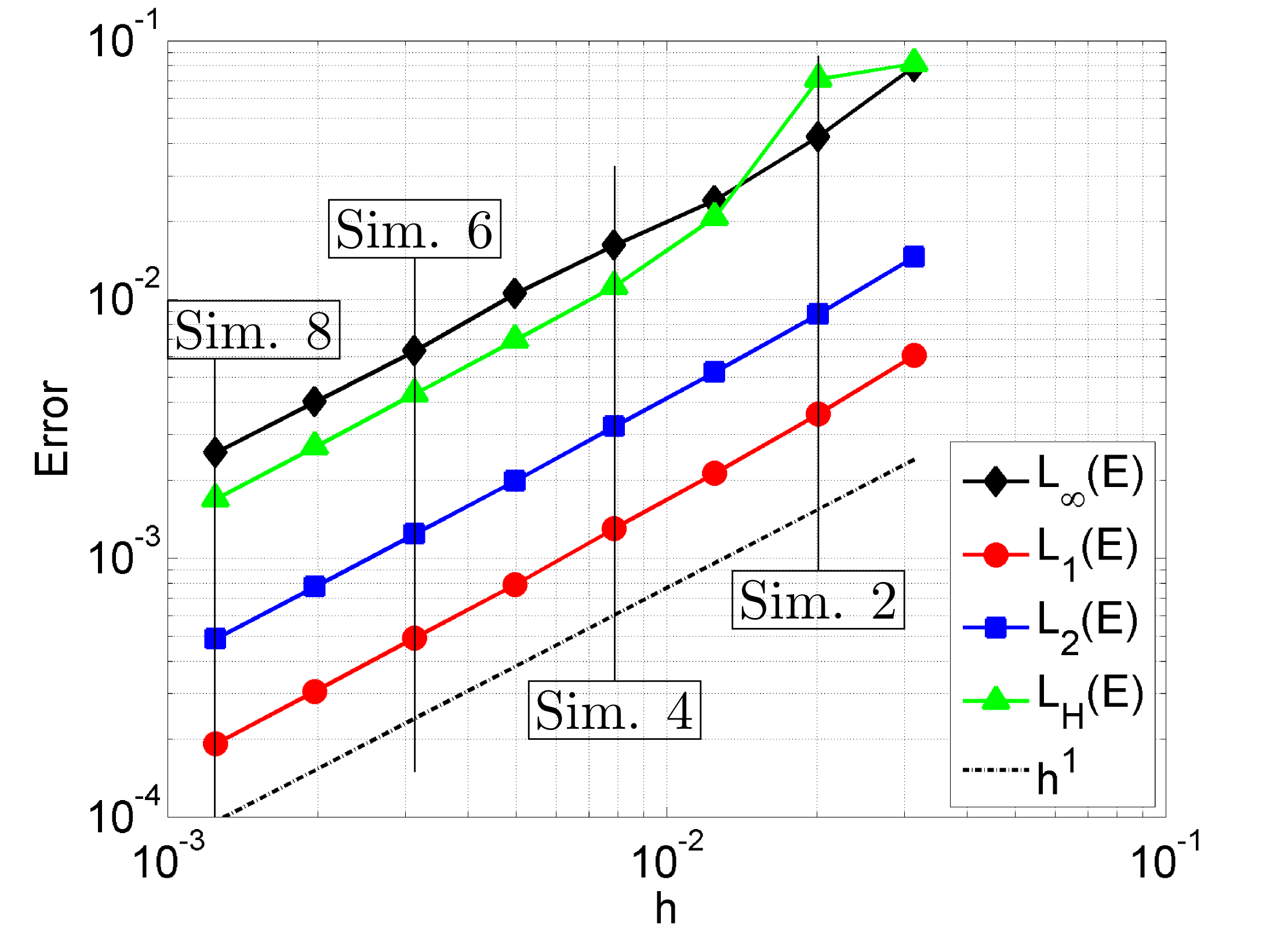}
  \caption{The escaping circle, \Ex \ExEscapingCircle.}
  \label{fig:Conve}
\end{subfigure}
\begin{subfigure}{.48\textwidth}
  \centering
  \includegraphics[width=\linewidth]{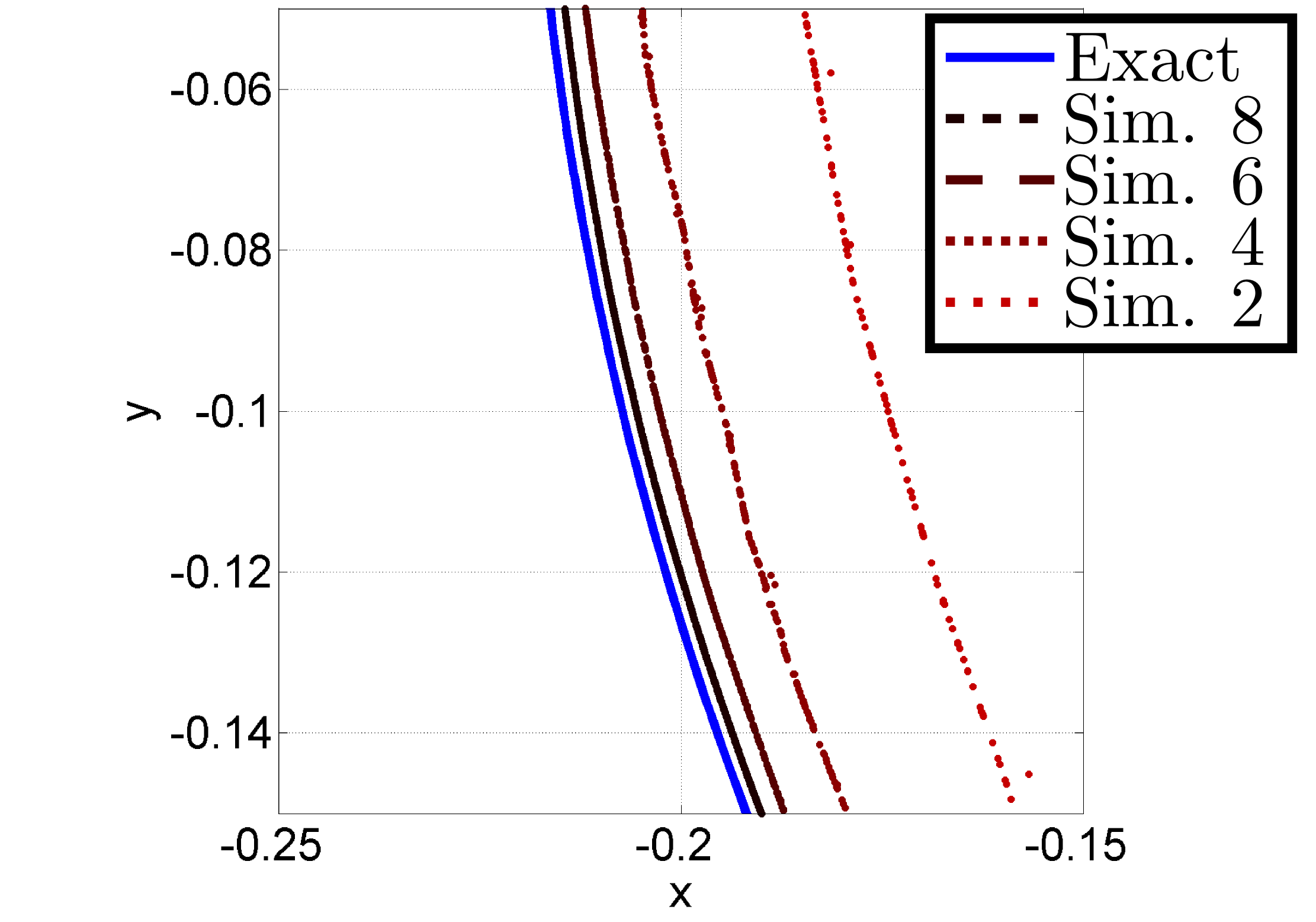}
  \caption{The escaping circle, \Ex \ExEscapingCircle.}
  \label{fig:Curvese}
\end{subfigure} 

\begin{subfigure}{.48\textwidth}
  \centering
  \includegraphics[width=\linewidth]{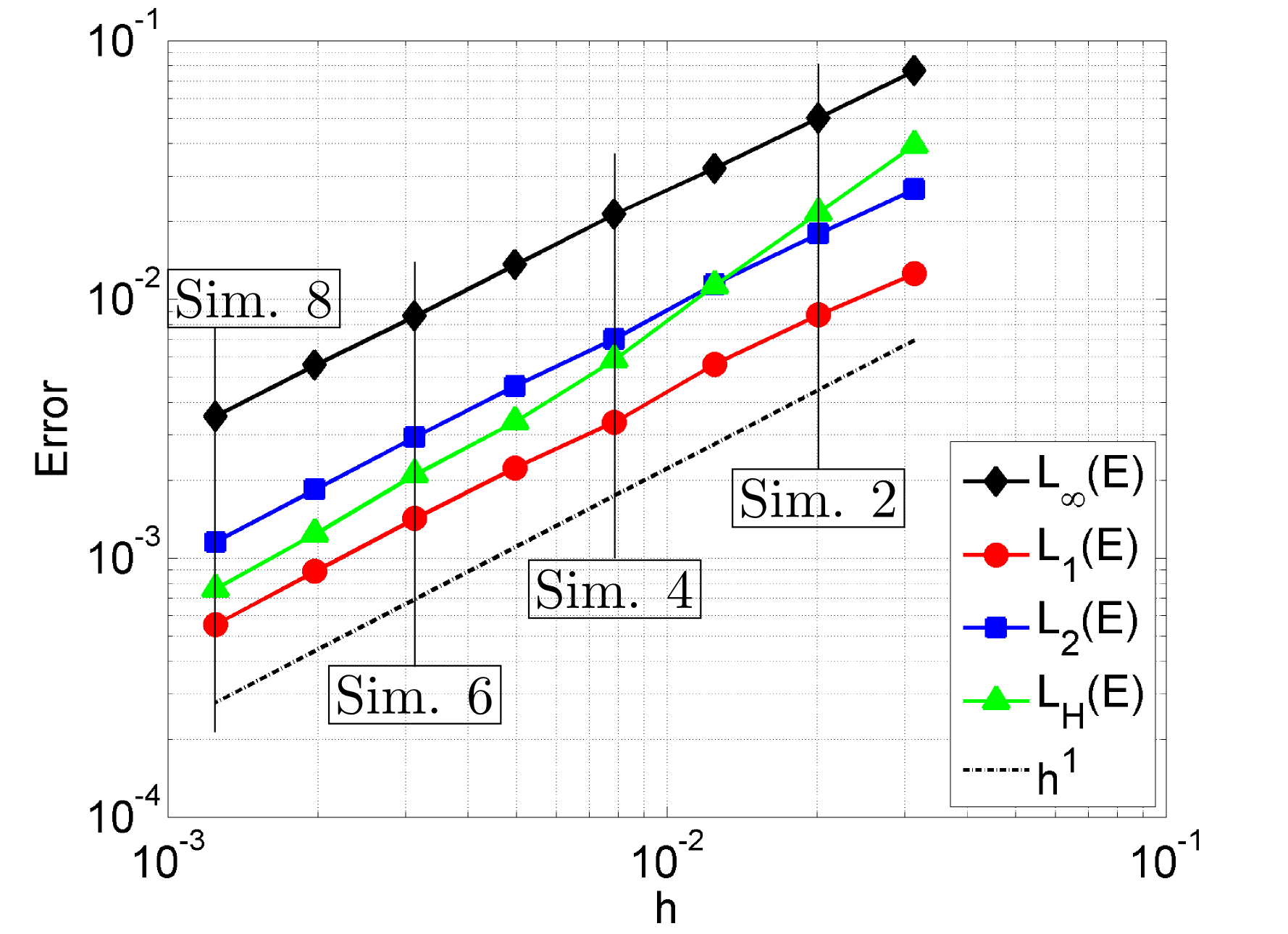}
  \caption{The football, \Ex \ExFootball.}
  \label{fig:Convb}
\end{subfigure}
\begin{subfigure}{.48\textwidth}
  \centering
  \includegraphics[width=\linewidth]{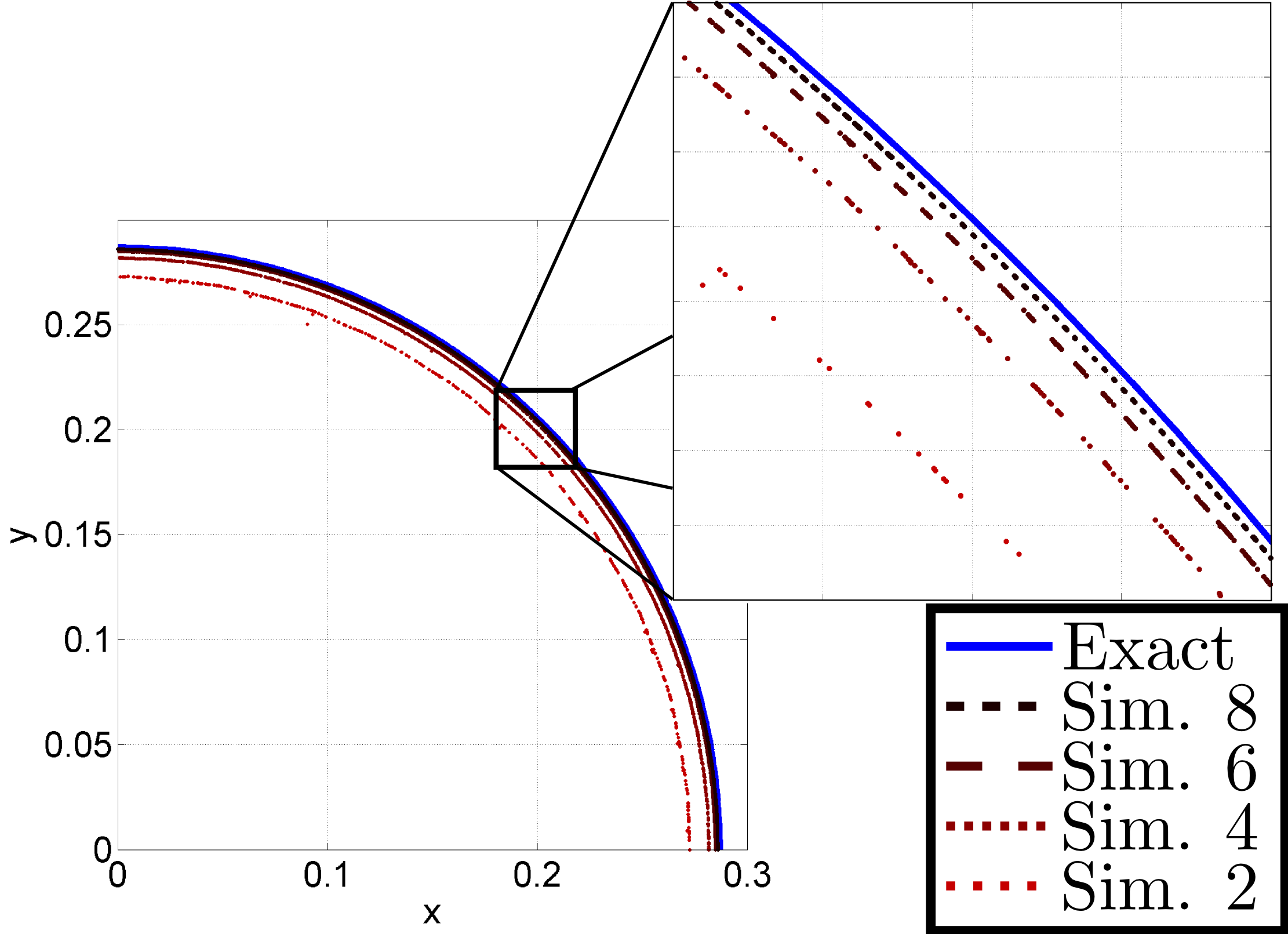}
  \caption{The football, \Ex \ExFootball.}
  \label{fig:Curvesb}
\end{subfigure} 

\begin{subfigure}{.48\textwidth}
  \centering
  \includegraphics[width=\linewidth]{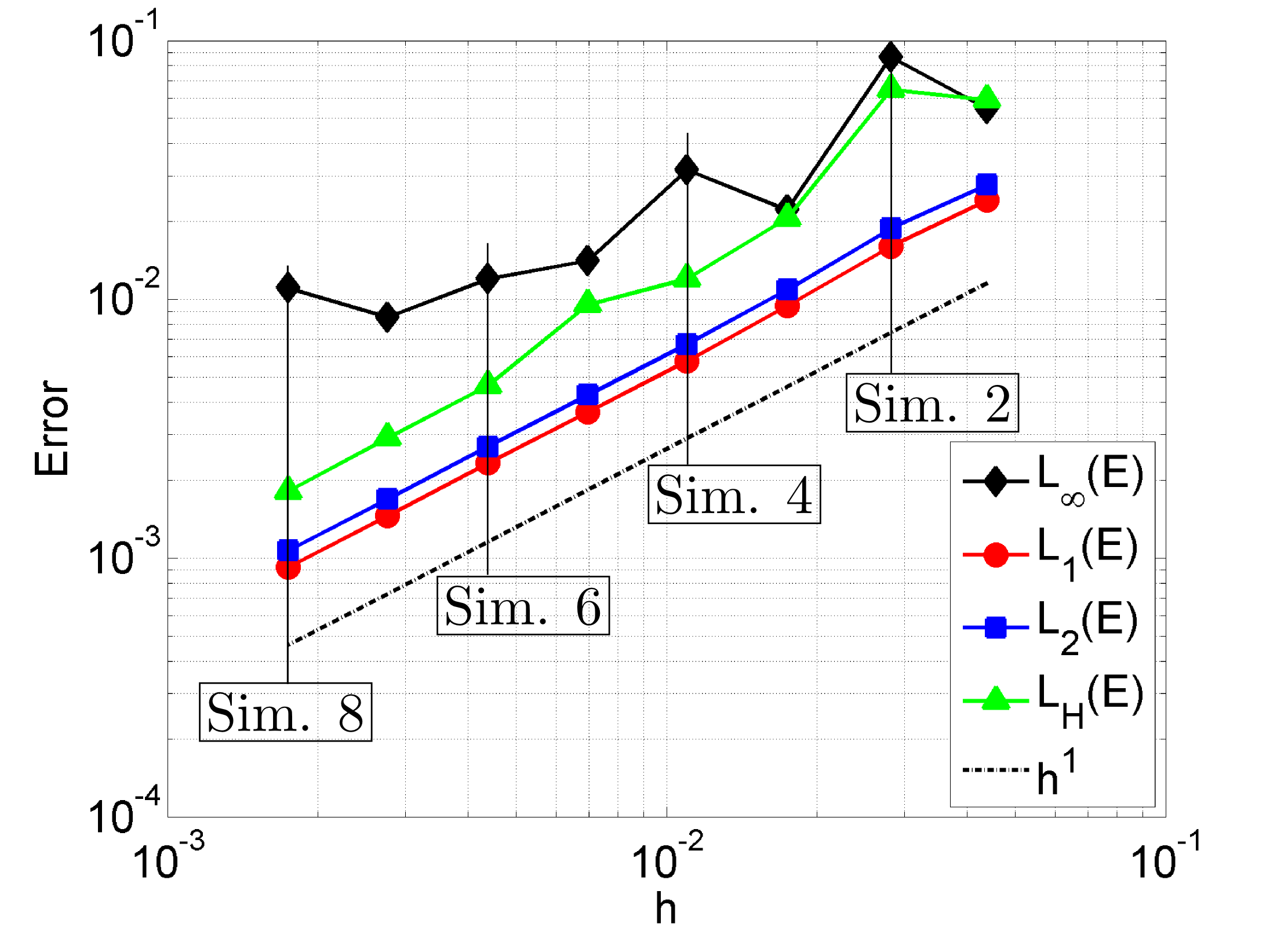}
  \caption{The two circles, Ex.\mbox{\,}\ExTwoCircles.\mbox{\,} (up to $T=0.5$)}
  \label{fig:Convf}
\end{subfigure}
\begin{subfigure}{.48\textwidth}
  \centering
  \includegraphics[width=\linewidth]{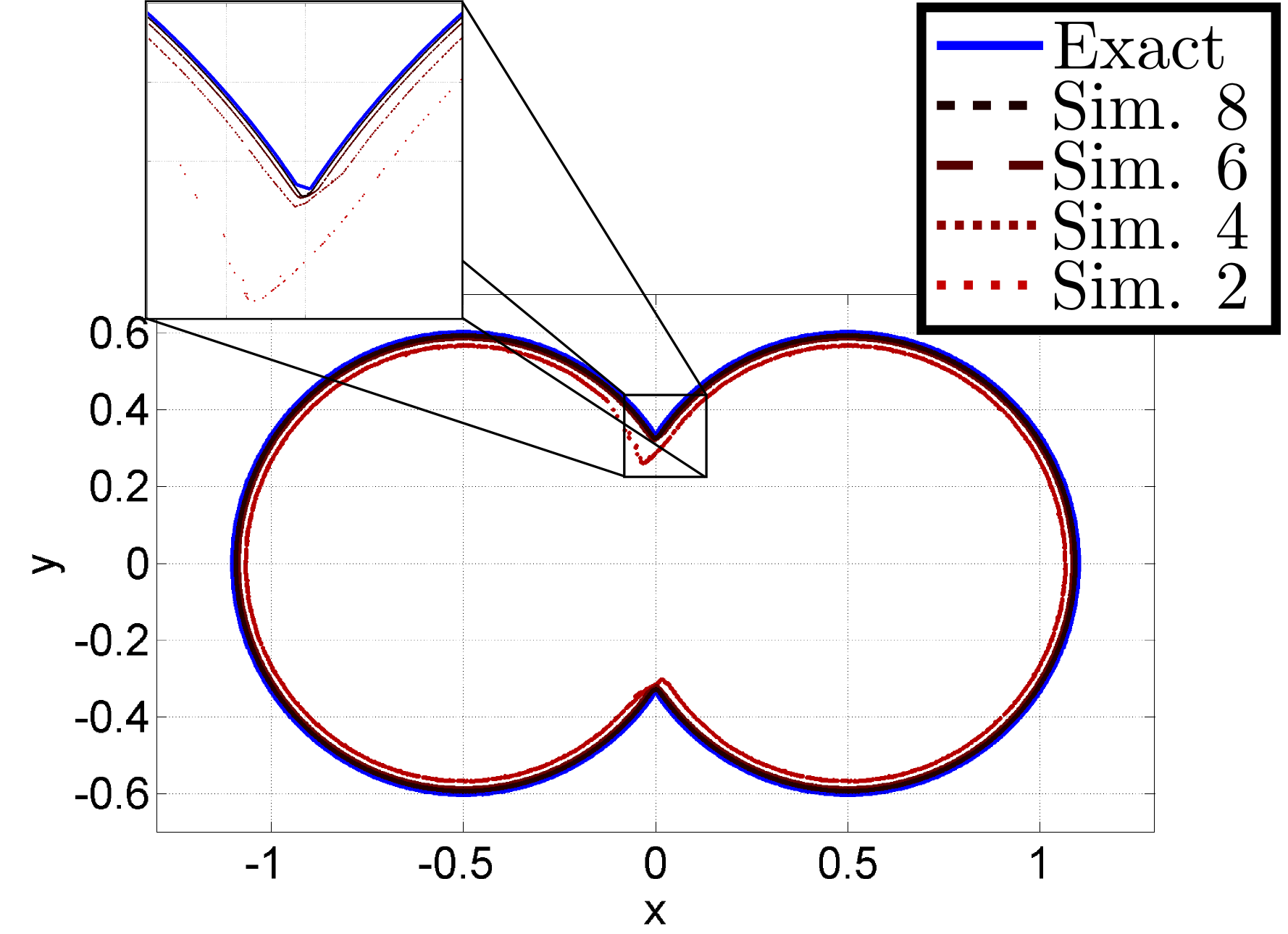}
  \caption{The two circles, \Ex \ExTwoCircles.}
  \label{fig:Curvesf}
\end{subfigure} 
\caption{\textbf{(left)} Global convergence results. \textbf{(right)} Exact and reconstructed curves.}
\label{fig:ConvergenceResults}
\end{figure}

\noindent Example \ExEscapingCircle \, is used to illustrate the robustness of the algorithm when $\MM$ is smooth. Qualitatively similar results are obtained for examples \ExOscillatingCircle, \ExEscapingCircle \, and \ExThreeLeavedRose, namely, first-order convergence with respect to $h$ is observed in all the norms considered. 
Similar results are obtained for the football example, despite the fact that $\MM$ is only $C^{0}$. The two circles example also converges with first-order accuracy in the $L_{1}$, $L_{2}$ and $L_{H}$ norms. Convergence in the $L_{\infty}$ norm is not as clear.

\subsubsection{Speed tests}

Our method is now compared to the standard first-order FMM, see for example \cite{Sethian}. The solution to \Ex (a) is computed on the set $|\xvec|<0.75$ for various gridsizes -- see Appendix \ref{app:TestsProcedures}. The CPU times are presented in Subfigure \ref{fig:SpeedTest}. Remark that the vertical axis is \emph{the $L_{1}$ norm of the error} associated with the sampling of $\MM$. It is apparent that for higher accuracies, our method is faster than the standard FMM for this example. 
Subfigure \ref{fig:SpeedTestNonMonotone} also presents CPU times obtained for non-monotone examples: The trend is found to be similar to the monotone case. 

\begin{figure}[h!]

\begin{subfigure}{.48\textwidth}
  \centering
  \includegraphics[width=\linewidth]{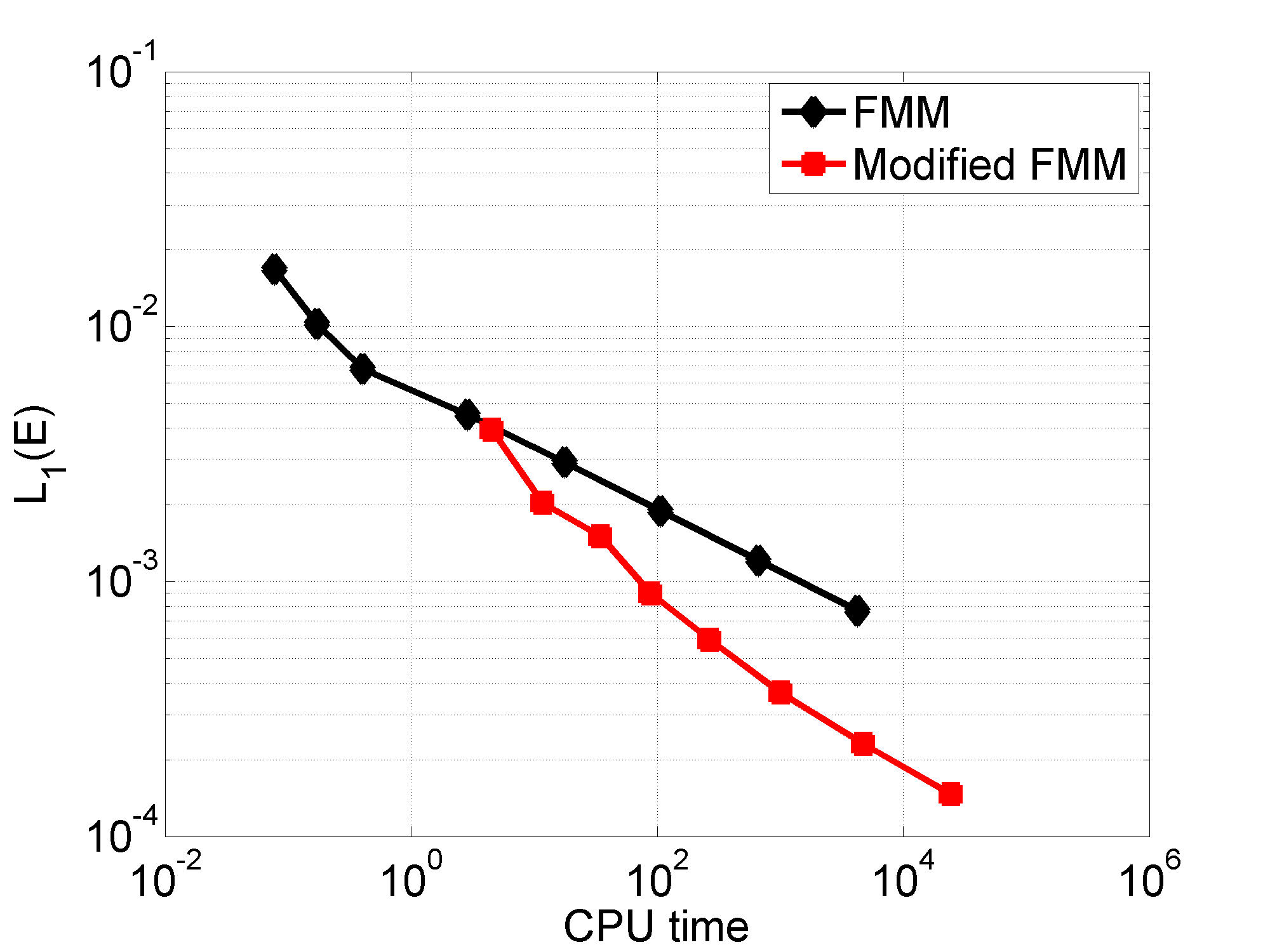}
  \caption{The expanding circle, \Ex \ExExpanding.
}
\label{fig:SpeedTest}
\end{subfigure} 
\begin{subfigure}{.48\textwidth}
  \centering
  \includegraphics[width=\linewidth]{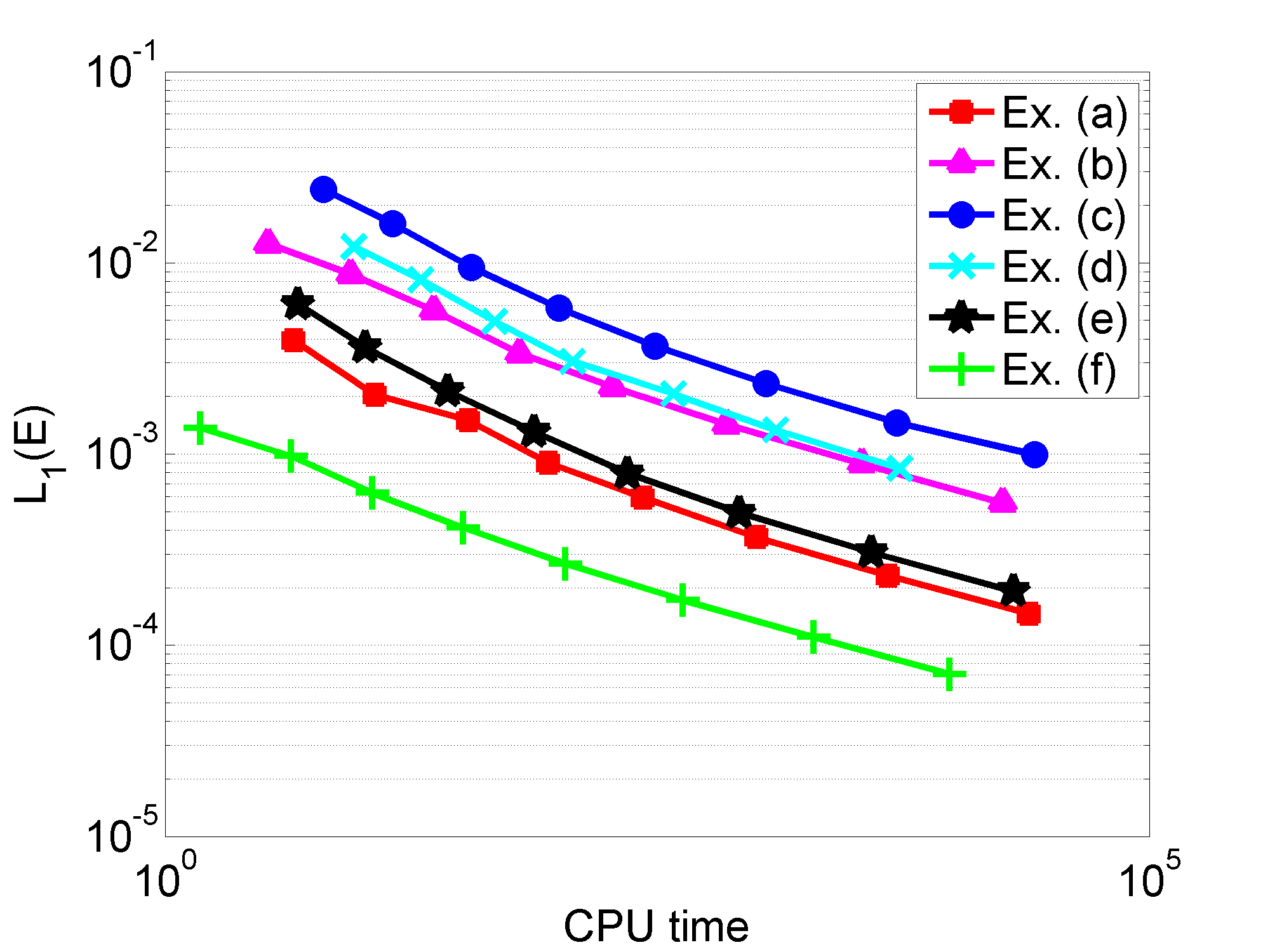}
  \caption{Results obtained using our method.}
\label{fig:SpeedTestNonMonotone}
\end{subfigure}
  \caption{Accuracy of the solution \vs CPU times.}
\end{figure}

\textsl{Remark} (See Appendices  \ref{app:IterativeSolver} and \ref{app:GridMethod}.) We point out a few changes that can be made to improve the computational speed of the code. As noted at the end of \S \ref{subsubsec:OptimizationProperties}, constraints (E) and (V2) are not necessarily equivalent. However, since in practice (V2) $\Longrightarrow$ (E) for an overwhelming number of points, the verification of constraint (E) may be taken out of the grid method, and checked \textsl{a posteriori}. Based on Figure \ref{fig:LocalConvergence}, running the iterative solver does not necessarily improve the accuracy of the solution. Making the stopping criterion depend on $h$, \eg $|w^{n+1}_{d}-w^{n}_{d}|\leq 10h \Delta \tau$ may avoid unnecessary computations.

~\\
\noindent \textbf{Conclusion.} 
We presented a scheme that describes the non-monotone propagation of fronts while featuring a numerical complexity comparable to that of the standard FMM. 
Local convergence was demonstrated and verified. Evidences of global convergence were supported by several examples where $F\in C^{1}$. 
The most general case, \ie $F \in C^{0}$, requires further theoretical investigation, and will therefore be the subject of subsequent papers.  
The theory presented in \S \ref{sec:Hyperplane} and \ref{sec:Discretization} trivially extends to higher dimensions. Nonetheless, some practical obstacles currently prevent the design of an algorithm when $n>2$. The most prominent one is to properly monitor the global features of the manifold, as in Algorithm \ref{algo:BookKeeping} -- see \cite{GlimmGroveEtAl1,GlimmGroveEtAl2,FronTierDocu} for the case $n=3$. Moreover, choosing an appropriate triplet of parents still requires some thought.  

As it stands, the algorithm is first order accurate. It may be extended to higher order using filtered schemes \cite{FroeseOberman,ObermanSalvador}. 
Allowing $h$ to depend on $(\xvec,t)$ would increase the accuracy in regions with high mean curvature, and avoid unecessary computations in other regions. The precise adaptivity criteria must be carefully addressed. 
A different approach is to resort to \emph{reseeding} by introducing points to the Narrow Band in regions of the front that are undersampled. 
We conclude by remarking that the novel ideas presented in this paper may apply to other evolution equations, such as linear advection, anisotropic propagation, or mean curvatuve flow.

\clearpage

\appendix

\section{$uvw$-coordinate system}
\label{app:A}\subsection{Tilted plane} Let a unit normal vector $\nav = (n_{1},n_{2},n_{3}) \in S^{3}$ define a plane through the origin of the form $t=ax+by$, where $a=-n_{1}/n_{3}$ and $b=-n_{2}/n_{3}$. 
Define $\mu = \sqrt{1+a^{2}+b^{2}}$, as well as $\hat{w} = \pm (a,b,-1)/\mu$ where the sign is chosen such that $\mathrm{sign}(\hat{w}_{3}) = \mathrm{sign} (n_{3})$. The following coordinate system can be verified to satisfy the conditions listed in \S \ref{sec:Hyperplane}:
\begin{eqnarray}
\label{eq:CoeffExplicit}
\left\{ \begin{array}{rcl}
u &=& \pm \frac{1}{\sqrt{a^{2}+b^{2}}} \left(~ b ~,~ -a ~,~ 0 ~\right) \\
v &=& \frac{1}{\sqrt{a^{2}+b^{2}} \, \mu} \left(~ a ~,~ b ~,~ a^{2}+b^{2} ~\right) \\
w &=& \pm \frac{1}{\mu} \left(~ -a ~,~ -b ~,~ 1 ~\right) 
\end{array} \right.
\end{eqnarray}

\subsection{Vertical plane}
\label{subsec:VerticalPlane}
If $n_{3}=0$, then $\bar{\nu} \in S^{3}$ describes a vertical plane. The orthonormal coordinate system we use is then: 
\begin{eqnarray}
\left\{ \begin{array}{rcl}
u &=& \left(~ -n_{2} ~,~ n_{1} ~,~ 0 ~\right) \\
v &=& \left(~ 0 ~,~ 0 ~,~ 1 ~\right) \\
w &=& \left(~ n_{1} ~,~ n_{2} ~,~ 0 ~\right) 
\end{array} \right.
\end{eqnarray}

\subsection{$\bar{\nu}$ is the outward normal to $\MM$}
\label{subseq:NuIsNormal}
Suppose that $\MM$ is $C^{1}$ at some point $p_{a}$ and that $\bar{\nu} = \nthree(p_{a})$. If $|F(p_{a})| >0$, then a neighbourhood of $p_{a}$ is locally described by $\psi(\xvec)$, where $\psi$ satisfies \cite{Sethian}:
\begin{eqnarray}
\label{eq:PDEquick}
|\nabla \psi(\xvec_{a})| = \frac{1}{|F(\xvec_{a})|}
\end{eqnarray}
Then $\phi(\xvec,t) = \mathrm{sign}(F(p_{a}))\, (\psi(\xvec)-t)$, and we have:
\begin{eqnarray}
\nthree 
\, = \, \mathrm{sign}(F(p_{a})) \frac{(\psi_{x},\psi_{y},-1)}{\sqrt{\psi^{2}_{x} + \psi^{2}_{y} + 1}}
\end{eqnarray}
Using the PDE (\ref{eq:PDEquick}) and the results from \S \ref{subsec:VerticalPlane} with $a=\psi_{x}$ and $b=\psi_{y}$, we get:
\begin{eqnarray}
\beta_{3} 
= \frac{1}{\sqrt{1+F^{2}(p_{a})}} 
\qquad \mathrm{and} \qquad
\gamma_{3} 
= \frac{-F(p_{a})}{\sqrt{1+F^{2}(p_{a})}}
\end{eqnarray}
If $F(p_{a}) =0$, then $\nthree(p_{a})$ describes a vertical plane, and $\beta_{3}=1$ and $\gamma_{3}=0$.

\section{The iterative solver} 
\label{app:IterativeSolver}
The size of the pseudo-time step used in the iterative solver is determined as follows. Defining the left-hand side of (\ref{eq:IterativeSolver}) as $-\tilde{H}(w^{n})$, we set  $\Delta \tau = \frac{9}{10}\, \frac{Q}{2\epsilon}$ where 
\begin{eqnarray*}
&~& | \tilde{H}(w^{1}) - \tilde{H}(w^{2}) | \\
&\leq& | \left( \tilde{\nu}^{2} - \tilde{\nu}^{1} \right) \cdot \hat{R} | + | G^{1} | ~ \left| \sqrt{\tilde{\nu}^{1} \cdot \tilde{\nu}^{1} - \left( \tilde{\nu}^{1} \cdot \hat{R} \right)^{2}} - \sqrt{\tilde{\nu}^{2} \cdot \tilde{\nu}^{2} - \left( \tilde{\nu}^{2} \cdot \hat{R} \right)^{2}} \right| \\
&~& \quad + | G^{1} - G^{2} | \sqrt{\tilde{\nu}^{2} \cdot \tilde{\nu}^{2} - \left( \tilde{\nu}^{2} \cdot \hat{R} \right)^{2}} ~ =: ~ Q
\end{eqnarray*} 
$\epsilon = h/10$ and $w_{1}$, $w_{2} \in (w^{0}_{d}-\epsilon,w^{0}_{d}+\epsilon)$. The neighbourhood is sampled with 10 points.

\section{The grid method: Pseudo-code} 
\label{app:GridMethod}
This is the method we use in practice. Note that Lines 6-10 can be completely vectorized. 

\begin{algorithm}
\caption{Constrained Optimization: The grid method} \label{algo:GridMethod}
\begin{algorithmic}[1]

	\State $\bar{h} \gets h/2$, $~u_{0} \gets u_{\min}$, $~v_{0} \gets h + v_{a}$, 
	\State $Q=1$, ~tol $=10^{-15}$, ~$i=1$, ~$f_{0}=10^6$. 

	\While{$Q > $tol} 
		\State Build a grid of $s\times s$ points, with meshsize $h$, centered at $(u_{0},v_{0})$
		\State On the grid, initialize $f=+\infty$.

		\For{each point $(u,v)$ on the grid}
			\State Compute $w$ using the direct solver (using $G(\uvec_{a})$).
			\State Compute $g_{1}$, $g_{2}$ and $g_{3}$. 
			\State If all three constraints are satisfied, compute $f(u,v)$. 
		\EndFor 
		\State Let the minimum of $f$ be $f_{i}$ and occur at $(u_{d},v_{d})$.
		\State $Q = |f_{i}-f_{i-1}|$
		\State $\bar{h} \gets \bar{h}/2$, $~u_{0} \gets u_{d}$, $~v_{0} \gets v_{d}$, $i \gets i+1$		
		\State \textbf{if} $i == 5$ or $f_{i} == +\infty$ \textbf{then} $Q = 0$ \textbf{end if}
	\EndWhile 

\State Return: $(u_{d},v_{d})$.

\end{algorithmic}
\end{algorithm}

\newpage

\vspace{-.3cm}

\section{Examples \& Tests procedures}
\label{app:TestsProcedures}

\subsection{Examples} With the exception of the third example, $\CC_{0}$ always consists of a single circle of radius $r_0$ centered at the origin.

\paragraph{The expanding circle} The speed is $F \equiv 1$ and the signed distance function that solves the LSE is $\phi(x,y,t) = \sqrt{x^{2}+y^{2}}-t-r_{0}$ with $r_{0}=0.25$.

\paragraph{The football} The speed is $F(t)=1-e^{10t-1}$ and the signed distance function that solves the LSE is $\tilde{\phi}(x,y,t) = \sqrt{x^2+y^2} - \left( r_0-\left(e^{ct}-1\right)/(ce)+t \right)$ with $r_{0}=0.25$.

\paragraph{Two circles} The speed is $F(t)=1-2t$ and the signed distance function that solves the LSE up to time $T=0.5$ is $\phi(x,y,t) = \sqrt{(x- \mathrm{sign}(x)0.5)^2+y^2} - \left( r_{0}+t-t^{2} \right)$ with $r_{0}=0.35$.

\paragraph{The oscillating circle} The speed is $F(t)=a\sin(b(t+c))$ with $a=0.7$, $b=10$ and $c=0.3$, and $\phi(x,y,t) = \sqrt{x^2+y^2}- \left(  r_0 + (a/b)( \cos(bc)-\cos(b(t+c)) ) \right)$ is the signed distance function that solves the LSE with $r_{0}=0.25$.

\paragraph{The escaping circle} The speed is $F = (x-gt)(g't+g)/\sqrt{(x-gt)^{2}+y^{2}} + c$ where $g(t) = \arctan \left(b(t-0.5)\right) + \frac{\pi}{2}$,
and the signed distance function that solves the LSE is $\phi(x,y,t) = \sqrt{(x-g t)^2+y^2}-\left( r_{0}+c t \right)$ with $r_{0} = 0.25$, $b=10$ and $c=0.5$.

\paragraph{The 3-leaved rose} The speed is $F = \cos(l\theta) / \sqrt{1 + (lt/r)^2 (\sin(l\theta))^2 }$ and a solution of the LSE is $\phi(x,y,t) = r-(t\cos(l\theta)+r_0)$, where $l=3$ is the number of petals and $r_{0} = 0.25$. Note that this is not the signed distance function. 

\subsection{Tests procedures}

\paragraph{Time for computing $\uvec_{d}$} The estimate of the time taken to compute $\uvec_{d}$ provided in the remark at the end of \S \ref{subsubsec:ComputationalTime} is based on example \ExThreeLeavedRose, run with $m=150$. 

\paragraph{Optimization problem} Each point is obtained as follows: The parent points are exact: $p_{a}$, $p_{b} \in \MM$. The point $p_{a}$ is fixed, and $p_{b}$ is such that $(x_{b},y_{b}) = (x_{a},y_{a})+(-3,4)h/8$. The exact normal at $p_{a}$ is used to get the local coordinate frame. The optimization problem is solved using the direct solver and returns $p^{\mathrm{dir}}_{d}$. Then the iterative solver is initialized with $w^{0}_{d} = w^{\mathrm{dir}}_{d}$ and run until $|w^{n+1}_{d}-w^{n}_{d}|<10^{-10}\Delta \tau$. 

\paragraph{Error propagation} 
Figures \ref{fig:ErrorVsTimec} and \ref{fig:ErrorVsTimed} were produced using $m=60$. 

\paragraph{Global properties} 
The results presented on Figure \ref{fig:EvenSampling} were obtained using $M=25$ for the expanding circle, $M=60$ for the football, and $M=150$ for the 3-leaved rose.  The histograms of Figure \ref{fig:Distance} correspond to the football simulation. The binwidth is 0.2 and the total number of points $N$ is 6229. On Figure \ref{fig:Topo}, there are initially 80 points on each circle. The data used to get the results presented on Figure \ref{fig:ConvergenceResults} are: Ex.\mbox{\,}\ExEscapingCircle \, Final T $=0.4$, $t_{H} = 0.35$, Ex.\mbox{\,}\ExFootball \, $t_{H} = 0.1$, Ex.\mbox{\,}\ExTwoCircles \, Final T $= 0.5$, $t_{H} = 0.45$.

\paragraph{Hausdorff distance} To get the Hausdorff distance between the exact and the reconstructed curves, samplings of each set were first obtained. For the exact one, we used the exact solution to the LSE and the MatLab \verb contour ~function. For the reconstructed one, local Delaunay triangulations were obtained using appropriate subsets of $\Aa$. Let the resulting two cloud of points be respectively $\CC \ell_{\mathrm{rec}}$ and $\CC \ell_{\mathrm{ex}}$. Then the Hausdorff distance between those sets is:
\begin{eqnarray}
L_{\mathrm H}(\CC \ell_{\mathrm{rec}},\CC \ell_{\mathrm{ex}}) 
= \max\left\{\,
\sup_{\xvec \in \CC \ell_{\mathrm{rec}}} \inf_{\yvec \in \CC \ell_{\mathrm{ex}}} d(\xvec,\yvec),\,
\sup_{\yvec \in \CC \ell_{\mathrm{ex}}} \inf_{\xvec \in \CC \ell_{\mathrm{rec}}} d(\xvec,\yvec)\,\right\}
\end{eqnarray}
where $d$ is the Euclidean distance function. To get the first term in braces, the exact signed distance function is used, since for a given $\xvec \in \CC\ell_{\mathrm{rec}}$, we have $\inf_{\yvec \in \CC \ell_{\mathrm{ex}}} d(\xvec,\yvec) = \phi(\xvec,t_{H})$. To get the second one, given $\yvec \in \CC\ell_{\mathrm{ex}}$, the nearest point $\xvec \in \CC\ell_{\mathrm{rec}}$ is found.

\paragraph{Speed tests} To get the results labelled as `Modified FMM' on Figure \ref{fig:SpeedTest}, our method is run. The points labelled as `FMM' are obtained using the first-order Fast Marching Method with different gridsizes. Points with $|\xvec|<0.25+2dx$ are initialized with exact values. The solution to the FMM is computed on the set $|\xvec|<0.75+dx$ (the neighbours of those Accepted points with $|\xvec|>0.75$ were not added to the Narrow Band). The procedure used to sort and update the Narrow Band is the same in both methods, namely: \bull The point with the smallest time value is found using the Matlab \verb min ~ command; \bull This point is removed from the Narrow Band by deleting the corresponding row (using $\NB(I,:)=[\,]$); \bull Each new point is added at the end of the Narrow Band. The CPU time was evaluated using the MatLab \verb cpu ~ command. The final times of the simulations are:  \ExExpanding \, 0.5, \ExTwoCircles \, 0.5, \ExOscillatingCircle \, 1.5 per., \ExEscapingCircle \, 0.4, \ExThreeLeavedRose \, 0.19.

\vspace{.2cm}
\noindent \textbf{Acknowledgements.}
The authors would like to thank Prof.\mbox{ }Bruce Shepherd and Dr.\mbox{ }Jan Feys for helpful discussions about the optimization problem. The first author acknowledges the support of the Schulich Graduate Fellowship. This research was partly supported through the NSERC Discovery and Discovery Accelerator Supplement grants of the second author. We also thank the anonymous referees for their very constructive comments. 

\vspace{-.5cm}

\bibliography{Article2Biblio}{}
\bibliographystyle{plain}

\end{document}